\documentclass[11pt,reqno]{amsart}
\usepackage[margin=0.7in]{geometry}
\usepackage{amsmath,amssymb,amsthm,graphicx,amsxtra, setspace}
\usepackage[utf8]{inputenc}
\usepackage{mathrsfs}
\usepackage{hyperref}
\usepackage{upgreek}
\usepackage{mathtools}
\usepackage[mathcal]{euscript}
\usepackage{xcolor}
\usepackage{multirow}
\allowdisplaybreaks

\usepackage[pagewise]{lineno}
\newtheorem{theorem}{Theorem}[section]
\newtheorem{lemma}[theorem]{Lemma}
\newtheorem{proposition}[theorem]{Proposition}

\newtheorem{corollary}[theorem]{Corollary}
\newtheorem{definition}[theorem]{Definition}
\newtheorem{example}[theorem]{Example}
\newtheorem{remark}[theorem]{Remark}

\newtheorem{hypothesis}[theorem]{Hypothesis}

\let\originalleft\left
\let\originalright\right
\renewcommand{\left}{\mathopen{}\mathclose\bgroup\originalleft}
\renewcommand{\right}{\aftergroup\egroup\originalright}

\newcommand{\Tr}{\mathop{\mathrm{Tr}}}
\renewcommand{\d}{\/\mathrm{d}\/}

\def\w{\textbf{W}^{\varepsilon}_{{\theta}^{\varepsilon}}}

\def\L{\mathbb{L}}
\def\A{\mathrm{A}}

\def\C{\mathrm{C}}
\def\f{\boldsymbol{\mathfrak{K}}}
\def\J{\mathrm{J}}
\def\K{\mathrm{K}}
\def\B{\mathrm{B}}
\def\D{\mathrm{D}}
\def\y{\boldsymbol{y}}

\def\E{\mathbb{E}}
\def\Arm{\boldsymbol{\mathcal{E}}}
\def\diver{\mathrm{div} }
\def\X{\mathbb{X}}
\def\x{\boldsymbol{x}}

\def\z{\Upsilon}
\def\u{\boldsymbol{\mathcal{Y}}}
\def\v{\boldsymbol{\mathcal{Z}}}
\def\w{\boldsymbol{\mathcal{X}}}
\def\yfrak{\mathfrak{y}}
\def\yscr{\mathscr{Y}}
\def\W{\mathrm{W}}

\def\M{\mathrm{M}}
\def\N{\mathbb{N}}

\def\V{\mathbb{V}}
\def\O{\mathbb{Q}}
\def\wi{\widetilde}

\def\H{\mathbb{H}}

\newcommand{\R}{\mathbb{R}}

\renewcommand{\d}{\/\mathrm{d}\/}

\newcommand{\Addresses}{{% additional braces for segregating \footnotesize
		\footnote{
			\noindent \textsuperscript{1}Center for Mathematics and Applications (NOVA Math), NOVA School of Science and Technology (NOVA FCT),	Portugal.\par\nopagebreak
			\noindent 
			\textit{e-mail:} \texttt{Kush Kinra: kushkinra@gmail.com, k.kinra@fct.unl.pt.}
			
			\noindent \textsuperscript{*}Corresponding author.
			
			\textit{Key words:} Stochastic third-grade fluids, pullback random attractors, invariant measures, ergodicity, general domains.
			
			Mathematics Subject Classification (2020): Primary 35B41, 35Q35; Secondary 37L55, 37N10, 35R60.

}}}

\begin{document}
	%	\linenumbers
	
	\title[Stochastic third-grade fluids on 2D and 3D general domains]{Large time behaviour for a class of 2D and 3D stochastic non-Newtonian fluids of differential types: Attractors and invariant measures
		\Addresses}
	
	\author[K. Kinra]
	{Kush Kinra\textsuperscript{1*}}

	\maketitle
	
	\begin{abstract}

		This study investigates a stochastic version of a class of non-Newtonian fluids governed by third-grade fluid equations, which exhibit complex and highly nonlinear dynamics. In particular, we address the random dynamics and asymptotic behavior of stochastic third-grade fluid equations (STGFEs) driven by a \emph{linear multiplicative It\^o-type white noise} on general domains $\mathbb{Q}\subseteq\mathbb{R}^d$, $d\in\{2,3\}$. We first prove that the non-autonomous STGFEs generate a continuous non-autonomous random dynamical system $\Phi$, and we establish the existence of a pullback absorbing set. Using compact Sobolev embeddings on bounded domains and uniform tail estimates on unbounded domains, we show the pullback asymptotic compactness of $\Phi$, which leads to the existence of pullback random attractors that are compact and attracting in $\mathbb{L}^2(\mathbb{Q})$.  As a consequence, we demonstrate the existence of an invariant measure associated with the STGFEs and, exploiting the linear multiplicative structure of the noise along with the exponential stability of solutions, we prove uniqueness of the invariant measure in the case of zero external forcing. These results are entirely new for STGFEs on general domains, and, in particular, the existence of pullback random attractors with linear multiplicative noise is obtained here for the first time. We further note that, unlike Stratonovich noise, which is widely used in the literature to study random attractors, It\^o noise is more appropriate for domains that do not satisfy the Poincar\'e inequality. Overall, this work resolves several open problems regarding random attractors, invariant measures, and ergodicity for stochastic third-grade fluids on general unbounded domains $\mathbb{Q}\subseteq\mathbb{R}^d$, $d\in\{2,3\}$.
	\end{abstract}

	\section{Introduction} \label{Sec1}\setcounter{equation}{0}
		From a physical standpoint, the concept of an attractor provides a fundamental explanation of how complex fluid flows evolve and stabilize over time. Although the equations governing fluid motion, such as the Navier-Stokes or non-Newtonian flow equations, are highly nonlinear and sensitive to initial conditions, real fluid systems often display a tendency to approach a statistically steady or recurrent state after transient effects decay. This long-term behavior reflects the balance between energy input from external forcing and energy dissipation due to viscosity. The attractor represents the mathematical manifestation of this physical stabilization process, describing the set of all possible states toward which the flow eventually evolves. Studying attractors therefore helps to identify and quantify the persistent structures, coherent patterns, or steady regimes that govern fluid motion at large times. Moreover, the finite-dimensional nature of attractors in many dissipative systems implies that, despite the infinite-dimensional character of the governing equations, the essential dynamics of the flow can be captured by a relatively small number of dominant modes, an observation that aligns closely with experimental and computational findings in fluid mechanics and turbulence theory (see \cite{Robinson2,R.Temam} etc.). Understanding attractors thus provides deep physical insight into how complex flows self-organize, lose memory of initial conditions, and exhibit predictable long-term behavior despite underlying nonlinearities.

	\medskip
	Extensive research on stochastic perturbations of evolution equations has led to the development of the theory of random dynamical systems and random attractors (see \cite{Arnold}). The concept of a pathwise pullback random attractor, which can also be forward attracting in probability, was first introduced in \cite{CF,FS} and further studied in \cite{HCJA}. Since then, the existence and properties of random attractors for various stochastic partial differential equations (SPDEs) have been established in numerous works, see \cite{Han+Zhou_2025,KK+FC2,Kinra+Mohan_2025_JDDE,Wu+Nguyen+Bai} etc. for some recent works. Moreover, \cite{CCLR} highlighted the distinct impacts that different types of stochastic perturbations can have on the long-term behavior of deterministic systems.

\subsection{The Model Being Studied}	
	It is worth emphasizing that, although much of the literature focuses on Newtonian fluids governed by the Navier-Stokes equations, many real-world industrial and biological flows (see \cite{DR95,FR80,yas-fer_JNS} and references therein) do not follow Newton’s law of viscosity and therefore cannot be accurately described by these models. Such fluids often display complex rheological features, such as shear-thinning, shear-thickening, or viscoelastic behavior, that classical Newtonian's law fail to capture. As a result, more sophisticated mathematical models are required to describe their dynamics and predict their behavior in realistic settings. In recent years, non-Newtonian viscoelastic fluids of differential type have attracted significant attention (see, for example, \cite{Cioran2016}). In particular, third-grade fluid models have been employed in a number of simulation studies to better understand the behavior of nanofluids (see for instance \cite{PP19,RHK18}). Nanofluids are engineered suspensions of nanoparticles in a base fluid, such as water, oil, or ethylene glycol, and are known to exhibit enhanced thermal conductivity compared to the base fluid alone, making them highly relevant in technological and microelectronic applications. Therefore, a rigorous mathematical analysis of third-grade fluid equations is essential for understanding the behavior of these fluids.

	\medskip
	We now briefly outline the derivation of the governing equations for non-Newtonian fluids of differential type. For a detailed discussion of the kinematics of such fluids, we refer the reader to \cite{Cioran2016}. Let $\u$ denote the velocity field of the fluid, and introduce the Rivlin-Ericksen kinematic tensors $\Arm_n$ for $n \geq 1$ (see \cite{RE55}), defined by
	\begin{align*}
		\Arm_1(\u)&=\nabla \u+(\nabla \u)^T;	\,\Arm_n(\u)=\dfrac{\D}{\D t} \Arm_{n-1}(\u)+\Arm_{n-1}(\u)(\nabla \u)+(\nabla \u)^T\Arm_{n-1}(\u), \ n=2,3,\cdots,
	\end{align*}
	where $
	\dfrac{\mathrm{D}}{\mathrm{D}t}\equiv
	\dfrac{\partial}{\partial t}+(\u\cdot \nabla)$ is material derivative.
	
	The constitutive law of fluids of grade $n$ reads $\mathbb{T}=-p\mathrm{I} + F(\Arm_1,\cdots,\Arm_n),$ where $\mathbb{T}$ is the Cauchy stress tensor, $p$ is the pressure and $F$ is an isotropic polynomial function of degree $n$, subject to the usual	requirement of material frame indifference,	see	\cite{Cioran2016}. The constitutive law of   third-grade fluids	$(n=3)$ is given by the following equation	 
	
	For fluids of grade $n$, the constitutive relation is $\mathbb{T}=-p\mathrm{I} + F(\Arm_1,\cdots,\Arm_n),$ where $\mathbb{T}$ is the Cauchy stress tensor, $p$ is the pressure and $F$ is an isotropic polynomial function of degree $n$, satisfying the principle of material frame indifference (see \cite{Cioran2016}).
	For third-grade fluids $(n=3)$, this constitutive law can be written as follows:
	\begin{align*}
		\mathbb{T}=-p\mathrm{I}+\nu \Arm_1+\alpha_1\Arm_2+\alpha_2\Arm_1^2+\beta_1 \Arm_3+\beta_2(\Arm_1\Arm_2+\Arm_2\Arm_1)+\beta_3\Tr(\Arm_1^2)\Arm_1,
	\end{align*}
	where $\nu$ is the viscosity and $\alpha_1, \alpha_2, \beta_1, \beta_2, \beta_3$ are material moduli. Recall that, by Newton’s second law, the momentum equations are given by
	$$\dfrac{\mathrm{D}\u}{\mathrm{D}t}=
	\dfrac{\partial\u}{\partial t}+(\u\cdot \nabla) \u=\text{div}(\mathbb{T}).$$ 
	If $\beta_1=\beta_2=\beta_3=0$, the constitutive equations reduce to those of second-grade fluids. It has been demonstrated that the Clausius–Duhem inequality, together with the requirement that the Helmholtz free energy attain a minimum in equilibrium, imposes the following conditions on the viscosity and material moduli:
	\begin{align}\label{secondlaw}
		\nu \geq 0,\quad \alpha_1+\alpha_2=0, \quad \alpha_1\geq 0. 
	\end{align}
	Although second-grade fluids are mathematically easier to handle, rheologists working with various non-Newtonian fluids have not verified the restrictions in \eqref{secondlaw}. They have therefore concluded that the fluids tested are not truly second-grade fluids but instead possess a different constitutive structure. For further details, see \cite{FR80} and the references therein. Following \cite{FR80}, to ensure that the motion of the fluid is consistent with thermodynamics, one must impose that
	\begin{equation}\label{third-grade-paremeters}
		\nu \geq 0, \quad \alpha_1\geq 0, \quad |\alpha_1+\alpha_2 |\leq \sqrt{24\nu\beta}, \quad \beta_1=\beta_2=0, \beta_3=\beta \geq 0.
	\end{equation}	
	Hence, the incompressible third-grade fluid equations read as
	\begin{equation}
		\label{third-grade-fluids-equations}
		\left\{\begin{aligned}
			\partial_t(z(\u))-\nu \Delta \u+(\u\cdot \nabla)z(\u) +\displaystyle\sum_{j=1}^d[z(\u)]_j\nabla \u_j & -(\alpha_1+\alpha_2)\text{div}((\Arm(\u))^2)\\ -\beta \text{div}[\Tr(\Arm(\u)\Arm(\u)^T)\Arm(\u)] + \nabla \mathbf{P}
			& =  \f,  \\
			\diver (\u) & =  0, \\
			z(\u)&:=\u-\alpha_1\Delta \u, \\[0.2cm]
			\Arm(\u) & := \nabla \u+(\nabla \u)^T,
		\end{aligned}\right.
	\end{equation}
	where the viscosity $\nu$ and the material moduli	 $\alpha_1,\alpha_2$, $\beta$ 	verify	\eqref{third-grade-paremeters}.  Note that setting $\alpha_1=\alpha_2=0$	and $\beta=0$ recovers the classical Navier-Stokes equations (NSEs). Mathematically, $n$-grade fluids form a hierarchy of increasing complexity; compared to Newtonian (grade 1) and second-grade fluids, third-grade fluids involve more nonlinear terms and require a more intricate analysis.
	
	\medskip
	
	Let $\O$ be an open and connected subset (may be bounded or unbounded) of $\R^d$, $d\in\{2,3\}$, the boundary of which is sufficiently smooth. In this article,  we consider the system \eqref{third-grade-fluids-equations} under a subset of the physical conditions \eqref{third-grade-paremeters}, specifically
	\begin{equation}\label{third-grade-paremeters-res}
		\nu > 0, \quad \alpha_1= 0, \ \alpha_2=\alpha, \quad |\alpha | < \sqrt{2\nu\beta}, \quad \beta_1=\beta_2=0, \beta_3=\beta > 0,
	\end{equation}	
	in the presence of a non-autonomous deterministic forcing and linear multiplicative scalar white noise, and subject to Dirichlet boundary conditions, the equations take the form for a given $\mathfrak{r}\in\R$
	\begin{equation}\label{equationV1}
		\left\{		\begin{aligned}
			\d	\u - [ \nu \Delta \u - (\u\cdot \nabla)\u+\alpha\text{div}((\Arm(\u))^2) + \beta \text{div}(|\Arm(\u)|^2\Arm(\u))- \nabla \textbf{P}]\d t&=\f \d t +\sigma\u\d \W , &&   \text{in }  \O \times (\mathfrak{r},\infty),\\
			\text{div}\; \u&=0, \quad &&   \text{in } \O \times [\mathfrak{r},\infty),\\
			\u &= \boldsymbol{0}, &&  \hspace{-3mm} \text{on } \partial\O\times [\mathfrak{r},\infty),\\
			\u(x,\mathfrak{r})&=\u_{\mathfrak{r}},  \quad &&  \text{in } \O,
		\end{aligned}\right.
	\end{equation}
	where $\u:\O\times [\mathfrak{r},\infty) \to \R^d$,  $\mathbf{P}:\O\times   [\mathfrak{r},\infty) \to \mathbb{R}$ and	 $\f:\O\times[\mathfrak{r},\infty)\to \R^d $ represent the velocity field, pressure	an	external force, respectively. Here  $\sigma>0$ is known noise intensity, the stochastic integral is understood in the It\^o sense and $\W(t,\omega)$ is the standard scalar Wiener process on some filtered probability space $(\Omega, \mathscr{F}, (\mathscr{F}_{t})_{t\in\R}, \mathbb{P})$. When $\O = \R^d$, the boundary condition $\u = \mathbf{0}$ on $\partial\O \times [\mathfrak{r},\infty)$ is replaced by a decay condition at infinity, that is,
	\begin{align*}
		|\u(x,t)| \to 0 \quad \text{as} \quad |x| \to \infty, \quad \text{for all } \mathfrak{r} \leq t < \infty.
	\end{align*}

\subsection{Review of Existing Studies}
	
	We begin by recalling several foundational results concerning system \eqref{third-grade-fluids-equations} in the case $\alpha_1>0$. In \cite{AC97}, the authors proved the existence of a local-in-time solution in $\H^3(\O)$ for \eqref{third-grade-fluids-equations} ($\alpha_1>0$) on bounded domains equipped with Dirichlet boundary conditions. Subsequently, \cite{Bus-Ift-1} established global-in-time existence of $\H^2$-valued solutions in $\R^d$, $d\in\{2,3\}$, and uniqueness in the two-dimensional setting; the question of uniqueness of $\H^2$-solutions in three-dimensional domain, however, remains unresolved. In \cite{Bus-Ift-2}, the same system was studied under Navier-slip boundary conditions, where the authors obtained global existence for initial data in $\H^2(\R^d)$ for $d\in\{2,3\}$ and again proved uniqueness in dimension two.
	A significant difficulty in analyzing \eqref{third-grade-fluids-equations} ($\alpha_1>0$) arises from the highly nonlinear structure of the model, which prevents one from working with initial data of lower regularity unless additional constraints on the parameters $\alpha_1,\alpha_2,\beta$, and $\nu$ are imposed. In particular, \cite{Paicu2008} proved the existence of a global weak solution in $\R^d$, $d\in\{2,3\}$, by introducing suitable parameter restrictions that allow the use of monotonicity methods for initial data in $\H^1(\R^d)$ when $\alpha_1>0$. That work also established a weak–strong uniqueness principle together with the validity of the corresponding energy equality. The stochastic counterpart of system \eqref{third-grade-fluids-equations} ($\alpha_1>0$) has been studied in the works \cite{Cip+Did+Gue_2021,YT+FC-SPDE} etc. 
	
	\medskip 
	We now turn to the literature concerning system \eqref{third-grade-fluids-equations} with $\alpha_1=0$, that is, \eqref{equationV1} with $\sigma=0$. A natural starting point is the seminal work of Ladyzhenskaya \cite{Ladyzhenskaya67}, who introduced a new model for incompressible viscous fluids in which the viscosity depends on the velocity gradient. This system features nonlinear terms similar to those in \eqref{third-grade-fluids-equations} (with $\alpha_1=0$). In \cite{Hamza+Paicu_2007}, the authors established global well-posedness for system \eqref{third-grade-fluids-equations} (with $\alpha_1=0$) in $\R^3$, assuming divergence-free initial data in $\L^2(\R^3)$. Their approach relied on a monotonicity method along with condition \ref{third-grade-paremeters-res}. The stochastic counterpart of system \eqref{third-grade-fluids-equations} (with $\alpha_1=0$) perturbed by multiplicative Lipschitz noise on bounded domains was treated in \cite{yas-fer_JNS}, where the authors proved well-posedness and also constructed an ergodic invariant measure. More recently, in \cite{KK+FC1}, the author together with Cipriano proved the existence of a unique pullback attractor for system \eqref{third-grade-fluids-equations} (with $\alpha_1=0$) on Poincar\'e domains (both bounded and unbounded). They subsequently extended this work in \cite{KK+FC2} by proving the existence of random attractors for a stochastic version of system \eqref{third-grade-fluids-equations} (with $\alpha_1=0$) with infinite-dimensional additive white noise, again on Poincaré domains.

	\subsection{Challenges and Methodological Approaches}
	
	If one examines the system \eqref{third-grade-fluids-equations} under the parameter conditions \eqref{third-grade-paremeters-res} and includes a linear damping term, namely
		\begin{align*}
		\frac{\partial \u}{\partial t} + \varpi \u - \nu \Delta \u + (\u\cdot \nabla)\u - \alpha\, \mathrm{div}\big((\Arm(\u))^{2}\big) - \beta\, \mathrm{div}\big(|\Arm(\u)|^{2}\Arm(\u)\big) + \nabla \mathbf{P} = \f , \qquad \nabla\!\cdot\u = 0,
	\end{align*}
	where $\varpi>0$, then the existence of global and pullback attractors can be established on arbitrary domains; see, for instance, \cite[p.~306]{R.Temam} for a discussion regarding linear damping. Although system \eqref{equationV1} does not contain such a linear damping mechanism, the application of a Doss-Sussman transformation \cite{Doss_1977,Sussmann_1978} (see \eqref{COV} below) allows us to transform it into an equivalent deterministic system with random coefficients (see \eqref{CTGF} below).  This reformulation plays a crucial role in proving our results on any general domain $\mathcal{O}\subseteq\mathbb{R}^{d}$, $d\in\{2,3\}$. This is an advantage of working with the It\^o formulation in \eqref{equationV1}, rather than the Stratonovich interpretation.

	The Doss-Sussman approach is a classical technique that converts certain stochastic (ordinary or partial) differential equations into deterministic (ordinary or partial) differential equations with random coefficients, thereby enabling a pathwise analysis of solutions. A discussion of the Doss-Sussman transformation is also provided in \cite[Chapter IX]{Revuz+Yor_1999}. This method has been widely used to study SPDEs with linear multiplicative noise; see, e.g., \cite{BL,CGSV,CF,KK+MTM_SEE,KK+MTM_SNSE,MRRZXZ}. We also note that \cite{HV} employed this transformation to obtain global smooth pathwise solutions of the two-dimensional Euler equations and local solutions in three-dimensional domains. Moreover, the particular form of the transformation used in the present paper has been applied previously in \cite{KK+MTM_SEE,KK+MTM_SNSE} to investigate the random dynamics of the two-dimensional stochastic Euler (on bounded domains) and Navier-Stokes (on the whole space $\R^2$) systems perturbed by linear multiplicative It\^o noise, all without adding any artificial damping.

	  In this work, we employ the abstract framework developed in \cite{SandN_Wang} to establish the existence of pullback random attractors for system \eqref{equationV1}. To apply this theory, we first verify that system \eqref{equationV1} generates a non-autonomous random dynamical system $\Phi$ (see Subsection~\ref{NRDS}). We then demonstrate that $\Phi$ admits a random absorbing set and is asymptotically compact. The absorbing set is obtained by deriving suitable uniform energy estimates.
	
	On bounded domains, asymptotic compactness follows from the compact Sobolev embedding $\H^1(\O)\hookrightarrow \L^2(\O)$. In contrast, on unbounded domains this embedding is no longer compact, and alternative techniques must be employed. Two standard approaches in the literature are:
	\begin{itemize}
		\item[1.] Ball’s energy equation method \cite{Ball},
		\item[2.] Wang’s uniform tail estimate method \cite{UTE-Wang}.
	\end{itemize}
	
	In contrast to other fluid systems for which pullback random attractors have been successfully investigated, such as the Navier-Stokes equations \cite{PeriodicWang} and the convective Brinkman-Forchheimer equations \cite{Kinra+Mohan_2023_DIE}, the solutions of system \eqref{equationV1} lack the required weak continuity with respect to initial conditions. This absence of weak continuity prevents the use of Ball’s energy equation approach to obtain asymptotic compactness on unbounded domains. Therefore, in the unbounded domain setting, we instead employ Wang’s uniform tail estimate technique to establish the asymptotic compactness of $\Phi$ (see Lemma~\ref{Asymptotic_UB_GS}).
	
	\medskip
	
	The existence of random attractors helps us to obtain the existence of invariant measures (see \cite[Corollary 4.6]{CF}.  The question of uniqueness of invariant measures for arbitrary forcing terms $\f$ presents substantial difficulties and will not be addressed here. We consider the deterministic forcing term $\f=\mathbf{0}$ and obtain the uniqueness of invariant measure (Dirac measure centered at zero) for any $\nu>0,$ where the linear multiplicative structure of the white noise coefficient and exponential stability of solutions (see Lemma \ref{ExpoStability}) play a crucial role. The uniqueness of invariant measures  for $\f\neq \mathbf{0}$ is still an open problem for system \eqref{equationV1}.

\subsection{Goals of the Study}We now present our main results, which are direct consequences of 
Theorems~\ref{PRA_B}, \ref{PRA_U}, \ref{thm6.3}, and \ref{UEIM}.

\begin{theorem}
	Assume that condition \eqref{third-grade-paremeters-res} holds. Then:
	\begin{enumerate}
		\item[(i)] on bounded domains, system \eqref{equationV1} possesses pullback random attractors whenever the forcing satisfies  $\f\in\mathrm{L}^2_{\mathrm{loc}}(\R;\H^{-1}(\O))$ and there exists a number $\delta\in[0,\frac{\sigma^2}{2})$ such that for every $c>0$,
		\begin{align*}
			\lim_{s\to-\infty}e^{cs}\int_{-\infty}^{0} e^{\delta\zeta}\|\f(\cdot,\zeta+s)\|^2_{\H^{-1}}\d \zeta=0.
		\end{align*}
		
		\item[(ii)] on unbounded domains, system \eqref{equationV1} possesses a unique random attractor provided the forcing satisfies $\f\in\mathrm{L}^2_{\mathrm{loc}}(\R;\L^{2}(\O))$ and there exists a number $\delta\in[0,\frac{\sigma^2}{2})$ such that for every $c>0$,
		\begin{align*}
			\lim_{s\to-\infty}e^{cs}\int_{-\infty}^{0} e^{\delta\zeta}\|\f(\cdot,\zeta+s)\|^2_{2}\d \zeta=0.
		\end{align*}
		
		\item[(iii)] on both bounded and unbounded domains, system \eqref{equationV1} admits an invariant measure, corresponding to $\f\in\H^{-1}(\O)$ and $\f\in\L^{2}(\O)$, respectively.
		\item[(iv)] on both bounded and unbounded domains, system \eqref{equationV1} admits at most one invariant measure for $\f=\boldsymbol{0}$.
	\end{enumerate}
\end{theorem}

	\begin{remark}\label{rem-forcing-hypo}
		For the attractor analysis on bounded domains, we work with external forcing $\f$ satisfying assumption in terms of $\H^{-1}$-norm. However, in the case of unbounded domains, we must impose a stronger regularity assumption, namely in terms of $\L^{2}$-norm. If one could establish weak continuity of solutions to system \eqref{equationV1} with respect to the initial data, then, by employing the energy equality method, it would be possible to treat the unbounded domain case under the same regularity assumptions as in the bounded domain setting (see, for example, \cite{PeriodicWang,Kinra+Mohan_2023_DIE} for the Navier-Stokes equations and related models). Unfortunately, we are unable to prove such weak continuity for system \eqref{equationV1}, and therefore the energy equality approach is not applicable.
	\end{remark}

	\begin{remark}\label{rem-weak-conti}
		Compared with the Navier-Stokes equations and convective Brinkman-Forchheimer equations, system \eqref{equationV1} contains stronger nonlinear terms; nevertheless, owing to its structural symmetry, the nonlinear expression 
		\begin{align*}
			\operatorname{div}\big(|\Arm(\v)|^{2}\Arm(\v)\big)
		\end{align*}
		provides a significant regularizing effect, which in particular guarantees uniqueness even in three-dimensional domains. At first glance, one might therefore expect system \eqref{equationV1} to inherit the weak continuity property as well. However, establishing this property turns out to be highly nontrivial.
		
		To explain the difficulty, let $\v_n(\cdot)$ denote the unique solution of \eqref{equationV1} with initial data $\v_{0,n}$ for $n\in\mathbb{N}$. By definition, weak continuity would require that if $\v_{0,n}$ converges weakly to $\v_0$ in $\H$, then the corresponding solutions $\v_n(\cdot)$ converge weakly in $\H$ to $\v(\cdot)$, the solution of \eqref{equationV1} with initial data $\v_0$. However, because of the gradient-type nonlinearities in the system, passing to the limit as $n\to\infty$ in \eqref{equationV1} necessitates an application of the monotonicity method, which in turn requires strong convergence of the initial data. This method does not seem compatible with proving weak continuity of solutions for \eqref{equationV1}. 
		
		In contrast, for the Navier-Stokes equations and convective Brinkman-Forchheimer equations, the corresponding limit process can be carried out without relying on monotonicity arguments, making weak continuity accessible in that setting.
	\end{remark}

	The existence of pullback, global, or random attractors on general unbounded domains remains an open question. To the best of the author’s knowledge, the long-time dynamics of solutions to system \eqref{equationV1} have not yet been explored in either bounded or unbounded domains. Likewise, the study of invariant measures for this subclass on general unbounded domains is carried out here for the first time. It is worth noting that in two- and three-dimensional bounded domains, this class of non-Newtonian fluids was analyzed in \cite{yas-fer_JNS}, where the authors established well-posedness and proved the existence of an ergodic invariant measure for multiplicative Lipschitz noise. However, that work did not address the random dynamical behavior of the solutions. In bounded domains, our existence result for random attractors directly implies a special case of the findings in \cite{yas-fer_JNS}.

	\subsection{Structure of the Paper} 
	In the next section, we introduce the necessary function spaces required for the subsequent analysis, along with the linear and nonlinear operators needed to formulate an abstract version of system \eqref{equationV1}. Additionally, we present an abstract formulation of \eqref{equationV1}, a pathwise deterministic system \eqref{CTGF} equivalent to \eqref{equationV1}, the associated random dynamical system $\Phi$ (see \eqref{Phi}), and the universe of tempered sets $\mathfrak{D}$.	In Section~\ref{Sec3}, we establish that the random dynamical system $\Phi$ admits a $\mathfrak{D}$-pullback random absorbing set (Lemma~\ref{PAS}). Section~\ref{Sec4} is devoted to proving the existence of $\mathfrak{D}$-pullback random attractors for $\Phi$ on bounded domains (Theorem~\ref{PRA_B}). The corresponding result for general unbounded domains is obtained in Section~\ref{Sec5} (Theorem~\ref{PRA_U}). 	In the final section, we demonstrate the existence of an invariant measure for system \eqref{equationV1} (Theorem~\ref{thm6.3}). Furthermore, we establish the exponential stability of solutions of \eqref{equationV1} when $\f=\mathbf{0}$ (Lemma~\ref{ExpoStability}), and based on this result, we prove the uniqueness of the invariant measure in the case $\f=\mathbf{0}$ (Theorem~\ref{UEIM}).
	
	\section{Functional framework and auxiliary results}\label{Sec2}\setcounter{equation}{0}

	In this section, we introduce the functional framework and operators needed for the abstract formulation of system \eqref{equationV1}. We also collect auxiliary results on the Ornstein–Uhlenbeck process and several useful inequalities that will support later analysis and help to show that system \eqref{equationV1} generates a non-autonomous random dynamical system. Finally, we define a universe of tempered  random sets.

	\subsection{Notations and the functional setting }
	Let	$m\in	\mathbb{N}^*:=\N\cup\{\infty\}$	and	$1\leq	p<	\infty$,	we	denote	by	$\mathbb{W}^{m,p}(\O)$ (resp. $\mathrm{W}^{m,p}(\O)$) the standard Sobolev space of matrix/vector-valued (resp. scalar-valued) functions	whose weak derivative up to order $m$	belong	to	the	Lebesgue	space	$\L^p(\O)$ (resp. $\mathrm{L}^p(\O)$) and set	$\H^m(\O)=\mathbb{W}^{m,2}(\O)$	and	$\H^0(\O)=\L^2(\O)$.

	Let us first define the space $\mathscr{V}:=\{\v \in (\mathrm{C}^\infty_c(\O))^d \,: \text{ div}\v=0\}$. Then, we denote the spaces $\H$, $\V$ and ${\L}^{p}_{\sigma}$ ($1<p<\infty$) as the closure of $\mathscr{V}$ in $\L^2(\O)$, $\H^1(\O)$ and $\L^p(\O)$ ($1<p<\infty$), respectively. For the characterization of spaces, we refer readers to \cite[Chapter 1]{Temam_1984}.

	Next, let us introduce  the scalar product between two matrices $A:B=\Tr(AB^T)$
	and denote $\vert A\vert^2:=A:A.$
	The divergence of a  matrix $A\in \mathcal{M}_{d\times d}(E)$ is given by 
	$\left\{\text{div}(A)_i\right\}_{i=1}^{i=d}=\left\{\displaystyle\sum_{j=1}^d\partial_ja_{ij}\right\}_{i=1}^{i=d}. $
	The space $\H$ is endowed with the  $\L^2$-inner product $(\cdot,\cdot)$ and the associated norm $\Vert \cdot\Vert_{2}$. On the functional space $\V$, we will consider  the following inner product $(\u,\v)_{\V}:= (\u,	\v) + (\nabla	\u,\nabla	\v),$ and denote by $\Vert \cdot\Vert_{\V}$ the corresponding norm. The usual norms on the classical Lebesgue and Sobolev spaces $\mathrm{L}^p(\O)$ (resp. $\mathbb{L}^p(\O)$) and $\mathrm{W}^{m,p}(\O)$ (resp. $\mathbb{W}^{m,p}(\O)$) will be denoted by $\|\cdot \|_p$ and 
	$\|\cdot\|_{\mathrm{W}^{m,p}}$ (resp. $\|\cdot\|_{\mathbb{W}^{m,p}}$), respectively.
	In addition, given a Banach space $E$, we will denote by $E^\prime$ its dual, and we denote	by	$\langle	\cdot,\cdot\rangle$ the duality pairing between $E^{\prime}$ and $E$.

	Let	us	introduce	the	following	Banach	space	$(\X:=\mathbb{W}_0^{1,4}(\O)\cap \V,\Vert	\cdot\Vert_{\X} :=\Vert	\cdot\Vert_{\mathbb{W}^{1,4}}+\|\cdot\|_{\V}.$
	Indeed,	we	recall	that	$\mathbb{W}_0^{1,4}(\O)$	endowed	with $\Vert	\cdot\Vert_{\mathbb{W}^{1,4}}$-norm	is a Banach	space,	where		$$\Vert	\w \Vert_{\mathbb{W}^{1,4}}^4=\int_{\O}	\vert	\w(x)\vert^4 \d x+\int_{\O} \vert	\nabla	\w(x)\vert^4 \d x.$$

	\begin{lemma}
		In this lemma, we recall a couple to inequalities which will be used in the sequel:
		\begin{itemize}
			\item   The well-known Gagliardo-Nirenberg inequality  (\cite[Theorem 1]{Nirenberg_1959}) gives (for $d\in \{2,3\}$)
			\begin{align}
				\|\w\|_{{3}} &\leq C \|\w\|^{\frac{12+d}{3(4+d)}}_{{2}} \|\nabla \w\|^{\frac{2d}{3(4+d)}}_{{4}}, && \quad \text{	for all } \; \w\in \L^{2} (\O)\cap\mathbb{W}_0^{1,4}(\O), \label{Gen_lady-3}\\
				\|\w\|_{{4}} &\leq C \|\w\|^{\frac{4}{4+d}}_{{2}} \|\nabla \w\|^{\frac{d}{4+d}}_{{4}}, && \quad \text{	for all } \; \w\in \L^{2} (\O)\cap\mathbb{W}_0^{1,4}(\O).\label{Gen_lady-4}
			\end{align}
			\item  There exists a constant $C_{S,d}>0$ (\cite[Subsection 2.4]{Kesavan_1989}) such that 
			\begin{align}\label{Sobolev-embedding3}
				\|\w\|_{\infty} \leq C_{S,d} \|\nabla\w\|_{4}, \quad \text{	for all } \;	\w\in	\mathbb{W}_0^{1,4}(\O).
			\end{align}
			
			\item  There	exists a constant $C_{K,d}>0$	such	that (a Korn-type inequality, see	\cite[Theorem 2 (ii)]{DiFratta+Solombrino_Arxiv} and \cite[Theorem 4, p.90]{Kondratev+Oleinik_1988})
			\begin{align}\label{Korn-ineq}
				\Vert	\nabla \w \Vert_{4}	\leq	C_{K,d}\Vert	\Arm(\w)	\Vert_4,\quad \text{	for all } \;	\w\in	\mathbb{W}_0^{1,4}(\O).
			\end{align}		
			\item In view of \eqref{Sobolev-embedding3} and \eqref{Korn-ineq}, there exists a constant $\mathrm{M}_d>0$ such that
			\begin{align}\label{Sobolev+Korn}
				\|\w\|_{\infty} \leq \mathrm{M}_d \|\Arm(\w)\|_{4}, \quad \text{	for all } \;	\w\in	\mathbb{W}_0^{1,4}(\O).
			\end{align}
		
		\item In view of \eqref{Gen_lady-4}, \eqref{Korn-ineq} and Young's inequality, there exists a constant $\mathrm{N}_d>0$ such that
		\begin{align}\label{Gag-Korn-ineq}
			\|\w\|^4_4 \leq \mathrm{N}_d\left[ \|\Arm(\w)\|_{4}^4 + \|\w\|_{2}^4 \right], \quad \text{	for all } \;  \w\in \L^{2} (\O )\cap\mathbb{W}_0^{1,4}(\O ).
		\end{align}
		\end{itemize}
		
	\end{lemma}

	Throughout the article,  we denote by $C$   generic constant, which may vary from line to line.

	\subsection{The Helmholtz projection}
	The following content is motivated by the studies \cite{Farwig+Kozono+Sohr_2005,Farwig+Kozono+Sohr_2007,Kunstmann_2010}, which address the Helmholtz projection in the setting of general unbounded domains. We define
	\begin{align}
		\widetilde{\L}^p(\O) : = \begin{cases}
			\L^p(\O)\cap \L^2(\O), & p\in[2,\infty)\\
			\L^p(\O)+ \L^2(\O), & p\in (1,2),
		\end{cases}
	\end{align}
	and 
	\begin{align}
		\widetilde{\L}^p_{\sigma}  : = \begin{cases}
			\L^p_{\sigma} \cap \L^2_{\sigma} , & p\in[2,\infty)\\
			\L^p_{\sigma} + \L^2_{\sigma}, & p\in (1,2).
		\end{cases}
	\end{align}
	
	Now, we recall that the Helmholtz projector $\mathcal{P}_{p}: \widetilde{\L}^p(\O) \to \widetilde{\L}^p_{\sigma}$, which is a linear bounded operator characterized by the following decomposition
	$\v=\mathcal{P}_p\v+\nabla \varphi, \;  \varphi \in \widetilde{\mathrm{G}}^p(\O)$ satisfying $(\mathcal{P}_p)^* = \mathcal{P}_{p'}$ with $\frac{1}{p}+\frac{1}{p'}=1$ (cf. \cite{Farwig+Kozono+Sohr_2007}), where
	\begin{align}
		\widetilde{\mathrm{G}}^p(\O) : = \begin{cases}
			\mathrm{G}^p(\O)\cap \mathrm{G}^2(\O), & p\in[2,\infty)\\
			\mathrm{G}^p(\O)+ \mathrm{G}^2(\O), & p\in (1,2),
		\end{cases}
	\end{align}
	and 
	\begin{align}
		\mathrm{G}^p(\O) = \{\nabla q \in \L^p(\O) \; : \; q\in \mathrm{L}^p_{loc} (\O)\}.
	\end{align}
	Let us fix the notations $\widetilde{\mathbb{W}}^{1,p}_{0,\sigma}  := \widetilde{\mathbb{W}}_0^{1,p}(\O) \cap \widetilde{\L}^{p}_{\sigma}$,  and $\widetilde{\mathbb{W}}^{-1,p}(\O):= (\widetilde{\mathbb{W}}_0^{1,p'}(\O))'$ and $\widetilde{\mathbb{W}}^{-1,p}_{\sigma} : = (\widetilde{\mathbb{W}}^{1,p'}_{0,\sigma})'$ with $\frac{1}{p}+\frac{1}{p'}=1$ for the dual spaces, where
	\begin{align}
		\widetilde{\mathbb{W}}^{1,p}_0(\O) : = \begin{cases}
			\mathbb{W}^{1,p}_0(\O)\cap \mathbb{W}^{1,2}_0(\O), & p\in[2,\infty)\\
			\mathbb{W}^{1,p}_0(\O)+ \mathbb{W}^{1,2}_0(\O), & p\in (1,2).
		\end{cases}
	\end{align}
	
	\begin{lemma}[{\cite[Proposition 3.1]{Kunstmann_2010}}]
		The Helmholtz projection $\mathcal{P}_{p}$ has a continuous linear extension $\widetilde{\mathcal{P}}_{p}: \widetilde{\mathbb{W}}^{-1,p}(\O) \to \widetilde{\mathbb{W}}^{-1,p}_{\sigma}$ given by the restriction
		\begin{align}
			\widetilde{\mathcal{P}}_{p} \Psi  = \Psi |_{\widetilde{\mathbb{W}}^{1,p'}_{0,\sigma}}, \;\; \text{ for } \; \; \Psi \in \widetilde{\mathbb{W}}^{-1,p}_{\sigma} \;\; \text{ with } \; \;\frac{1}{p}+\frac{1}{p'}=1.
		\end{align}
	\end{lemma}
	
	\begin{remark}
		We will denote $\widetilde{\mathcal{P}}_{\frac43}$ by $\mathcal{P}$ in the sequel.
	\end{remark}

	\subsection{Linear and nonlinear operators}
	Let us define linear operator $\A:\X\to \X'$ by $\A\v:= - \mathcal{P}\Delta\v$ such that
	\begin{equation*}
		\langle \A\v, \u\rangle = (\nabla\v,\nabla\u).
	\end{equation*}
	Remember that the operator $\A$ is a non-negative, self-adjoint operator in $\H$  and \begin{align*}%\label{2.7a}
		\left<\A\v,\v\right>=\|\v\|_{\V}^2,\ \textrm{ for all }\ \v\in\X, \ \text{ so that }\ \|\A\v\|_{\X^{\prime}}\leq \|\v\|_{\X}.
	\end{align*}

	\medskip
	
	Next, we define the {trilinear form} $b(\cdot,\cdot,\cdot):\X\times\X\times\X\to\R$ by $$b(\u,\v,\w)=\int_{\O}(\u(x)\cdot\nabla)\v(x)\cdot\w(x)\d x= \sum_{i,j=1}^d\int_{\O}\u_i(x)\frac{\partial \v_j(x)}{\partial x_i}\w_j(x)\d x.$$ We also define an operator $\B: \X\times\X \to \X'$  by $\B(\u, \v):= \mathcal{P}(\u\cdot\nabla)\v$ such that
	\begin{align*}
		\left\langle   \B(\u, \v), \w  \right \rangle = b(\u,\v,\w).
	\end{align*}
	Using an integration by parts, it is immediate that 
	\begin{equation}\label{b0}
		\left\{
		\begin{aligned}
			b(\u,\v,\v) &= 0, && \text{ for all }\ \u,\v \in\X,\\
			b(\u,\v,\w) &=  -b(\u,\w,\v), && \text{ for all }\ \u,\v,\w\in \X.
		\end{aligned}
		\right.\end{equation}
	\iffalse 
	\begin{remark}\label{rem-trilinear-ext}
		For $\u,\v\in  \mathbb{L}^{4}_{\sigma}$ and $\w\in\mathbb{W}_0^{1,4}(\O)$, we have 
		\begin{align*}
			|b(\u,\v,\w)|=|b(\u,\w,\v)| \leq \|\u\|_{4}\|\v\|_{4}\|\nabla\w\|_{4} \leq C \|\u\|_{4}\|\v\|_{4}\|\w\|_{\mathbb{W}^{1,4}},
		\end{align*}
		for some positive constant $C$. Thus $b$ can be uniquely extended to the trilinear form (still denoted by the same) $b:\mathbb{L}^{4}_{\sigma}\times\mathbb{L}^{4}_{\sigma} \times \mathbb{W}_0^{1,4}(\O) \to \R$. In parallel, the operator $\B$ can be uniquely extended to a bounded linear operator (still denoted by the same) $\B:\mathbb{L}^{4}_{\sigma}\times\mathbb{L}^{4}_{\sigma}\to \mathbb{W}^{-1,\frac43}(\O)$ such that 
		\begin{align*}%\label{eqn-trilinear-ext}
			\|\B(\u,\v)\|_{\mathbb{W}^{-1,\frac43}}\leq C \|\u\|_{4}\|\v\|_{4}, \;\;\; \u,\v\in\H.
		\end{align*}
	\end{remark}
\fi 
	
	\medskip
	
	Let us now define an operator $\J: \X \to \X'$  by $\J(\v):=-\mathcal{P}\diver(\Arm(\v)\Arm(\v))$ such that 
	\begin{align*}
		\langle \J(\v),\u \rangle = \frac{1}{2}\int_{\O} \Arm(\v(x))\Arm(\v(x)):\Arm(\u(x))\d x.
	\end{align*}
	Note that for $\v\in\X$, we have 
	\begin{align*}
		\|\J(\v)\|_{\X^{\prime}} \leq C \|\v\|^2_{\X}.
	\end{align*}

	\medskip

	Finally, we define an operator $\K: \X \to \X'$  by $\K(\v):=-\mathcal{P}\diver(|\Arm(\v)|^{2}\Arm(\v))$ such that 
	\begin{align*}
		\langle \K(\v),\u \rangle = \frac{1}{2}\int_{\O} |\Arm(\v(x))|^2\Arm(\v(x)):\Arm(\u(x)) \d x.
	\end{align*} It is immediate that $$\langle\mathcal{K}(\v),\v\rangle =\frac12\|\Arm(\v)\|_{4}^{4}.$$ Note that for $\v\in\X$, we have 
	\begin{align*}
		\|{\K}(\v)\|_{\X^{\prime}} \leq C \|\v\|^3_{\X},
	\end{align*}
	
	\iffalse 
		\begin{remark}\label{J=0_d=2}
		In case of $\R^2$, due to divergence-free condition, we obtain $\Tr([\Arm(\v)]^3)=0$. Hence, for $d=2$, we have for $\v\in\X$
		\begin{align*}%\label{J-d=2}
			\langle \J(\v),\v \rangle = 	\int_{\R^2}\Tr[(\Arm(\v(x)))^3]\d x =0.
		\end{align*}
	\end{remark}
\fi 
	
	\begin{remark}
	1. From \cite[Equation (2.20)]{Hamza+Paicu_2007}, we have 
		\begin{align}\label{J1-J2}
			& |\alpha\left<	\J(\v_1) - \J(\v_2), \v_1 - \v_2 \right>| 
			  \leq  \frac{|\alpha|}{2}\int_{\O} |\Arm(\v_1-\v_2)|^2( |\Arm(\v_1)|+|\Arm(\v_2)|) \d x,
		\end{align}
	
	2. From \cite[Equation (2.13)]{Hamza+Paicu_2007}, we have 
	\begin{align}\label{K1-K2}
		&  \beta\left<	\K(\v_1) - \K(\v_2), \v_1 - \v_2 \right>
		\nonumber\\ & = \frac{\beta}{2}\int_{\O}( |\Arm(\v_1)|^2-|\Arm(\v_2)|^2)^2 \d x + \frac{\beta}{2}\int_{\O} |\Arm(\v_1-\v_2)|^2( |\Arm(\v_1)|^2+|\Arm(\v_2)|^2) \d x.
		%\nonumber\\ &   \geq \frac{\beta}{2}\int_{\O}( |\Arm(\v_1)|^2-|\Arm(\v_2)|^2)^2 \d x + \frac{\beta}{4}\|\Arm(\v_1-\v_2)\|^4_4.
	\end{align}
	\end{remark}

	\subsection{Abstract formulation}	Taking the projection $\mathcal{P}$ on  \eqref{equationV1}, we obtain for $t\geq \mathfrak{r},$ $\mathfrak{r}\in\mathbb{R}$ 
	\begin{equation}\label{STGF}
		\left\{
		\begin{aligned}
			\d\u&=-\left[\nu \A\u+\B(\u)+ \alpha \J(\u) +\beta \K(\u)-\mathcal{P}\f\right]\d t +\sigma\u\d \W, && \text{in } \O \times(\mathfrak{r},\infty), \\ 
			\u|_{t=\mathfrak{r}}&=\u_{\mathfrak{r}}, &&  x\in \O,
		\end{aligned}
		\right.
	\end{equation}
	where  $\sigma>0$ and the stochastic integral is understood in the It\^o sense. Here, $\W(t,\omega)$ is the standard scalar Wiener process on the filtered probability space $(\Omega, \mathscr{F}, (\mathscr{F}_{t})_{t\in\R}, \mathbb{P})$, where $$\Omega=\{\omega\in C(\R;\R):\omega(0)=0\},$$ endowed with the compact-open topology given by the complete metric
	\begin{align*}
		d_{\Omega}(\omega,\omega'):=\sum_{m=1}^{\infty} \frac{1}{2^m}\frac{\|\omega-\omega'\|_{m}}{1+\|\omega-\omega'\|_{m}},\text{ where } \|\omega-\omega'\|_{m}:=\sup_{-m\leq t\leq m} |\omega(t)-\omega'(t)|,
	\end{align*}
	and $\mathscr{F}$ is the Borel sigma-algebra induced by the compact-open topology of $(\Omega,d_{\Omega}),$ $\mathbb{P}$ is the two-sided Wiener measure on $(\Omega,\mathscr{F})$. Also, define $\{\vartheta_{t}\}_{t\in\R}$ by 
	\begin{equation*}
		\vartheta_{t}\omega(\cdot)=\omega(\cdot+t) -\omega(t), \ \ \   t\in\R\ \text{ and }\ \omega\in\Omega.
	\end{equation*}
	Hence, $(\Omega,\mathscr{F},\mathbb{P},\{\vartheta_{t}\}_{t\in\R})$ is a metric dynamical system. We expect a solution to the system \eqref{STGF} in the following sense:
	\begin{definition}\label{GWS}
		A stochastic process $\u(t):=\u(t;\mathfrak{r},\omega,\u_{\mathfrak{r}})$ is a $\textsl{global (analytic) weak solution}$ of the system \eqref{STGF} for $t\geq\mathfrak{r}$ and $\omega\in\Omega$ if for $\mathbb{P}$-a.s. $\omega\in\Omega$
		\begin{align*}
			\u(\cdot)\in\mathrm{C}([\mathfrak{r},\infty);\H)\cap\mathrm{L}^{2}_{\mathrm{loc}}(\mathfrak{r},\infty;\V)\cap \mathrm{L}^{4}_{\mathrm{loc}}(\mathfrak{r},\infty;\mathbb{W}^{1,4}(\O))
		\end{align*}
		and 
		\begin{align*}
			&(\u(t),\boldsymbol{\psi})+\nu\int_{\mathfrak{r}}^{t}(\nabla\u(\xi),\nabla\boldsymbol{\psi})\d\xi+\int_{\mathfrak{r}}^{t}b(\u(\xi),\u(\xi),\boldsymbol{\psi})\d\xi\nonumber\\&=(\u_{\mathfrak{r}},\boldsymbol{\psi}) - \alpha \int_{\mathfrak{r}}^{t}\langle\J(\u(\xi)), \boldsymbol{\psi}\rangle \d\xi - \beta \int_{\mathfrak{r}}^{t}\langle\K(\u(\xi)) , \boldsymbol{\psi}\rangle \d\xi +\int_{\mathfrak{r}}^{t}(\f(\xi),\boldsymbol{\psi})\d\xi +\sigma\int_{\mathfrak{r}}^{t}(\u(\xi),\boldsymbol{\psi})\d\W(\xi),
		\end{align*}
		for every $t\geq\mathfrak{r}$ and for every $\boldsymbol{\psi}\in\X$.
	\end{definition}
	
\subsection{Ornstein-Uhlenbeck process}	Let us now consider
	\begin{align}\label{OU1}
		y(\vartheta_{t}\omega) =  \int_{-\infty}^{t} e^{-(t-\xi)}\d \W(\xi), \ \ \omega\in \Omega,
	\end{align} which is the \textsl{stationary solution} of the one-dimensional \textsl{Ornstein-Uhlenbeck equation}
	\begin{align}\label{OU2}
		\d y(\vartheta_t\omega) +  y(\vartheta_t\omega)\d t =\d\W(t).
	\end{align}
	It is known from \cite{FAN} that there exists a $\vartheta$-invariant subset $\widetilde{\Omega}\subset\Omega$ of full measure such that $y(\vartheta_t\omega)$ is continuous in $t$ for every $\omega\in \widetilde{\Omega},$ and
	\begin{align}
		%\mathbb{E}\left(|y(\vartheta_s\omega)|^{\xi}\right)&=\frac{\Gamma\left(\frac{1+\xi}{2}\right)}{\sqrt{\pi\sigma^{\xi}}}, \ \text{ for all } \xi>0, s\in \R,\label{Z2}\\
		\lim_{t\to\infty}\frac{|\omega(t)|}{t}=	\lim_{t\to \pm \infty} \frac{|y(\vartheta_t\omega)|}{|t|}&=	\lim_{t\to \pm \infty} \frac{1}{t} \int_{0}^{t} y(\vartheta_{\xi}\omega)\d\xi =\lim_{t\to \infty} e^{-\delta t}|y(\vartheta_{-t}\omega)| =0,\label{Z3}
	\end{align}
	for all $\delta>0$. For further analysis of this work, we do not distinguish between $\widetilde{\Omega}$ and $\Omega$.
	%where $\Gamma$ is the Gamma function.
	Since, $\omega(\cdot)$ has sub-exponential growth  (cf. \cite[Lemma 11]{CGSV}), $\Omega$ can be written as $\Omega=\bigcup\limits_{N\in\N}\Omega_{N}$, where
	\begin{align*}
		\Omega_{N}:=\{\omega\in\Omega:|\omega(t)|\leq Ne^{|t|},\ \text{ for all }\ t\in\R\}, \ \text{ for every } \ N\in\N.
	\end{align*}
	Moreover, for each $N\in\N$, $(\Omega_{N},d_{\Omega_{N}})$ is a polish space (cf. \cite[Lemma 17]{CGSV}).
	\begin{lemma}\label{conv_z}
		For each $N\in\N$, suppose $\omega_\mathfrak{k},\omega_0\in\Omega_{N}$ are such that $d_{\Omega}(\omega_\mathfrak{k},\omega_0)\to0$ as $\mathfrak{k}\to+\infty$. Then, for each $\mathfrak{r}\in\R$ and $T\in\R^+$ ,
		\begin{align}
			\sup_{t\in[\mathfrak{r},\mathfrak{r}+T]}&\bigg[|y(\vartheta_{t}\omega_\mathfrak{k})-y(\vartheta_{t}\omega_0)|+|e^{ y(\vartheta_{t}\omega_\mathfrak{k})}-e^{ y(\vartheta_{t}\omega_0)}|\bigg]\to 0 \ \text{ as } \ \mathfrak{k}\to+\infty,\nonumber\\
			\sup_{\mathfrak{k}\in\N}\sup_{t\in[\mathfrak{r},\mathfrak{r}+T]}&|y(\vartheta_{t}\omega_\mathfrak{k})|\leq C(\mathfrak{r},T,\omega_0).\label{conv_z2}
		\end{align}
	\end{lemma}
	\begin{proof}
		See the proofs of \cite[Corollary 22]{CLL} and \cite[Lemma 2.5]{YR}.
	\end{proof}

\subsection{Non-autonomous random dynamical system}\label{NRDS}
	Now, define $\z(t,\omega):=e^{-\sigma y(\vartheta_{t}\omega)}$ and a new function  $\v$  by 
	\begin{align}\label{COV}
		\v(t;\mathfrak{r},\omega,\v_{\mathfrak{r}})=\z(t,\omega)\u(t;\mathfrak{r},\omega,\u_{\mathfrak{r}}) \ \ \
		\text{ with }
		\ \ \	\v_{\mathfrak{r}}=\z(\mathfrak{r},\omega)\u_{\mathfrak{r}},
	\end{align}
	where $\u(t;\mathfrak{r},\omega,\u_{\mathfrak{r}})$ and $y(\vartheta_{t}\omega)$ are the solutions of \eqref{STGF} (in the sense of Definition \ref{GWS}) and \eqref{OU2}, respectively. Then $\v(\cdot;\mathfrak{r},\omega,\v_{\mathfrak{r}})$ satisfies the following system:
	\begin{equation}\label{eqn-CTGF-Pressure}
		\left\{		\begin{aligned}
			\dfrac{\partial	\v}{\partial t} - \nu \Delta \v +\left[\frac{\sigma^2}{2}-\sigma y(\vartheta_{t}\omega)\right]\v + \z^{-1}(t,\omega)(\v\cdot \nabla)\v & - \alpha \z^{-1}(t,\omega) \text{div}((\Arm(\v))^2) 
			\\  -  \beta \z^{-2}(t,\omega)  \text{div}(|\Arm(\v)|^2\Arm(\v)) &= \z(t,\omega) \f- \z(t,\omega) \nabla \textbf{P}, &&\text{in }  \O  \times (\mathfrak{r},\infty),\\
			\text{div}\; \v&=0,  &&   \text{in } \O  \times [\mathfrak{r},\infty),\\
			\v &= \boldsymbol{0}, &&   \text{on } \partial\O \times [\mathfrak{r},\infty),\\
			\v(x,\mathfrak{r})&=\v_\mathfrak{r}(x), \quad &&    \text{in } \O ,
		\end{aligned}\right.
	\end{equation}
and the following projected form:
	\begin{equation}\label{CTGF}
		\left\{
		\begin{aligned}
			\frac{\d\v}{\d t}+\nu \A\v +\left[\frac{\sigma^2}{2}-\sigma y(\vartheta_{t}\omega)\right]\v+\z^{-1}(t,\omega)\B\big(\v\big) & + \alpha \z^{-1}(t,\omega)\J\big(\v\big) 
			\\  +  \beta \z^{-2}(t,\omega) \K\big(\v\big) & = \z(t,\omega)\mathcal{P}\f , && t> \mathfrak{r}, \\ 
			\v|_{t=\mathfrak{r}}&=\v_{\mathfrak{r}}, && x\in \O,
		\end{aligned}
		\right.
	\end{equation}
	in $\X'$.  The solution of the system \eqref{CTGF} is understood in the following sense.
	
	\begin{definition}\label{def-CTGF}
	Let $\mathfrak{r}\in \R$, $\omega\in\Omega$ and $\v_\mathfrak{r}\in\H$.	A function $\v(\cdot):=\v(\cdot;\mathfrak{r},\omega,\v_{\mathfrak{r}})$ is called a \textit{weak solution} to system \eqref{CTGF} on time interval $[\mathfrak{r}, \infty)$, if $$\v \in  \mathrm{C}([\mathfrak{r},\infty); \H) \cap \mathrm{L}^{2}_{\mathrm{loc}}(\mathfrak{r},\infty; \V)\cap\mathrm{L}^{4}_{\mathrm{loc}}(\mathfrak{r},\infty; \mathbb{W}^{1,4}_0(\O)), \;\;\; \frac{\d\v}{\d t}\in\mathrm{L}^{2}_{\mathrm{loc}}(\mathfrak{r},\infty;\V^{\prime})+\mathrm{L}^{\frac43}_{\mathrm{loc}}(\mathfrak{r},\infty;\mathbb{W}^{-1,\frac43}(\O)),$$ and it satisfies 
		\begin{itemize}
			\item [(i)] for any $\boldsymbol{\psi}\in \X,$ 
			\begin{align}\label{W-TGF}
				\left<\frac{\d\v(t)}{\d t}, \boldsymbol{\psi}\right>&=  - \bigg\langle \nu \A\v(t)+\left[\frac{\sigma^2}{2}-\sigma y(\vartheta_{t}\omega)\right]\v(t)+\z^{-1}(t,\omega)\B\big(\v(t)\big)  + \alpha \z^{-1}(t,\omega) \J\big(\v(t)\big) 
				\nonumber \\ & \qquad + \beta \z^{-2}(t,\omega)\K\big(\v(t)\big) -\z(t,\omega) \f(t) , \boldsymbol{\psi} \bigg\rangle,
			\end{align}
			for a.e. $t\in[\mathfrak{r},\infty);$
			\item [(ii)] the initial data:
			$$\v(\mathfrak{r})=\v_\mathfrak{r} \ \text{ in }\ \H.$$
		\end{itemize}
	\end{definition}
The following result discuss  the existence and uniqueness of weak solution of system \eqref{CTGF} in the sense of Definition \ref{def-CTGF}.
	\begin{lemma}\label{lem-Sol-CTGF}
		Let $\mathfrak{r}\in\R$, $\omega\in\Omega$ and   \eqref{third-grade-paremeters-res} hold. For $\v_\mathfrak{r} \in \H$ and $\f\in \mathrm{L}^{2}_{\mathrm{loc}}(\R;\H^{-1}(\O))$, there exists a unique weak solution $\v(\cdot):=\v(\cdot;\mathfrak{r},\omega,\v_{\mathfrak{r}})$ to system \eqref{CTGF} in the sense of Definition \ref{def-CTGF}.
	\end{lemma}
	
	\begin{proof}
		An application of the standard Faedo-Galerkin approximation method and Minty-Browder technique ensures that for all $t\geq \mathfrak{r},\ \mathfrak{r}\in\R$, $\omega\in\Omega$, for every $\v_{\mathfrak{r}}\in\H$ and $\f\in\mathrm{L}^{2}_{\text{loc}}(\R;\H^{-1}(\O))$, the system \eqref{CTGF} has a unique weak solution $\v\in\mathrm{C}([\mathfrak{r},+\infty);\H)\cap\mathrm{L}^2_{\mathrm{loc}}(\mathfrak{r},+\infty;\V)\cap\mathrm{L}^4_{\mathrm{loc}}(\mathfrak{r},+\infty;\mathbb{W}^{1,4}(\O))$ in the sense of Definition \ref{def-CTGF}. Since the proof closely follows those of \cite[Theorem 3.1]{KK+FC1} and \cite[Theorem 3.8]{KK+FC2}, we are not repeating it here. 
	\end{proof}

	 Hence, for every $\u_{\mathfrak{r}}\in\H$ and $\f\in\mathrm{L}^{2}_{\text{loc}}(\R;\H^{-1}(\O))$, there exists a unique solution $\u(\cdot)$ to the system \eqref{STGF} in the sense of Definition \ref{GWS}. For $t\in\R^+, \mathfrak{r}\in\R, \omega\in\Omega$ and $\u_{\mathfrak{r}}\in\H$, define a map $\Phi:\R^+\times\R\times\Omega\times\H\to\H$  by 
	\begin{align}\label{Phi}
		\Phi(t,\mathfrak{r},\omega,\u_{\mathfrak{r}}) =\u(t+\mathfrak{r};\mathfrak{r},\vartheta_{-\mathfrak{r}}\omega,\u_{\mathfrak{r}})=\frac{\v(t+\mathfrak{r};\mathfrak{r},\vartheta_{-\mathfrak{r}}\omega,\v_{\mathfrak{r}})}{\z(t+\mathfrak{r},\vartheta_{-\mathfrak{r}}\omega)},
	\end{align}
	with $\v_{\mathfrak{r}}=\z(\mathfrak{r},\vartheta_{-\mathfrak{r}}\omega)\u_{\mathfrak{r}}$.

	The following lemma plays a crucial role in  proving the $(\mathscr{F},\mathscr{B}(\H))$-measurability  of the map $\Phi(t,\mathfrak{r},\cdot,\x):\Omega\to\mathbb{H}$, for every fixed $t>0$, $\mathfrak{r}\in\R$ and $\x\in\H$.

 \begin{lemma}\label{LusinC}
	Assume that $\mathfrak{r}\in\R$, $t>\mathfrak{r}$, $\f\in\mathrm{L}^2_{\mathrm{loc}}(\R;\H^{-1}(\O))$ and $\v_{\mathfrak{r}}\in\H$. For each $N\in\N$, the mapping $\omega\mapsto\v(t;\mathfrak{r},\omega,\v_{\mathfrak{r}})$ $($solution of \eqref{CTGF}$)$ is continuous from $(\Omega_{N},d_{\Omega_N})$ to $\H$.
\end{lemma}
\begin{proof}
	Assume that $\omega_\ell,\omega_0\in\Omega_N,\ N\in\mathbb{N}$ such that $d_{\Omega_N}(\omega_\ell,\omega_0)\to0$ as $\ell\to+\infty$. Let $\yscr^\ell:=\v^\ell-\v^0,$ where $\v^\ell(\cdot)=\v(\cdot;\mathfrak{r},\omega_\ell,\v_{\mathfrak{r}})$ and $\v^0(\cdot)=\v(\cdot;\mathfrak{r},\omega_0,\v_{\mathfrak{r}})$. Then, $\yscr^\ell$ satisfies:
	\begin{align}\label{LC1}
		 & \frac{\d\yscr^\ell}{\d t}
		 \nonumber\\ &=-\nu \A\yscr^\ell-\left[\frac{\sigma^2}{2}-\sigma y(\vartheta_{t}\omega_\ell)\right]\yscr^\ell -\z^{-1}(t,\omega_\ell)\B\big(\v^\ell\big)+\z^{-1}(t,\omega_0)\B\big(\v^0\big)
		 - \alpha\z^{-1}(t,\omega_\ell)\J\big(\v^\ell\big) 
		  \nonumber\\&\quad+ \alpha\z^{-1}(t,\omega_0)\J\big(\v^0\big)  -\beta\z^{-2}(t,\omega_\ell)\K\big(\v^\ell\big)+\beta\z^{-2}(t,\omega_0)\K\big(\v^0\big)\
		 +\sigma \left[y(\vartheta_{t}\omega_\ell)-y(\vartheta_{t}\omega_0)\right]\v^0
		 \nonumber\\ & \quad + \left[\z(t,\omega_\ell)-\z(t,\omega_0)\right]\mathcal{P}\f
		\nonumber\\ & =-\nu \A\yscr^\ell-\left[\frac{\sigma^2}{2}-\sigma y(\vartheta_{t}\omega_\ell)\right]\yscr^\ell -\z^{-1}(t,\omega_\ell)[\B\big(\v^\ell,\yscr^\ell\big)-\B\big(\yscr^\ell,\v^0\big)] - [\z^{-1}(t,\omega_\ell) - \z^{-1}(t,\omega_0)]\B\big(\v^0\big)
		\nonumber\\&\quad - \alpha\z^{-1}(t,\omega_\ell)[\J\big(\v^\ell\big)-\J\big(\v^0\big)]- \alpha[\z^{-1}(t,\omega_0)-\z^{-1}(t,\omega_\ell)]\J\big(\v^0\big)  -\beta\z^{-2}(t,\omega_\ell)[\K\big(\v^\ell\big)-\K\big(\v^0\big)]
		\nonumber\\ & \quad -\beta[\z^{-2}(t,\omega_\ell)-\z^{-2}(t,\omega_0)]\K\big(\v^0\big)\
		  +\sigma \left[y(\vartheta_{t}\omega_\ell)-y(\vartheta_{t}\omega_0)\right]\v^0+ \left[\z(t,\omega_\ell)-\z(t,\omega_0)\right]\mathcal{P}\f,
	\end{align}
in $\X'$. Taking the inner product to equation \eqref{LC1} with $\yscr^\ell(\cdot)$ and using \eqref{b0},  we find
\begin{align}\label{LC2}
	& \frac{1}{2} \frac{\d }{\d t} \|\yscr^\ell\|^2_2 + \frac{\nu}{2}\|\Arm(\yscr^\ell)\|^2_2 + \frac{\sigma^2}{2}\|\yscr^\ell\|^2_2 
	\nonumber\\ & = \sigma y(\vartheta_{t}\omega_\ell)\|\yscr^\ell\|^2_2 + \z^{-1}(t,\omega_\ell)b\big(\yscr^\ell, \yscr^\ell ,\v^0\big) + [\z^{-1}(t,\omega_\ell) - \z^{-1}(t,\omega_0)]b\big(\v^0, \yscr^\ell,\v^0\big) 
	\nonumber\\&\quad - \alpha\z^{-1}(t,\omega_\ell)\left< \J\big(\v^\ell\big)-\J\big(\v^0\big),\yscr^\ell\right> - \alpha[\z^{-1}(t,\omega_0)-\z^{-1}(t,\omega_\ell)]\left<\J\big(\v^0\big) ,\yscr^\ell \right> 
	\nonumber\\ & \quad  -\beta\z^{-2}(t,\omega_\ell)\left<\K\big(\v^\ell\big)-\K\big(\v^0\big) ,\yscr^\ell\right> 
	 -\beta[\z^{-2}(t,\omega_\ell)-\z^{-2}(t,\omega_0)]\left<\K\big(\v^0\big) ,\yscr^\ell\right> 
	 \nonumber\\ & \quad +\sigma \left[y(\vartheta_{t}\omega_\ell)-y(\vartheta_{t}\omega_0)\right]\left(\v^0, \yscr^\ell\right)+ \left[\z(t,\omega_\ell)-\z(t,\omega_0)\right] \left<\f,\yscr^\ell\right>.
\end{align}
For $\varepsilon_0:=1-\sqrt{\frac{\alpha^2}{2\beta\nu}}\in(0,1)$, an application of H\"older's inequality, \eqref{Sobolev+Korn}, \eqref{J1-J2}, \eqref{K1-K2} and Young's inequality give
\begin{align}
	|\z^{-1}(t,\omega_\ell)b\big(\yscr^\ell, \yscr^\ell ,\v^0\big)|& \z^{-1}(t,\omega_\ell)\|\yscr^\ell\|_2 \|\nabla\yscr^\ell\|_2 \|\v^0\|_{\infty}
	\nonumber\\ & \leq C \z^{-1}(t,\omega_\ell)\|\yscr^\ell\|_2 \|\Arm(\yscr^\ell)\|_2 \|\Arm(\v^0)\|_{4}
	\nonumber\\ &\leq  \frac{\nu\varepsilon_0}{16}\|\Arm(\yscr^\ell)\|^2_2 +  C \z^{-2}(t,\omega_\ell) \|\Arm(\v^0)\|^2_{4}\|\yscr^\ell\|_2^2,\label{LC3} \\
	|[\z^{-1}(t,\omega_\ell) - \z^{-1}(t,\omega_0)] b\big(\v^0, \yscr^\ell,\v^0\big)|& \leq \frac{\nu\varepsilon_0}{16}\|\Arm(\yscr^\ell)\|^2_2 +  C [\z^{-1}(t,\omega_\ell) - \z^{-1}(t,\omega_0)] ^2 \|\Arm(\v^0)\|^2_{4}\|\v^0\|_2^2, \label{LC4} \\
     |\alpha\z^{-1}(t,\omega_\ell)\left< \J\big(\v^\ell\big)-\J\big(\v^0\big),\yscr^\ell\right>| & \leq  \frac{|\alpha|}{2} \z^{-1}(t,\omega_\ell) \int_{\O} |\Arm(\yscr^\ell)|^2( |\Arm(\v^\ell)|+|\Arm(\v^0)|) \d x
     \nonumber\\ & \leq \frac{\alpha^2}{4\nu(1-\varepsilon_0)} \z^{-2}(t,\omega_\ell) \int_{\O} |\Arm(\yscr^\ell)|^2 ( |\Arm(\v^\ell)|^2+|\Arm(\v^0)|^2) \d x  
     \nonumber\\ & \quad + \frac{\nu(1-\varepsilon_0)}{2}\|\Arm(\yscr^\ell)\|_{2}^2
     \nonumber\\ & = \frac{\beta(1-\varepsilon_0)}{2} \z^{-2}(t,\omega_\ell) \int_{\O} |\Arm(\yscr^\ell)|^2( |\Arm(\v^\ell)|^2+|\Arm(\v^0)|^2) \d x 
     \nonumber\\ & \quad + \frac{\nu(1-\varepsilon_0)}{2}\|\Arm(\yscr^\ell)\|_{2}^2 , \label{LC5} \\
      |\alpha[\z^{-1}(t,\omega_0)-\z^{-1}(t,\omega_\ell)]\left<\J\big(\v^0\big) ,\yscr^\ell \right> | & \leq \frac{\alpha}{2}|\z^{-1}(t,\omega_0)-\z^{-1}(t,\omega_\ell)| \|\Arm(\v^0)\|^2_4 \|\Arm(\yscr^\ell)\|_2
      \nonumber\\ & \leq \frac{\nu\varepsilon_0}{16}\|\Arm(\yscr^\ell)\|^2_2  + C [\z^{-1}(t,\omega_0)-\z^{-1}(t,\omega_\ell)]^2 \|\Arm(\v^0)\|^4_4,  \label{LC6}  \\
	\beta\z^{-2}(t,\omega_\ell)\left<\K\big(\v^\ell\big)-\K\big(\v^0\big) ,\yscr^\ell\right>  & \geq  \frac{\beta}{2} \z^{-2}(t,\omega_\ell) \int_{\O} |\Arm(\yscr^\ell)|^2( |\Arm(\v^\ell)|^2+|\Arm(\v^0)|^2) \d x, \label{LC7} \\
	|\beta[\z^{-2}(t,\omega_\ell)-\z^{-2}(t,\omega_0)]\left<\K\big(\v^0\big) ,\yscr^\ell\right>|  & \leq  \beta|\z^{-2}(t,\omega_\ell)-\z^{-2}(t,\omega_0)|\|\Arm(\v^0)\|^3_4 \|\Arm(\yscr^\ell)\|_4
	\nonumber\\ & \leq  C |1-\z^{2}(t,\omega_\ell)\z^{-2}(t,\omega_0)|^{\frac{4}{3}}\|\Arm(\v^0)\|^4_4 
	  \nonumber\\ & \quad + \frac{\beta\varepsilon_0}{8}\z^{-2}(t,\omega_\ell)\|\Arm(\yscr^\ell)\|^4_4 
	  \nonumber\\ & \leq  C |1-\z^{2}(t,\omega_\ell)\z^{-2}(t,\omega_0)|^{\frac{4}{3}}\|\Arm(\v^0)\|^4_4 
	  \nonumber\\ & \quad + \frac{\beta\varepsilon_0}{4} \z^{-2}(t,\omega_\ell) \int_{\O} |\Arm(\yscr^\ell)|^2( |\Arm(\v^\ell)|^2+|\Arm(\v^0)|^2) \d x, \label{LC8}   \\  
	|\sigma \left[y(\vartheta_{t}\omega_\ell)-y(\vartheta_{t}\omega_0)\right]\left(\v^0, \yscr^\ell\right)| & \leq  C |y(\vartheta_{t}\omega_\ell)-y(\vartheta_{t}\omega_0)|^2\|\v^0\|_2^2+ \frac{\sigma^2}{8}\| \yscr^\ell\|^2_2, \label{LC9} \\
	 |\left[\z(t,\omega_\ell)-\z(t,\omega_0)\right] \left<\f,\yscr^\ell\right>| & \leq |\z(t,\omega_\ell)-\z(t,\omega_0)| \|\f\|_{\H^{-1}}  \|\yscr^\ell\|_{\V} 
	 \nonumber\\ & \leq  \min\left\{\frac{\nu\varepsilon_0}{8}, \frac{\sigma^2}{8}\right\} \|\yfrak^k\|^2_{\V} + C|\z(t,\omega_\ell)-\z(t,\omega_0)|^2 \|\f\|^2_{\H^{-1}} 
	 \nonumber\\ & \leq \frac{\nu\varepsilon_0}{8} \|\nabla\yscr^\ell\|^2_{2} + \frac{\sigma^2}{8} \|\yscr^\ell\|^2_{2} + C |\z(t,\omega_\ell)-\z(t,\omega_0)|^2 \|\f\|^2_{\H^{-1}}
	 \nonumber\\ & = \frac{\nu\varepsilon_0}{16} \|\Arm(\yscr^\ell)\|^2_{2} + \frac{\sigma^2}{8} \|\yscr^\ell\|^2_{2} + C |\z(t,\omega_\ell)-\z(t,\omega_0)|^2 \|\f\|^2_{\H^{-1}}.\label{LC10}
\end{align}
A combination of \eqref{LC2}-\eqref{LC10} infers
\begin{align}\label{LC11}
	  \frac{\d }{\d t} \|\yscr^\ell\|^2_2 
	  & \leq  2\sigma |y(\vartheta_{t}\omega_\ell)|\|\yscr^\ell\|^2_2 + C \z^{-2}(t,\omega_\ell) \|\Arm(\v^0)\|^2_{4}\|\yscr^\ell\|_2^2 +  C |\z^{-1}(t,\omega_\ell) - \z^{-1}(t,\omega_0)|^2 \|\Arm(\v^0)\|^2_{4}\|\v^0\|_2^2
	\nonumber\\ & \quad + C |\z^{-1}(t,\omega_0)-\z^{-1}(t,\omega_\ell)|^2 \|\Arm(\v^0)\|^4_4 + C |1-\z^{2}(t,\omega_\ell)\z^{-2}(t,\omega_0)|^{\frac{4}{3}}\|\Arm(\v^0)\|^4_4 
	\nonumber\\ & \quad + C |y(\vartheta_{t}\omega_\ell)-y(\vartheta_{t}\omega_0)|^2\|\v^0\|_2^2 + C |\z(t,\omega_\ell)-\z(t,\omega_0)|^2 \|\f\|^2_{\H^{-1}}.
\end{align}
In view of Gronwall's inequality, we  deduce from \eqref{LC11} 
\begin{align}\label{LC12}
	&\|\yscr^\ell(t)\|^2_2 
	\nonumber\\ & \leq C \int_{\mathfrak{r}}^{\mathfrak{r}+T} \bigg[ |\z^{-1}(t,\omega_\ell) - \z^{-1}(t,\omega_0)|^2 \|\Arm(\v^0(t))\|^2_{4}\|\v^0(t)\|_2^2
	 +  |\z^{-1}(t,\omega_0)-\z^{-1}(t,\omega_\ell)|^2 \|\Arm(\v^0(t))\|^4_4 
	 \nonumber\\ & \quad +  |1-\z^{2}(t,\omega_\ell)\z^{-2}(t,\omega_0)|^{\frac{4}{3}}\|\Arm(\v^0(t))\|^4_4 
  +  |y(\vartheta_{t}\omega_\ell)-y(\vartheta_{t}\omega_0)|^2\|\v^0(t)\|_2^2 
  \nonumber\\ & \quad +  |\z(t,\omega_\ell)-\z(t,\omega_0)|^2 \|\f(t)\|^2_{\H^{-1}}\bigg] \d t \times \exp\left\{\int_{\mathfrak{r}}^{\mathfrak{r}+T}\left[2\sigma |y(\vartheta_{t}\omega_\ell)| + C \z^{-2}(t,\omega_\ell) \|\Arm(\v^0(t))\|^2_{4}\right]\d t\right\},
\end{align}
for any $T>0$. Now, making use of Lemma \ref{conv_z} and the fact that $\v^0\in  \mathrm{C}([\mathfrak{r},\mathfrak{r}+T];\H)\cap\mathrm{L}^2(\mathfrak{r},\mathfrak{r}+T;\V)\cap\mathrm{L}^4(\mathfrak{r},\mathfrak{r}+T;\mathbb{W}^{1,4}(\O))$ in \eqref{LC12}, we  conclude the proof.
\end{proof}

The following lemma plays a crucial role in  proving the continuity  of the map $\Phi(t,\mathfrak{r},\omega,\cdot):\H\to\H$, for every fixed $t>0$, $\mathfrak{r}\in\R$ and $\omega\in\Omega$.

\begin{lemma}\label{Continuity}
	Assume that $\f\in \mathrm{L}^2_{\mathrm{loc}}(\mathbb{R};\H^{-1}(\O))$. Then,  the solution of \eqref{CTGF} is continuous with respect to initial data $\v_{\mathfrak{r}}.$
\end{lemma}
\begin{proof}
	Let $\wi\v(\cdot):=\wi\v(\cdot;\mathfrak{r},\omega,\wi\v_{\mathfrak{r}})$ and $\v(\cdot):=\v(\cdot;\mathfrak{r},\omega,\v_{\mathfrak{r}})$ be two solutions of \eqref{CTGF}. Then $\yfrak(\cdot)=\wi\v(\cdot)-\v(\cdot)$ with $\yfrak(\mathfrak{r})=\wi\v_{\mathfrak{r}}-\v_{\mathfrak{r}}$ satisfies
	\begin{align}\label{Conti1}
		&	\frac{\d\yfrak}{\d t}+\nu \A\yfrak  +\left[\frac{\sigma^2}{2}-\sigma y(\vartheta_{t}\omega)\right]\yfrak \nonumber\\&=-\z^{-1}(t,\omega)\left\{\B\big(\wi\v\big)-\B\big(\v\big)\right\} - \alpha \z^{-1}(t,\omega)\left\{\J\big(\wi\v\big)-\J\big(\v\big)\right\} -\z^{-2}(t,\omega)\beta\left\{\K\big(\wi\v\big)-\K\big(\v\big)\right\}
		\nonumber\\ & = -\z^{-1}(t,\omega)\left\{\B\big(\yfrak,\wi\v\big)-\B\big(\v, \yfrak\big)\right\} - \alpha \z^{-1}(t,\omega)\left\{\J\big(\wi\v\big)-\J\big(\v\big)\right\}
		  -\z^{-2}(t,\omega)\beta\left\{\K\big(\wi\v\big)-\K\big(\v\big)\right\}.
	\end{align}
	in $\X'$.  Taking the inner product to equation \eqref{Conti1} by $\yfrak(\cdot)$ and using \eqref{b0}, we find
	\begin{align}\label{Conti2}
		&	\frac{1}{2}\frac{\d}{\d t}\|\yfrak\|_2^2 + \nu \|\nabla\yfrak\|_2^2  +\frac{\sigma^2}{2}\|\yfrak\|_2^2 
		\nonumber\\ & = \sigma y(\vartheta_{t}\omega) \|\yfrak\|_2^2  + \z^{-1}(t,\omega) b\big(\yfrak, \yfrak ,\wi\v\big)  - \alpha \z^{-1}(t,\omega)\left\{\J\big(\wi\v\big)-\J\big(\v\big)\right\}
		   -\z^{-2}(t,\omega)\beta\left\{\K\big(\wi\v\big)-\K\big(\v\big)\right\}.
	\end{align}
	 An application of \eqref{Sobolev+Korn}, and H\"older's and Young's inequalities give
	\begin{align}\label{Conti3}
		|\z^{-1}(t,\omega) b\big(\yfrak , \yfrak ,\wi\v\big)| & \leq \z^{-1}(t,\omega) \|\yfrak\|_2 \|\nabla\yfrak\|_2 \|\wi\v\|_{\infty} 
		\nonumber\\ & \leq \M_d \z^{-1}(t,\omega) \|\nabla\yfrak\|_2 \|\yfrak\|_2 \|\Arm(\wi\v)\|_{4} 
		\nonumber\\ & \leq  \frac{\nu\varepsilon_0}{2}\|\nabla\yfrak\|_2^2 + \frac{(\M_d)^2}{2\nu\varepsilon_0} \z^{-2}(t,\omega)   \|\Arm(\wi\v)\|^2_{4} \|\yfrak\|^2_2.
	\end{align}

	For $\varepsilon_0:=1-\sqrt{\frac{\alpha^2}{2\beta\nu}}\in(0,1)$, from \eqref{J1-J2} and H\"older's inequality, we get
	\begin{align}\label{Conti4}
		& |\alpha\z^{-1}(t,\omega)\left<	\J(\wi\v) - \J(\v), \yfrak \right>| 
		  \nonumber\\ & \leq \frac{\nu(1-\varepsilon_0)}{2}\|\Arm(\yfrak)\|_{2}^2 + \frac{\alpha^2}{4\nu(1-\varepsilon_0)} \z^{-2}(t,\omega) \int_{\O} |\Arm(\yfrak)|^2 ( |\Arm(\wi\v)|^2+|\Arm(\v)|^2) \d x
		\nonumber\\ & = \nu(1-\varepsilon_0)\|\nabla\yfrak\|_{2}^2 + \frac{\beta(1-\varepsilon_0)}{2} \z^{-2}(t,\omega) \int_{\O} |\Arm(\yfrak)|^2( |\Arm(\wi\v)|^2+|\Arm(\v)|^2) \d x,
	\end{align}

From \eqref{K1-K2}, we have 
\begin{align}\label{Conti5}
	& \z^{-2}(t,\omega)\beta\left<	\K(\wi\v) - \K(\v), \yfrak \right>
	\nonumber\\ & = \frac{\beta}{2}\z^{-2}(t,\omega)\int_{\O}( |\Arm(\wi\v)|^2-|\Arm(\v)|^2)^2 \d x + \frac{\beta}{2}\z^{-2}(t,\omega)\int_{\O} |\Arm(\yfrak)|^2( |\Arm(\wi\v)|^2+|\Arm(\v)|^2) \d x.
\end{align}
Using \eqref{Conti3}-\eqref{Conti5} in \eqref{Conti2}, we reach at
\begin{align}\label{Conti6}
	&	\frac{1}{2}\frac{\d}{\d t}\|\yfrak\|_2^2 + \frac{\nu\varepsilon_0}{2} \|\nabla\yfrak\|_2^2  +\frac{\sigma^2}{2}\|\yfrak\|_2^2  + \frac{\beta\varepsilon_0}{2}\z^{-2}(t,\omega)\int_{\O} |\Arm(\yfrak)|^2( |\Arm(\wi\v)|^2+|\Arm(\v)|^2) \d x
	\nonumber\\ & \leq  \sigma |y(\vartheta_{t}\omega)| \|\yfrak\|_2^2  + \frac{(\M_d)^2}{2\nu\varepsilon_0} \z^{-2}(t,\omega)   \|\Arm(\wi\v)\|^2_{4} \|\yfrak\|^2_2.
\end{align}
 Now, making use of Gronwall inequality in \eqref{Conti6}, we conclude
\begin{align*}
	\|\yfrak(t)\|^2_2 \leq \|\yfrak(\mathfrak{r})\|^2_2 \cdot \exp \left\{\int_{\mathfrak{r}}^{t}\left[2\sigma |y(\vartheta_{s}\omega)|  + \frac{(\M_d)^2}{\nu\varepsilon_0} \z^{-2}(s,\omega)   \|\Arm(\wi\v(s))\|^2_{4}\right]\d s\right\},
\end{align*}
for all $t\geq \mathfrak{r}$. This completes the proof.
\end{proof}

Using the Lusin continuity established in Lemma \ref{LusinC}, we obtain that $\Phi$ is $\mathscr{F}$-measurable. Hence, combining Lemma \ref{lem-Sol-CTGF} with Lemma \ref{LusinC} and Lemma \ref{Continuity} shows that the mapping $\Phi$ defined in \eqref{Phi} constitutes a non-autonomous random dynamical system on $\H$ (cf. \cite{PeriodicWang}).

\subsection{A universe of tempered  random sets} Assume that $D=\{D(\mathfrak{r},\omega):\mathfrak{r}\in\R,\omega\in\Omega\}$ is a family of non-empty subsets of $\H$ satisfying, for every $c>0, \mathfrak{r}\in\R$ and $\omega\in\Omega$, 
\begin{align}\label{D_1}
	\lim_{t\to\infty}e^{-ct}\|D(\mathfrak{r}-t,\vartheta_{-t}\omega)\|^2_{\H}=0.
\end{align}
Let us define the \emph{universe} of tempered subsets of $\H$ as $$\mathfrak{D}:=\{D=\{D(\mathfrak{r},\omega):\mathfrak{r}\in\R\text{ and }\omega\in\Omega\}:D \text{ satisfying } \eqref{D_1}\}.$$
It is immediate that $\mathfrak{D}$ is neighborhood closed, \cite[Definition 2.2]{PeriodicWang}.

	\section{Random absorbing set associated with $\Phi$}\label{Sec3}\setcounter{equation}{0}

	In this section, we show the existence of $\mathfrak{D}$-pullback random absorbing set associated with non-autonomous random dynamical system $\Phi$.	The following assumption on the external forcing term $\f$ is needed to prove the results of this section. 
	
	\begin{hypothesis}\label{Hyp-f-B}
		For the external forcing term $\f\in\mathrm{L}^{2}_{\emph{loc}}(\R;\H^{-1}(\O))$, there exists a number $\delta\in[0,\frac{\sigma^2}{2})$ such that for every $c>0$,
		\begin{align}\label{forcing2-B}
			\lim_{s\to-\infty}e^{cs}\int_{-\infty}^{0} e^{\delta\zeta}\|\f(\cdot,\zeta+s)\|^2_{\H^{-1}}\d \zeta=0.
		\end{align}
	\end{hypothesis}
	A direct consequence of  Hypothesis \ref{Hyp-f-B} is as follows:
	\begin{proposition}[Proposition 4.2, \cite{Kinra+Mohan_2023_DIE}]\label{Hypo_Conseq-B}
		Assume that Hypothesis \ref{Hyp-f-B} holds. Then
		\begin{align}\label{forcing1-B}
			\int_{-\infty}^{\mathfrak{r}} e^{\delta\zeta}\|\f(\cdot,\zeta)\|^2_{\H^{-1}}\d \zeta<\infty, \ \ \text{ for all } \mathfrak{r}\in\R,
		\end{align}
		where $\delta$ is the same as in \eqref{forcing2-B}.
	\end{proposition}
	\begin{example}
		Take $\f(\cdot,t)=t^{p}\f_1$, for any $p\geq0$ and $\f_1\in\H^{-1}(\O)$. Note that the conditions \eqref{forcing2-B}-\eqref{forcing1-B} do not need $\f$ to be bounded in $\H^{-1}(\O)$ at $\pm\infty$.
	\end{example}
	
	\begin{remark}
		We note that the assumption imposed on the deterministic forcing $\f$ in Hypothesis~\ref{Hyp-f-B} is sufficient for establishing the existence of a $\mathfrak{D}$-pullback random absorbing set on both bounded and unbounded domains. However, in order to prove the $\mathfrak{D}$-pullback asymptotic compactness on unbounded domains, this assumption needs to be refined. The corresponding strengthened condition is stated in Hypothesis~\ref{Hyp-f-U} below.
	\end{remark}
	
	\begin{lemma}\label{LemmaUe}
		Assume that condition \eqref{third-grade-paremeters-res} and Hypothesis \ref{Hyp-f-B} are satisfied. Then, for every $\mathfrak{r}\in \R$, $\omega\in \Omega$ and $B=\{B(\mathfrak{r},\omega):\mathfrak{r}\in\R,\; \omega\in \Omega\}\in \mathfrak{D}$, there exists $\mathcal{T}=\mathcal{T}(\mathfrak{r},\omega,B)\geq 1$ such that for all $t\geq \mathcal{T}$, $s \in [\mathfrak{r}-1,\mathfrak{r}]$ and $\v_{\mathfrak{r}-t}\in B(\mathfrak{r}-t,\vartheta_{-t}\omega)$, the solution $\v(\cdot)$ of system \eqref{CTGF} with $\omega$ replaced by $\vartheta_{-\mathfrak{r}}\omega$ satisfies
		\begin{align}\label{UE}
			&\|\v(s;\mathfrak{r}-t,\vartheta_{-\mathfrak{r}}\omega,\v_{\mathfrak{r}-t})\|^2_{2} + \frac{\nu\varepsilon_0}{2}\int_{\mathfrak{r}-t}^{s}e^{-\int^{s}_{\zeta}\left(\frac{\sigma^2}{2}-2\sigma y(\vartheta_{\upeta-\mathfrak{r}}\omega)\right)\d \upeta}\|\Arm(\v(\zeta;\mathfrak{r}-t,\vartheta_{-\mathfrak{r}}\omega,\v_{\mathfrak{r}-t}))\|^2_{2}\d\zeta
			\nonumber\\ &	+ \beta\varepsilon_0 \int_{\mathfrak{r}-t}^{s}e^{-\int^{s}_{\zeta}\left(\frac{\sigma^2}{2}-2\sigma y(\vartheta_{\upeta-\mathfrak{r}}\omega)\right)\d \upeta}\z^{-2}(\zeta,\vartheta_{-\mathfrak{r}}\omega)\|\Arm(\v(\zeta;\mathfrak{r}-t,\vartheta_{-\mathfrak{r}}\omega,\v_{\mathfrak{r}-t}))\|^4_{4}\d\zeta
			\nonumber\\&\leq   \frac{4e^{\int_{s}^{\mathfrak{r}}\left(\frac{\sigma^2}{2}-2\sigma y(\vartheta_{\upeta-\mathfrak{r}}\omega)\right)\d \upeta} }{\min\left\{2\nu\varepsilon_0,\sigma^2\right\}} \int_{-\infty}^{0} e^{\frac{\sigma^2}{2}\zeta+2\sigma\int^{0}_{\zeta} y(\vartheta_{\upeta}\omega)\d \upeta} \z^2(\zeta,\omega) \|\f(\cdot,\zeta+\mathfrak{r})\|^2_{\H^{-1}}\d \zeta < + \infty.
		\end{align}
	\end{lemma}
	\begin{proof}
		From $\eqref{CTGF}_1$ and \eqref{b0}, we get 
		\begin{align}\label{UE1}
			& \frac{1}{2}\frac{\d}{\d t}\|\v\|_2^2 + \frac{\nu}{2}\|\Arm(\v)\|^2_2 + \left[\frac{\sigma^2}{2} - \sigma y (\vartheta_{t}\omega)\right]\|\v\|^2_2 + \z^{-2}(t,\omega)\frac{\beta}{2} \|\Arm(\v)\|^4_4 
			\nonumber\\ & = - \alpha \z^{-1}(t,\omega) \left< \J(\v),\v \right>  + \z(t,\omega)\left< \f , \v \right>.
		\end{align}
	Using H\"older's and Young's inequalities, we have 
	\begin{align}
		\left|\alpha \z^{-1}(t,\omega) \left< \J(\v),\v \right> \right| & = \left| \frac{\alpha}{2} \z^{-1}(t,\omega)\int_{\O} \Arm(\v(x))\Arm(\v(x)):\Arm(\v(x))\d x \right|
		\nonumber\\ & \leq \frac{|\alpha|}{2} \z^{-1}(t,\omega) \|\Arm(\v)\|^2_4\|\Arm(\v)\|_2
		\nonumber\\ & \leq \frac{\nu(1-\varepsilon_0)}{2}\|\Arm(\v)\|_{2}^2 + \frac{\alpha^2}{4\nu(1-\varepsilon_0)} \z^{-2}(t,\omega) \|\Arm(\v)\|_{4}^4
		\nonumber\\ & = \frac{\nu(1-\varepsilon_0)}{2}\|\Arm(\v)\|_{2}^2 + \frac{\beta(1-\varepsilon_0)}{2} \z^{-2}(t,\omega) \|\Arm(\v)\|_{4}^4, \label{UE2} \\
	|\z(t,\omega)\left< \f , \v \right>|	& \leq \z(t,\omega) \|\f\|_{\H^{-1}}  \|\v\|_{\V} 
	\nonumber\\ & \leq \frac{1}{2} \min\left\{\nu\varepsilon_0,\frac{\sigma^2}{2}\right\} \|\v\|^2_{\V} + \frac{1}{2\min\left\{\nu\varepsilon_0,\frac{\sigma^2}{2}\right\}}\z^2(t,\omega) \|\f\|^2_{\H^{-1}} 
	\nonumber\\ & \leq \frac{\nu\varepsilon_0}{2} \|\nabla\v\|^2_{2} + \frac{\sigma^2}{4} \|\v\|^2_{2} + \frac{1}{\min\left\{2\nu\varepsilon_0,\sigma^2\right\}}\z^2(t,\omega) \|\f\|^2_{\H^{-1}}
	\nonumber\\ & = \frac{\nu\varepsilon_0}{4} \|\Arm(\v)\|^2_{2} + \frac{\sigma^2}{4} \|\v\|^2_{2} + \frac{1}{\min\left\{2\nu\varepsilon_0,\sigma^2\right\}}\z^2(t,\omega) \|\f\|^2_{\H^{-1}},\label{UE3}
	\end{align}
where $\varepsilon_0:=1-\sqrt{\frac{\alpha^2}{2\beta\nu}}\in(0,1)$. A combination of \eqref{UE1}-\eqref{UE3} gives
	\begin{align}\label{UE4}
	& \frac{\d}{\d t}\|\v\|_2^2 + \frac{\nu\varepsilon_0}{2}\|\Arm(\v)\|^2_2 + \left[\frac{\sigma^2}{2} - 2\sigma y (\vartheta_{t}\omega)\right]\|\v\|^2_2 + \beta\varepsilon_0  \z^{-2}(t,\omega)\|\Arm(\v)\|^4_4  \leq \frac{2\z^2(t,\omega)}{\min\left\{2\nu\varepsilon_0,\sigma^2\right\}} \|\f\|^2_{\H^{-1}}.
\end{align}
Applying the variation of constant formula to \eqref{UE4} and replacing $\omega$ by $\vartheta_{-\mathfrak{r}}\omega$ in the above inequality, we obtain
\begin{align}\label{UE5}
	&\|\v(s;\mathfrak{r}-t,\vartheta_{-\mathfrak{r}}\omega,\v_{\mathfrak{r}-t})\|^2_{2} + \frac{\nu\varepsilon_0}{2}\int_{\mathfrak{r}-t}^{s}e^{-\int^{s}_{\zeta}\left(\frac{\sigma^2}{2}-2\sigma y(\vartheta_{\upeta-\mathfrak{r}}\omega)\right)\d \upeta}\|\Arm(\v(\zeta;\mathfrak{r}-t,\vartheta_{-\mathfrak{r}}\omega,\v_{\mathfrak{r}-t}))\|^2_{2}\d\zeta
\nonumber\\ &	+ \beta\varepsilon_0\int_{\mathfrak{r}-t}^{s}e^{-\int^{s}_{\zeta}\left(\frac{\sigma^2}{2}-2\sigma y(\vartheta_{\upeta-\mathfrak{r}}\omega)\right)\d \upeta}\z^{-2}(\zeta,\vartheta_{-\mathfrak{r}}\omega)\|\Arm(\v(\zeta;\mathfrak{r}-t,\vartheta_{-\mathfrak{r}}\omega,\v_{\mathfrak{r}-t}))\|^4_{4}\d\zeta
	\nonumber\\&\leq e^{-\int^{s}_{\mathfrak{r}-t}\left(\frac{\sigma^2}{2}-2\sigma y(\vartheta_{\upeta-\mathfrak{r}}\omega)\right)\d \upeta}\|\v_{\mathfrak{r}-t}\|^2_{2}
	\nonumber\\ & \quad + \frac{2}{\min\left\{2\nu\varepsilon_0,\sigma^2\right\}}\int_{\mathfrak{r}-t}^{s} e^{-\int^{s}_{\zeta}\left(\frac{\sigma^2}{2}-2\sigma y(\vartheta_{\upeta-\mathfrak{r}}\omega)\right)\d \upeta} \z^2(\zeta,\vartheta_{-\mathfrak{r}}\omega) \|\f(\cdot,\zeta)\|^2_{\H^{-1}}\d \zeta
	\nonumber\\&\leq e^{\int_{s}^{\mathfrak{r}}\left(\frac{\sigma^2}{2}-2\sigma y(\vartheta_{\upeta-\mathfrak{r}}\omega)\right)\d \upeta}\bigg[ e^{ -\frac{\sigma^2}{2}t+2\sigma\int^{0}_{-t} y(\vartheta_{\upeta}\omega)\d \upeta}\|\v_{\mathfrak{r}-t}\|^2_{2}
	\nonumber\\ & \quad + \frac{2}{\min\left\{2\nu\varepsilon_0,\sigma^2\right\}} \int_{-\infty}^{0} e^{\frac{\sigma^2}{2}\zeta+2\sigma\int^{0}_{\zeta} y(\vartheta_{\upeta}\omega)\d \upeta} \z^2(\zeta,\omega) \|\f(\cdot,\zeta+\mathfrak{r})\|^2_{\H^{-1}}\d \zeta\bigg].
\end{align}
First, we show that the final term of \eqref{UE5} is finite. By \eqref{Z3}, we have that there exist $R_1, R_2<0$ such that for all $\zeta\leq R_1$,
\begin{align*}
	-2\sigma y(\vartheta_{\zeta}\omega)\leq-\frac12\left(\frac{\sigma^2}{2}- \delta\right)\zeta,
\end{align*}
and for all $\zeta\leq R_2$,
\begin{align*}
	2\sigma\int^{0}_{\zeta} y(\vartheta_{\upeta}\omega)\d \upeta\leq - \frac12 \left(\frac{\sigma^2}{2}- \delta\right)\zeta,
\end{align*}
where $\delta$ is the positive constant appearing in \eqref{forcing1-B}. Therefore, for all $\zeta\leq R_1$
\begin{align*}
	[\z(\zeta,{\omega})]^2=e^{-2\sigma y(\vartheta_{\zeta}\omega)}\leq e^{-\frac12\left(\frac{\sigma^2}{2}- \delta \right)\zeta},
\end{align*}
and we have for all $\zeta\leq R=:\min\{R_1,R_2\}$,
\begin{align*}
	&e^{\frac{\sigma^2}{2}\zeta+2\sigma\int^{0}_{\zeta} y(\vartheta_{\upeta}\omega)\d \upeta} \z^2(\zeta,\omega) \|\f(\cdot,\zeta+\mathfrak{r})\|^2_{\H^{-1}}\leq e^{\delta\zeta} \|\f(\cdot,\zeta+\mathfrak{r})\|^2_{\H^{-1}}.
\end{align*}
Therefore, it follows from \eqref{forcing1-B} that for every $\mathfrak{r}\in\R$ and $\omega\in\Omega$,
\begin{align}\label{UE7}
	& \int_{-\infty}^{0} e^{\frac{\sigma^2}{2}\zeta+2\sigma\int^{0}_{\zeta} y(\vartheta_{\upeta}\omega)\d \upeta} \z^2(\zeta,\omega) \|\f(\cdot,\zeta+\mathfrak{r})\|^2_{\H^{-1}}\d \zeta
	\nonumber\\ 
	& =
	\int_{R}^{0} e^{\frac{\sigma^2}{2}\zeta+2\sigma\int^{0}_{\zeta} y(\vartheta_{\upeta}\omega)\d \upeta} \z^2(\zeta,\omega) \|\f(\cdot,\zeta+\mathfrak{r})\|^2_{\H^{-1}}\d \zeta 
	\nonumber\\ & \quad + \int_{-\infty}^{R} e^{\frac{\sigma^2}{2}\zeta+2\sigma\int^{0}_{\zeta} y(\vartheta_{\upeta}\omega)\d \upeta} \z^2(\zeta,\omega) \|\f(\cdot,\zeta+\mathfrak{r})\|^2_{\H^{-1}}\d \zeta
	\nonumber\\ 
	& \leq \sup_{\zeta\in [R,0]}\left[e^{\frac{\sigma^2}{2}\zeta+2\sigma\int^{0}_{\zeta} y(\vartheta_{\upeta}\omega)\d \upeta} \z^2(\zeta,\omega)\right]
	\int_{R}^{0}  \|\f(\cdot,\zeta+\mathfrak{r})\|^2_{\H^{-1}}\d \zeta 
	\nonumber\\ & \quad + \int_{-\infty}^{R} e^{\delta\zeta} \z^2(\zeta,\omega) \|\f(\cdot,\zeta+\mathfrak{r})\|^2_{\H^{-1}}\d \zeta<\infty.
\end{align}

Since $\v_{\mathfrak{r}-t}\in B(\mathfrak{r}-t,\vartheta_{-t}\omega)$, we have from \eqref{Z3} that for sufficiently large $t>0$
%\label{initial_data}
\begin{align*}
	e^{-\frac{\sigma^2}{2}t+2\sigma\int^{0}_{-t} y(\vartheta_{\upeta}\omega)\d \upeta}\|\v_{\mathfrak{r}-t}\|^2_{2}\leq e^{-\frac{\sigma^2}{4} t}\|B(\mathfrak{r}-t,\vartheta_{-t}\omega)\|^2_{2}\to 0\ \text{ as } \ t \to \infty.
\end{align*}

Therefore, there exists $\mathcal{T}=\mathcal{T}(\mathfrak{r},\omega,B)>0$ such that 
\begin{align}\label{UE6}
	e^{-\frac{\sigma^2}{2}t+2\sigma\int^{0}_{-t} y(\vartheta_{\upeta}\omega)\d \upeta}\|\v_{\mathfrak{r}-t}\|^2_{2}\leq \frac{2}{\min\left\{2\nu\varepsilon_0,\sigma^2\right\}} \int_{-\infty}^{0} e^{\frac{\sigma^2}{2}\zeta+2\sigma\int^{0}_{\zeta} y(\vartheta_{\upeta}\omega)\d \upeta} \z^2(\zeta,\omega) \|\f(\cdot,\zeta+\mathfrak{r})\|^2_{\H^{-1}}\d \zeta,
\end{align}
for all  $t\geq \mathcal{T}$.  Hence, from \eqref{UE5}-\eqref{UE6}, \eqref{UE} follows. The proof is completed.
	\end{proof}
	Based on the previous lemma, we prove the existence of a $\mathfrak{D}$-pullback random absorbing set in the following result. 
	
	\begin{lemma}\label{PAS}
		Suppose that condition \eqref{third-grade-paremeters-res} and Hypothesis \ref{Hyp-f-B} are satisfied. Then the non-autonomous random dynamical system $\Phi$ associated with the system \eqref{STGF} possesses a $\mathfrak{D}$-pullback random absorbing set $\mathcal{M}=\{\mathcal{M}(\mathfrak{r},\omega):\mathfrak{r}\in\R,\omega\in\Omega\}\in\mathfrak{D}$, where  $\mathcal{M}(\mathfrak{r},\omega)$ is defined by 
		\begin{align}\label{PAB1}
			\mathcal{M}(\mathfrak{r},\omega)=\{\u\in\H:\|\u\|^2_{\H}\leq\mathcal{K}(\mathfrak{r},\omega)\},
		\end{align}
		and $\mathcal{K}(\mathfrak{r},\omega)$ is given by
		\begin{align*}
			\mathcal{K}(\mathfrak{r},\omega)=  \frac{4e^{2y(\omega)} }{\min\left\{2\nu\varepsilon_0,\sigma^2\right\}} \int_{-\infty}^{0} e^{\frac{\sigma^2}{2}\zeta+2\sigma\int^{0}_{\zeta} y(\vartheta_{\upeta}\omega)\d \upeta} \z^2(\zeta,\omega) \|\f(\cdot,\zeta+\mathfrak{r})\|^2_{\H^{-1}}\d \zeta.
		\end{align*}
	\end{lemma}
	\begin{proof}
		We infer from \eqref{UE7} that
		\begin{align}\label{IRAS2-N}
			\mathcal{K}(\mathfrak{r},\omega)&= \frac{4e^{2y(\omega)} }{\min\left\{2\nu\varepsilon_0,\sigma^2\right\}} \int_{-\infty}^{0} e^{\frac{\sigma^2}{2}\zeta+2\sigma\int^{0}_{\zeta} y(\vartheta_{\upeta}\omega)\d \upeta} \z^2(\zeta,\omega) \|\f(\cdot,\zeta+\mathfrak{r})\|^2_{\H^{-1}}\d \zeta < +\infty.
		\end{align}
		Hence, the absorption follows from Lemma \ref{LemmaUe}. For $c>0$, let $c_1=\min\{\frac{c}{2},\frac{\sigma^2}{4},\delta\}$ and consider
		\begin{align}\label{IRAS3-N}
			&\lim_{t\to+\infty}e^{-ct}\|\mathcal{M}(\mathfrak{r}-t,\vartheta_{-t}\omega)\|^2_{\H}\nonumber\\&\leq \lim_{t\to+\infty}e^{-ct}\left[\frac{4e^{2y(\vartheta_{-t}\omega)}}{\min\left\{2\nu\varepsilon_0,\sigma^2\right\}}\int_{-\infty}^{0} e^{\frac{\sigma^2}{2}\zeta+2\sigma\int^{0}_{\zeta} y(\vartheta_{\upeta-t}\omega)\d \upeta} \z^2(\zeta-t,\omega) \|\f(\cdot,\zeta+\mathfrak{r}-t)\|^2_{\H^{-1}}\d \zeta\right]\nonumber\\&=\lim_{t\to+\infty}e^{-ct}\left[\frac{4e^{2y(\vartheta_{-t}\omega)+\int_{0}^{-t}y(\vartheta_{\upeta}\omega)\d\upeta}}{\min\left\{2\nu\varepsilon_0,\sigma^2\right\}} \int_{-\infty}^{-t} e^{\frac{\sigma^2}{2}(\zeta+t)-2\sigma\int^{0}_{\zeta} y(\vartheta_{\upeta-t}\omega)\d \upeta} \z^2(\zeta,\omega) \|\f(\cdot,\zeta+\mathfrak{r})\|^2_{\H^{-1}}\d \zeta\right] \nonumber\\&\leq\lim_{t\to+\infty}e^{-\left(c-\frac{c_1}{2}\right)t}\left[\frac{4}{\min\left\{2\nu\varepsilon_0,\sigma^2\right\}}\int_{-\infty}^{-t} e^{\delta \zeta} \|\f(\cdot,\zeta+\mathfrak{r})\|^2_{\H^{-1}}\d \zeta\right] =0,
		\end{align}
		where we have used \eqref{Z3} and \eqref{forcing2-B}. It follows from \eqref{IRAS3-N} that $\mathcal{M}\in{\mathfrak{D}}$. This completes the proof.
	\end{proof}

	\section{Pullback random attractors associated with $\Phi$: Bounded domains}\label{Sec4}\setcounter{equation}{0}
	
	In this section, we establish first main result of this article, that is, the existence of $\mathfrak{D}$-pullback random attractor associated with $\Phi$ on bounded domains. We have already obtain the existence of  $\mathfrak{D}$-pullback random absorbing set associated with $\Phi$ in Section \ref{Sec3}. In order to apply the abstract result for the existence of $\mathfrak{D}$-pullback random attractor from \cite{SandN_Wang}, we are only remain to show the $\mathfrak{D}$-pullback asymptotic compactness for $\Phi$. We prove $\mathfrak{D}$-pullback asymptotic compactness of $\Phi$ using compact Sobolev embedding $\V\subset\H$ on bounded domains.
	
	\begin{lemma}\label{PCB}
		Let  the condition \eqref{third-grade-paremeters-res} and Hypothesis \ref{Hyp-f-B} be satisfied. Then for every $ \omega\in\Omega$, $\mathfrak{r}\in\R$ and $t>\mathfrak{r}$, the solution $\v(t;\mathfrak{r},\omega,\cdot):\H\to\H$ is compact, that is, for every bounded set $\mathcal{H}$ in $\H$, the image $\v(t;\mathfrak{r},\omega,B)$ is precompact in $\H$.
	\end{lemma}
	\begin{proof}
		Let $\v(s;\mathfrak{r},\omega,\cdot)$ be the unique solution of \eqref{CTGF} for $s\in[\mathfrak{r},\mathfrak{r}+T]$, where $T>0$. Assume that the sequence $\{\v^n_{0}\}_{n\in\N}\subset \H$. In view of Lemma \ref{LemmaUe}, we infer that 
		\begin{equation}\label{PCB1}
				\{\v(\cdot;\mathfrak{r},\omega,\v^n_{0})\}_{n\in\N}  \text{ is bounded in } \mathrm{L}^{\infty}(\mathfrak{r},\mathfrak{r}+T;\H)\cap\mathrm{L}^2(\mathfrak{r}, \mathfrak{r}+T;\V)\cap\mathrm{L}^{4}(\mathfrak{r},\mathfrak{r}+T;\mathbb{W}^{1,4}(\O)).
		\end{equation}
		We also have
		\begin{align}\label{PCB3}
			\{\A(\v(\cdot;\mathfrak{r},\omega,\v^n_{0}))\}_{n\in\N} \text{ and }   \{\B(\v(\cdot;\mathfrak{r},\omega,\v^n_{0}))\}_{n\in\N} \text{ are bounded in } \mathrm{L}^2(\mathfrak{r},\mathfrak{r}+T;\V'),
		\end{align}
		and
		\begin{align}\label{PCB4}
			\{\J(\v(\cdot;\mathfrak{r},\omega,\v^n_{0}))\}_{n\in\N} \text{ and } \{\K(\v(\cdot;\mathfrak{r},\omega,\v^n_{0}))\}_{n\in\N} \text{ is bounded in } \mathrm{L}^{\frac{4}{3}}(\mathfrak{r},\mathfrak{r}+T;\mathbb{W}^{-1,\frac{4}{3}}(\O)).
		\end{align}
		It follows from \eqref{PCB1}-\eqref{PCB4} and \eqref{CTGF} that
		\begin{align*}
			\left\{\frac{\d}{\d s}(\u(\cdot;\mathfrak{r},\omega,\v^n_{0}))\right\}_{n\in\N} \text{ is bounded in } \mathrm{L}^2(\mathfrak{r},\mathfrak{r}+T;\V')+\mathrm{L}^{\frac{4}{3}}(\mathfrak{r},\mathfrak{r}+T;\mathbb{W}^{-1,\frac{4}{3}}(\O)).
		\end{align*}
		Since $\mathrm{L}^2(\mathfrak{r},\mathfrak{r}+T;\V')+\mathrm{L}^{\frac{4}{3}}(\mathfrak{r},\mathfrak{r}+T;\mathbb{W}^{-1,\frac{4}{3}}(\O))\subset\mathrm{L}^{\frac{4}{3}}(\mathfrak{r},\mathfrak{r}+T;\V'+\mathbb{W}^{-1,\frac{4}{3}}(\O))$, the above sequence is bounded in $\mathrm{L}^{\frac{4}{3}}(\mathfrak{r},\mathfrak{r}+T;\V'+\mathbb{W}^{-1,\frac{4}{3}}(\O))$. Note also that $\V\cap\mathbb{W}^{1,4}(\O)\subset\V\subset\H\subset\V'\subset \V'+\mathbb{W}^{-1,\frac{4}{3}}(\O)$ and the embedding of $\V\subset\H$ is compact on bounded domains. By the \emph{Aubin-Lions compactness lemma}, there exists a subsequence (keeping as it is) and $\widehat{\v}\in\mathrm{L}^2(\mathfrak{r},\mathfrak{r}+T;\H)$ such that 
		\begin{align}\label{PCB6}
			\v(\cdot;\mathfrak{r},\omega,\v^n_{0})\to \widehat{\v}(\cdot) \ \text{ strongly in }\  \mathrm{L}^{2}(\mathfrak{r},\mathfrak{r}+T;\H).
		\end{align}
		Along a further subsequence (again not relabeling), we infer from \eqref{PCB6} that
		\begin{align}\label{PCB7}
			\v(s;\mathfrak{r},\omega,\v^n_{0})\to\widehat{\v}(s) \text{ in }  \H \ \text{ for almost all }\  s\in(\mathfrak{r},\mathfrak{r}+T).
		\end{align}
		Since $\mathfrak{r}<t<\mathfrak{r}+T$, we obtain from \eqref{PCB7} that there exists $s\in(\mathfrak{r},t)$ such that \eqref{PCB7} holds true for this particular $s$. Then  by Lemma \ref{Continuity}, we obtain
		\begin{align*}
			\v(t;\mathfrak{r},\omega,\v^n_{0})=\v(t;s,\omega,\u(s;\mathfrak{r},\omega,\v^n_{0}))\to \v(t;\mathfrak{r},\omega,\widehat{\v}(\mathfrak{r})),
		\end{align*}
		which completes the proof.
	\end{proof}
	In fact, Lemma \ref{PCB} helps us to prove the $\mathfrak{D}$-pullback asymptotic compactness of $\Phi$ in $\H$ on bounded domains.
	\begin{corollary}\label{Asymptotic_B}
		Let  the condition \eqref{third-grade-paremeters-res} and Hypothesis \ref{Hyp-f-B} be satisfied. Then for every $\mathfrak{r}\in \R$, $\omega\in \Omega,$ $B=\{B(\mathfrak{r},\omega):\mathfrak{r}\in \R,\omega\in \Omega\}\in \mathfrak{D}$ and $t_n\to \infty,$ $\v^n_{0}\in B(\mathfrak{r}-t_n, \vartheta_{-t_{n}}\omega)$, the sequence $\v(\mathfrak{r};\mathfrak{r}-t_n,\vartheta_{-\mathfrak{r}}\omega,\v^n_{0})$ of solutions of the system \eqref{CTGF} has a convergent subsequence in $\H$.
	\end{corollary}
	\begin{proof}
		From Lemma \ref{LemmaUe} with $s=\mathfrak{r}-1$, we have that there exists $\mathcal{T}=\mathcal{T}(\mathfrak{r},\omega,B)\geq 1$ such that for all $t\geq\mathcal{T}$ and $\v_{\mathfrak{r}-t}\in B(\mathfrak{r}-t, \vartheta_{-t}\omega)$,
		\begin{align}\label{AB1}
			\v(\mathfrak{r}-1;\mathfrak{r}-t,\vartheta_{-\mathfrak{r}}\omega,\v_{\mathfrak{r}-t})\in\H.
		\end{align}
		Since $t_n\to\infty$ and $\v^n_{0}\in B(\mathfrak{r}-t_n, \vartheta_{-t_{n}}\omega)$, from \eqref{AB1}, we infer that there exists $N_1=N_1(\mathfrak{r},\omega,B)>0$ such that 
		\begin{align}\label{AB2}
			\{\v(\mathfrak{r}-1;\mathfrak{r}-t_n,\vartheta_{-\mathfrak{r}}\omega,\v^n_{0})\}_{n\geq N_1}\subset\H.
		\end{align}
		Hence, by \eqref{AB2} and Lemma \ref{PCB}, we conclude that the sequence $$\v(\mathfrak{r};\mathfrak{r}-t_n,\vartheta_{-\mathfrak{r}}\omega,\v^n_{0})=\v(\mathfrak{r};\mathfrak{r}-1,\vartheta_{-\mathfrak{r}}\omega,\v(\mathfrak{r}-1;\mathfrak{r}-t_n,\vartheta_{-\mathfrak{r}}\omega,\v^n_{0}))$$ has a convergent subsequence in $\H$, which completes the proof.
	\end{proof}

\begin{theorem}\label{PRA_B}
	Let  the condition \eqref{third-grade-paremeters-res} and Hypothesis \ref{Hyp-f-B} be satisfied.Then, there exists a unique $\mathfrak{D}$-pullback random attractor $\mathscr{A}=\{\mathscr{A}(\mathfrak{r},\omega):\mathfrak{r}\in\R, \omega\in\Omega\}\in\mathfrak{D},$ for the  non-autonomous random dynamical system $\Phi$ associated with the system \eqref{STGF} in $\H$.
\end{theorem}
\begin{proof}
	The proof follows from Corollary \ref{Asymptotic_B} ($\mathfrak{D}$-pullback asymptotic compactness of $\Phi$), Lemma \ref{PAS} (existence of $\mathfrak{D}$-pullback random absorbing set) and the abstract theory given in \cite[Theorem 2.23]{SandN_Wang}.
\end{proof}

	\section{Pullback random attractors associated with $\Phi$: Unbounded domains}\label{Sec5}\setcounter{equation}{0}

In this section, we establish second main result of this article, that is, the existence of $\mathfrak{D}$-pullback random attractor associated with $\Phi$ on unbounded domains. Since the existence of $\mathfrak{D}$-pullback random absorbing set associated with $\Phi$ is known from Section \ref{Sec3},  we are only remain to show the $\mathfrak{D}$-pullback asymptotic compactness for $\Phi$ to obtain the existence of $\mathfrak{D}$-pullback random attractors for $\Phi$.  As discussed in Remark \ref{rem-forcing-hypo}, we require the following assumption on the external forcing term $\f$ to prove the result of this section. 

\begin{hypothesis}\label{Hyp-f-U}
	For the external forcing term $\f\in\mathrm{L}^{2}_{\emph{loc}}(\R;\L^2(\O))$, there exists a number $\delta\in[0,\frac{\sigma^2}{2})$ such that for every $c>0$,
	\begin{align}\label{forcing2-U}
		\lim_{s\to-\infty}e^{cs}\int_{-\infty}^{0} e^{\delta\zeta}\|\f(\cdot,\zeta+s)\|^2_{2}\d \zeta=0.
	\end{align}
\end{hypothesis}
A direct consequence of  Hypothesis \ref{Hyp-f-U} is as follows:
\begin{proposition}[Proposition 4.2, \cite{Kinra+Mohan_2023_DIE}]\label{Hypo_Conseq-U}
	Assume that Hypothesis \ref{Hyp-f-U} holds. Then
	\begin{align}\label{forcing1-U}
		\int_{-\infty}^{\mathfrak{r}} e^{\delta\zeta}\|\f(\cdot,\zeta)\|^2_{2}\d \zeta<\infty, \ \ \text{ for all } \mathfrak{r}\in\R,
	\end{align}
	where $\delta$ is the same as in \eqref{forcing2-U}. 
%	{\color{magenta} Moreover, \eqref{forcing1-U} implies that 
%	\begin{align}\label{forcing4}
%		\lim_{k\to\infty}\int_{-\infty}^{\mathfrak{r}}\int\limits_{\O\cap\{|x|\geq k\}} e^{\delta \zeta}|\f(x,\zeta)|^2\d x\d \zeta=0, \ \ \ \text{ for all }\ \mathfrak{r}\in\R.
%	\end{align}}
\end{proposition}
\begin{example}
	Take $\f(\cdot,t)=t^{p}\f_1$, for any $p\geq0$ and $\f_1\in\L^2(\O)$. Note that the conditions \eqref{forcing2-U}-\eqref{forcing1-U} do not need $\f$ to be bounded in $\L^2(\O)$ at $\pm\infty$.
\end{example}

 \begin{remark}
 	We note that Hypothesis~\ref{Hyp-f-U} is used for deriving the uniform tail estimates for solutions of \eqref{eqn-CTGF-Pressure} (see Lemma~\ref{LR} below), which in turn are crucial for establishing the $\mathfrak{D}$-pullback asymptotic compactness of $\Phi$ on unbounded domains. It is worth pointing out that Hypothesis~\ref{Hyp-f-U} alone would suffice if one could apply the energy equality approach introduced in \cite{Ball} to verify the $\mathfrak{D}$-pullback asymptotic compactness, as done in \cite{PeriodicWang} for the stochastic Navier-Stokes equations and in \cite{Kinra+Mohan_2023_DIE} for the stochastic convective Brinkman-Forchheimer equations. However, for the third-grade fluid equations \eqref{equationV1}, we are not able to use energy equality method (see Remarks \ref{rem-forcing-hypo} and \ref{rem-weak-conti}).
 \end{remark}

 The following result is used to obtain the $\mathfrak{D}$-pullback asymptotic compactness of $\Phi$.

 \begin{lemma}\label{D-convege}
 	Assume that $\f\in \mathrm{L}^{2}_{\mathrm{loc}}(\R;\H^{-1}(\O))$ and $\{\v^n_{0}\}_{n\in\N}$ be a bounded sequence in $\H$. Then, for every $\omega\in\Omega$ and $\mathfrak{r}\in\R$, there exists $\tilde{\v} \in \mathrm{L}^{\infty}(\mathfrak{r}, \mathfrak{r}+T; \H)\cap\mathrm{L}^{2}(\mathfrak{r}, \mathfrak{r}+T; \V)\cap\mathrm{L}^{4}(\mathfrak{r}, \mathfrak{r}+T; \mathbb{W}^{1,4}(\O))$  such that along a subsequence
 	\begin{align}\label{D-convege-4}
 		\v(\cdot;\mathfrak{r},\omega,\v^n_{0}) \ \to \tilde{\v} \ \text{ in	} \ \mathrm{L}^{2}(\mathfrak{r},\mathfrak{r}+T;\L^2_{\mathrm{loc}}(\O)),
 	\end{align}
 	for every $T>0$.
 \end{lemma}
	\begin{proof}
		The proof of this lemma closely follows those of \cite[Lemma 6.7]{KK+FC1} and \cite[Lemma 5.3]{KK+FC2}. 
	\end{proof}

Next, we aim to show that the solution to system \eqref{CTGF} satisfies the uniform tail estimates. For this purpose, we required the following result.

\begin{lemma}\label{LemmaUe-24}
	Assume that condition \eqref{third-grade-paremeters-res} and Hypothesis \ref{Hyp-f-U} are satisfied. Then, for every $\mathfrak{r}\in \R$, $\omega\in \Omega$ and $B=\{B(\mathfrak{r},\omega):\mathfrak{r}\in\R,\; \omega\in \Omega\}\in \mathfrak{D}$, there exists $\widehat{\mathcal{T}}=\widehat{\mathcal{T}}(\mathfrak{r},\omega,B)>0$ such that for all $t\geq \widehat{\mathcal{T}}$ and $\v_{\mathfrak{r}-t}\in B(\mathfrak{r}-t,\vartheta_{-t}\omega)$, the solution $\v(\cdot)$ of system \eqref{CTGF} with $\omega$ replaced by $\vartheta_{-\mathfrak{r}}\omega$ satisfies
	\begin{align}\label{UE-22}
		& \int_{\mathfrak{r}-t}^{\mathfrak{r}}e^{-\int^{\mathfrak{r}}_{\zeta}\left(\frac{\sigma^2}{2}-2\sigma y(\vartheta_{\upeta-\mathfrak{r}}\omega)\right)\d \upeta}\|\v(\zeta;\mathfrak{r}-t,\vartheta_{-\mathfrak{r}}\omega,\v_{\mathfrak{r}-t})\|^2_{2}\d\zeta
		  < + \infty,
	\end{align}
and 
\begin{align}\label{UE-24}
	& \int_{\mathfrak{r}-t}^{\mathfrak{r}}e^{-\int^{\mathfrak{r}}_{\zeta}\left(\frac{\sigma^2}{2}-2\sigma y(\vartheta_{\upeta-\mathfrak{r}}\omega)\right)\d \upeta} \z^{-2}(\zeta,\vartheta_{-\mathfrak{r}}\omega) \|\v(\zeta;\mathfrak{r}-t,\vartheta_{-\mathfrak{r}}\omega,\v_{\mathfrak{r}-t})\|^4_{2}\d\zeta
	  < + \infty.
\end{align}
\end{lemma}
\begin{proof}
	Applying the variation of constant formula to \eqref{UE4} and replacing $\omega$ by $\vartheta_{-\mathfrak{r}}\omega$ in the above inequality, we obtain
	\begin{align}\label{UE5-24}
		&\|\v(s;\mathfrak{r}-t,\vartheta_{-\mathfrak{r}}\omega,\v_{\mathfrak{r}-t})\|^2_{2} 
		\nonumber\\&\leq e^{-\int^{s}_{\mathfrak{r}-t}\left(\frac{\sigma^2}{2}-2\sigma y(\vartheta_{\upeta-\mathfrak{r}}\omega)\right)\d \upeta}\|\v_{\mathfrak{r}-t}\|^2_{2}
		\nonumber\\ & \quad + \frac{2}{\min\left\{2\nu\varepsilon_0,\sigma^2\right\}}\int_{\mathfrak{r}-t}^{s} e^{-\int^{s}_{\zeta}\left(\frac{\sigma^2}{2}-2\sigma y(\vartheta_{\upeta-\mathfrak{r}}\omega)\right)\d \upeta} \z^2(\zeta,\vartheta_{-\mathfrak{r}}\omega) \|\f(\cdot,\zeta)\|^2_{\H^{-1}}\d \zeta,
	\end{align}
	for a.e. $s\in[\mathfrak{r}-t, \mathfrak{r}]$. 
	\vskip 2mm
	\noindent
	\textbf{Step I:} \textit{In this step, we prove \eqref{UE-22}.} From \eqref{UE5-24}, we write 
	\begin{align}\label{UE51-24}
		&\int_{\mathfrak{r}-t}^{\mathfrak{r}}e^{-\int^{\mathfrak{r}}_{s}\left(\frac{\sigma^2}{2}-2\sigma y(\vartheta_{\upeta-\mathfrak{r}}\omega)\right)\d \upeta} \|\v(s;\mathfrak{r}-t,\vartheta_{-\mathfrak{r}}\omega,\v_{\mathfrak{r}-t})\|^2_{2} \d s
		\nonumber\\&\leq t e^{-\frac{\sigma^2}{4}t} e^{-\int^{\mathfrak{r}}_{\mathfrak{r}-t}\left(\frac{\sigma^2}{4}-2\sigma y(\vartheta_{\upeta-\mathfrak{r}}\omega)\right)\d \upeta}\|\v_{\mathfrak{r}-t}\|^2_{2}
		\nonumber\\ & \quad + \frac{2}{\min\left\{2\nu\varepsilon_0,\sigma^2\right\}} \int_{\mathfrak{r}-t}^{\mathfrak{r}} \int_{\mathfrak{r}-t}^{s} e^{-\int^{\mathfrak{r}}_{\zeta}\left(\frac{\sigma^2}{2}-2\sigma y(\vartheta_{\upeta-\mathfrak{r}}\omega)\right)\d \upeta} \z^2(\zeta,\vartheta_{-\mathfrak{r}}\omega) \|\f(\cdot,\zeta)\|^2_{\H^{-1}}\d \zeta\d s
		\nonumber\\& = t e^{-\frac{\sigma^2}{4}t} e^{-\int^{0}_{-t}\left(\frac{\sigma^2}{4}-2\sigma y(\vartheta_{\upeta}\omega)\right)\d \upeta}\|\v_{\mathfrak{r}-t}\|^2_{2}
		\nonumber\\ & \quad + \frac{2}{\min\left\{2\nu\varepsilon_0,\sigma^2\right\}} \int_{-\infty}^{0} \int_{-\infty}^{s} e^{-\int^{0}_{\zeta}\left(\frac{\sigma^2}{2}-2\sigma y(\vartheta_{\upeta}\omega)\right)\d \upeta} \z^2(\zeta,\omega) \|\f(\cdot,\zeta+\mathfrak{r})\|^2_{\H^{-1}}\d \zeta\d s.
	\end{align}	
	Let us establish that the final term of \eqref{UE51-24} is finite. By \eqref{Z3}, we have that there exist $R_3, R_4<0$ such that for all $\zeta\leq R_3$,
	\begin{align*}%\label{R3}
	\frac14 \left(\frac{\sigma^2}{2}- \delta \right)\zeta \leq 	-2\sigma y(\vartheta_{\zeta}\omega)\leq - \frac14 \left(\frac{\sigma^2}{2}- \delta \right)\zeta,
	\end{align*}
	and for all $\zeta\leq R_4$,
	\begin{align*}%\label{R4}
		\frac14 \left(\frac{\sigma^2}{2}- \delta \right)\zeta \leq  2\sigma\int^{0}_{\zeta} y(\vartheta_{\upeta}\omega)\d \upeta\leq - \frac14 \left(\frac{\sigma^2}{2} - \delta \right)\zeta,
	\end{align*}
	where $\delta$ is the positive constant appearing in \eqref{forcing1-B}. Therefore, for all $\zeta\leq \widehat{R}=:\min\{R_3,R_4\}<0$, we have
	\begin{align}\label{R-hat}
		&e^{\frac{\sigma^2}{2}\zeta+2\sigma\int^{0}_{\zeta} y(\vartheta_{\upeta}\omega)\d \upeta} \z^2(\zeta,\omega)= e^{\frac{\sigma^2}{2}\zeta+2\sigma\int^{0}_{\zeta} y(\vartheta_{\upeta}\omega)\d \upeta -2\sigma y(\vartheta_{\zeta}\omega)}   \leq e^{\frac12 \left(\frac{\sigma^2}{2} -  \delta \right)\zeta + \delta\zeta}.
	\end{align}
	Now, it follows from \eqref{forcing1-U} that for every $\mathfrak{r}\in\R$ and $\omega\in\Omega$,
	\begin{align}\label{UE71-24}
		& \int_{-\infty}^{0} \int_{-\infty}^{s} e^{-\int^{0}_{\zeta}\left(\frac{\sigma^2}{2}-2\sigma y(\vartheta_{\upeta}\omega)\right)\d \upeta} \z^2(\zeta,\omega) \|\f(\cdot,\zeta+\mathfrak{r})\|^2_{\H^{-1}}\d \zeta\d s
		\nonumber\\ & = \int_{\widehat{R}}^{0} \int_{\widehat{R}}^{s} e^{-\int^{0}_{\zeta}\left(\frac{\sigma^2}{2}-2\sigma y(\vartheta_{\upeta}\omega)\right)\d \upeta} \z^2(\zeta,\omega) \|\f(\cdot,\zeta+\mathfrak{r})\|^2_{\H^{-1}}\d \zeta\d s
		\nonumber\\ & \quad + \int_{\widehat{R}}^{0} \int_{-\infty}^{\widehat{R}} e^{-\int^{0}_{\zeta}\left(\frac{\sigma^2}{2}-2\sigma y(\vartheta_{\upeta}\omega)\right)\d \upeta} \z^2(\zeta,\omega) \|\f(\cdot,\zeta+\mathfrak{r})\|^2_{\H^{-1}}\d \zeta\d s
		\nonumber\\ & \quad + \int_{-\infty}^{\widehat{R}} \int_{-\infty}^{s} e^{-\int^{0}_{\zeta}\left(\frac{\sigma^2}{2}-2\sigma y(\vartheta_{\upeta}\omega)\right)\d \upeta} \z^2(\zeta,\omega) \|\f(\cdot,\zeta+\mathfrak{r})\|^2_{\H^{-1}}\d \zeta\d s
		\nonumber\\ & \leq  (-\widehat{R}) \sup_{\zeta\in[\widehat{R},0]}\left[e^{-\int^{0}_{\zeta}\left(\frac{\sigma^2}{2}-2\sigma y(\vartheta_{\upeta}\omega)\right)\d \upeta} \z^2(\zeta,\omega)\right] \int_{\widehat{R}}^{0}\|\f(\cdot,\zeta+\mathfrak{r})\|^2_{\H^{-1}}\d \zeta
		\nonumber\\ & \quad + (-\widehat{R}) \int_{-\infty}^{\widehat{R}} e^{\frac12\left(\frac{\sigma^2}{2}-\delta\right)\zeta + \delta\zeta}  \|\f(\cdot,\zeta+\mathfrak{r})\|^2_{\H^{-1}}\d \zeta  + \int_{-\infty}^{\widehat{R}} \int_{-\infty}^{s} e^{\frac12\left(\frac{\sigma^2}{2}-\delta\right)\zeta + \delta\zeta}  \|\f(\cdot,\zeta+\mathfrak{r})\|^2_{\H^{-1}}\d \zeta\d s
		\nonumber\\ & \leq  (-\widehat{R}) \sup_{\zeta\in[\widehat{R},0]}\left[e^{-\int^{0}_{\zeta}\left(\frac{\sigma^2}{2}-2\sigma y(\vartheta_{\upeta}\omega)\right)\d \upeta} \z^2(\zeta,\omega)\right] \int_{\widehat{R}}^{0}\|\f(\cdot,\zeta+\mathfrak{r})\|^2_{\H^{-1}}\d \zeta
		\nonumber\\ & \quad + (-\widehat{R}) \int_{-\infty}^{0} e^{\delta\zeta}  \|\f(\cdot,\zeta+\mathfrak{r})\|^2_{\H^{-1}}\d \zeta  + \int_{-\infty}^{0} e^{\frac12\left(\frac{\sigma^2}{2}-\delta\right)s } \d s \cdot \int_{-\infty}^{0} e^{ \delta\zeta}  \|\f(\cdot,\zeta+\mathfrak{r})\|^2_{\H^{-1}}\d \zeta
		\nonumber\\ & \quad < +\infty.
	\end{align} 
	
	Since $\v_{\mathfrak{r}-t}\in B(\mathfrak{r}-t,\vartheta_{-t}\omega)$, we have from \eqref{Z3} that for sufficiently large $t>0$
	%\label{initial_data}
	\begin{align*}
		e^{-\frac{\sigma^2}{4}t+2\sigma\int^{0}_{-t} y(\vartheta_{\upeta}\omega)\d \upeta}\|\v_{\mathfrak{r}-t}\|^2_{2}\leq e^{-\frac{\sigma^2}{8} t}\|B(\mathfrak{r}-t,\vartheta_{-t}\omega)\|^2_{\H}\to 0\ \text{ as } \ t \to \infty.
	\end{align*} 
Also $t e^{-\frac{\sigma^2}{4}t}\to 0$ as $t\to\infty$.  Therefore, there exists $\widehat{\mathcal{T}}_1=\widehat{\mathcal{T}}_1(\mathfrak{r},\omega,B)>0$ such that 
	\begin{align}\label{UE61-24}
		& te^{-\frac{\sigma^2}{2}t+2\sigma\int^{0}_{-t} y(\vartheta_{\upeta}\omega)\d \upeta}\|\v_{\mathfrak{r}-t}\|^2_{2}
		\nonumber\\ & \leq \frac{2}{\min\left\{2\nu\varepsilon_0,\sigma^2\right\}} \int_{-t}^{0} \int_{-t}^{s} e^{-\int^{0}_{\zeta}\left(\frac{\sigma^2}{2}-2\sigma y(\vartheta_{\upeta}\omega)\right)\d \upeta} \z^2(\zeta,\omega) \|\f(\cdot,\zeta+\mathfrak{r})\|^2_{\H^{-1}}\d \zeta\d s,
	\end{align}
	for all  $t\geq \widehat{\mathcal{T}}_1$. Hence, from \eqref{UE51-24}-\eqref{UE61-24}, \eqref{UE-22} follows.
	
	\vskip 2mm
	\noindent
	\textbf{Step I:} \textit{In this step, we prove \eqref{UE-24}.}  Again, from \eqref{UE5-24}, we write 
	\begin{align}\label{UE52-24}
		&\int_{\mathfrak{r}-t}^{\mathfrak{r}}e^{-\int^{\mathfrak{r}}_{s}\left(\frac{\sigma^2}{2}-2\sigma y(\vartheta_{\upeta-\mathfrak{r}}\omega)\right)\d \upeta} \z^{-2}(s,\vartheta_{-\mathfrak{r}}\omega) \|\v(s;\mathfrak{r}-t,\vartheta_{-\mathfrak{r}}\omega,\v_{\mathfrak{r}-t})\|^4_{2} \d s
		\nonumber\\ & \leq  \int_{\mathfrak{r}-t}^{\mathfrak{r}}e^{-\int^{\mathfrak{r}}_{s}\left(\frac{\sigma^2}{2}-2\sigma y(\vartheta_{\upeta-\mathfrak{r}}\omega)\right)\d \upeta} \z^{-2}(s,\vartheta_{-\mathfrak{r}}\omega) \bigg[e^{-\int^{s}_{\mathfrak{r}-t}\left(\frac{\sigma^2}{2}-2\sigma y(\vartheta_{\upeta-\mathfrak{r}}\omega)\right)\d \upeta}\|\v_{\mathfrak{r}-t}\|^2_{2}
		\nonumber\\ & \quad + \frac{2}{\min\left\{2\nu\varepsilon_0,\sigma^2\right\}}\int_{\mathfrak{r}-t}^{s} e^{-\int^{s}_{\zeta}\left(\frac{\sigma^2}{2}-2\sigma y(\vartheta_{\upeta-\mathfrak{r}}\omega)\right)\d \upeta} \z^2(\zeta,\vartheta_{-\mathfrak{r}}\omega) \|\f(\cdot,\zeta)\|^2_{\H^{-1}}\d \zeta\bigg]^2 \d s
%		\\
%		\\
%		\\
%		\\
%		\nonumber\\ & =  \int_{\mathfrak{r}-t}^{\mathfrak{r}}e^{\int^{\mathfrak{r}}_{s}\left(\frac{\sigma^2}{2}-2\sigma y(\vartheta_{\upeta-\mathfrak{r}}\omega)\right)\d \upeta} \bigg[e^{-\int^{\mathfrak{r}}_{\mathfrak{r}-t}\left(\frac{\sigma^2}{2}-2\sigma y(\vartheta_{\upeta-\mathfrak{r}}\omega)\right)\d \upeta}\|\v_{\mathfrak{r}-t}\|^2_{2}
%		\nonumber\\ & \quad + \frac{2}{\min\left\{2\nu\varepsilon_0,\sigma^2\right\}}\int_{\mathfrak{r}-t}^{s} e^{-\int^{\mathfrak{r}}_{\zeta}\left(\frac{\sigma^2}{2}-2\sigma y(\vartheta_{\upeta-\mathfrak{r}}\omega)\right)\d \upeta} \z^2(\zeta,\vartheta_{-\mathfrak{r}}\omega) \|\f(\cdot,\zeta)\|^2_{\H^{-1}}\d \zeta\bigg]^2 \d s
%		\nonumber\\ & =  \int_{-t}^{0}e^{\int^{0}_{s}\left(\frac{\sigma^2}{2}-2\sigma y(\vartheta_{\upeta}\omega)\right)\d \upeta} \bigg[e^{-\int^{0}_{-t}\left(\frac{\sigma^2}{2}-2\sigma y(\vartheta_{\upeta}\omega)\right)\d \upeta}\|\v_{\mathfrak{r}-t}\|^2_{2}
%		\nonumber\\ & \quad + \frac{2}{\min\left\{2\nu\varepsilon_0,\sigma^2\right\}}\int_{-t}^{s} e^{-\int^{0}_{\zeta}\left(\frac{\sigma^2}{2}-2\sigma y(\vartheta_{\upeta}\omega)\right)\d \upeta} \z^2(\zeta,\omega) \|\f(\cdot,\zeta+\mathfrak{r})\|^2_{\H^{-1}}\d \zeta\bigg]^2 \d s
%		\\
%		\\
%		\\
%		\\
		\nonumber\\ & \leq 2 \int_{\mathfrak{r}-t}^{\mathfrak{r}}e^{-\int^{\mathfrak{r}}_{s}\left(\frac{\sigma^2}{2}-2\sigma y(\vartheta_{\upeta-\mathfrak{r}}\omega)\right)\d \upeta} \z^{-2}(s,\vartheta_{-\mathfrak{r}}\omega) \bigg[e^{-2\int^{s}_{\mathfrak{r}-t}\left(\frac{\sigma^2}{2}-2\sigma y(\vartheta_{\upeta-\mathfrak{r}}\omega)\right)\d \upeta}\|\v_{\mathfrak{r}-t}\|^4_{2} \bigg]\d s
		  + \frac{4}{\min\left\{2\nu\varepsilon_0,\sigma^2\right\}} 
		\nonumber\\ & \quad \times \int_{\mathfrak{r}-t}^{\mathfrak{r}}e^{-\int^{\mathfrak{r}}_{s}\left(\frac{\sigma^2}{2}-2\sigma y(\vartheta_{\upeta-\mathfrak{r}}\omega)\right)\d \upeta} \z^{-2}(s,\vartheta_{-\mathfrak{r}}\omega) \bigg[\int_{\mathfrak{r}-t}^{s} e^{-\int^{s}_{\zeta}\left(\frac{\sigma^2}{2}-2\sigma y(\vartheta_{\upeta-\mathfrak{r}}\omega)\right)\d \upeta} \z^2(\zeta,\vartheta_{-\mathfrak{r}}\omega) \|\f(\cdot,\zeta)\|^2_{\H^{-1}}\d \zeta\bigg]^2 \d s
		\nonumber\\ & = 2 e^{-2\int^{\mathfrak{r}}_{\mathfrak{r}-t}\left(\frac{\sigma^2}{2}-2\sigma y(\vartheta_{\upeta-\mathfrak{r}}\omega)\right)\d \upeta}\|\v_{\mathfrak{r}-t}\|^4_{2}   \int_{\mathfrak{r}-t}^{\mathfrak{r}}e^{\int^{\mathfrak{r}}_{s}\left(\frac{\sigma^2}{2}-2\sigma y(\vartheta_{\upeta-\mathfrak{r}}\omega)\right)\d \upeta} \z^{-2}(s,\vartheta_{-\mathfrak{r}}\omega) \d s
		+ \frac{4}{\min\left\{2\nu\varepsilon_0,\sigma^2\right\}} 
		\nonumber\\ & \quad \times \int_{\mathfrak{r}-t}^{\mathfrak{r}}e^{\int^{\mathfrak{r}}_{s}\left(\frac{\sigma^2}{2}-2\sigma y(\vartheta_{\upeta-\mathfrak{r}}\omega)\right)\d \upeta} \z^{-2}(s,\vartheta_{-\mathfrak{r}}\omega) \bigg[\int_{\mathfrak{r}-t}^{s} e^{-\int^{\mathfrak{r}}_{\zeta}\left(\frac{\sigma^2}{2}-2\sigma y(\vartheta_{\upeta-\mathfrak{r}}\omega)\right)\d \upeta} \z^2(\zeta,\vartheta_{-\mathfrak{r}}\omega) \|\f(\cdot,\zeta)\|^2_{\H^{-1}}\d \zeta\bigg]^2 \d s
		\nonumber\\ & = 2 e^{-2\int^{0}_{-t}\left(\frac{\sigma^2}{2}-2\sigma y(\vartheta_{\upeta}\omega)\right)\d \upeta}\|\v_{\mathfrak{r}-t}\|^4_{2}   \int_{-t}^{0}e^{\int^{0}_{s}\left(\frac{\sigma^2}{2}-2\sigma y(\vartheta_{\upeta}\omega)\right)\d \upeta}  \z^{-2}(s,\omega) \d s
		+ \frac{4}{\min\left\{2\nu\varepsilon_0,\sigma^2\right\}} 
		\nonumber\\ & \quad \times \int_{-t}^{0}e^{\int^{0}_{s}\left(\frac{\sigma^2}{2}-2\sigma y(\vartheta_{\upeta}\omega)\right)\d \upeta} \z^{-2}(s,\omega) \bigg[\int_{-t}^{s} e^{-\int^{0}_{\zeta}\left(\frac{\sigma^2}{2}-2\sigma y(\vartheta_{\upeta}\omega)\right)\d \upeta} \z^2(\zeta,\omega) \|\f(\cdot,\zeta+\mathfrak{r})\|^2_{\H^{-1}}\d \zeta\bigg]^2 \d s
		\nonumber\\ & \leq  2 e^{-\sigma^2 t + 4\sigma \int^{0}_{-t} y(\vartheta_{\upeta}\omega) \d \upeta}\|\v_{\mathfrak{r}-t}\|^4_{2}   \int_{-t}^{0}e^{\int^{0}_{s}\left(\frac{\sigma^2}{2}-2\sigma y(\vartheta_{\upeta}\omega)\right)\d \upeta} \z^{-2}(s,\omega) \d s
		+ \frac{4}{\min\left\{2\nu\varepsilon_0,\sigma^2\right\}} 
		\nonumber\\ & \quad \times \int_{-\infty}^{0}e^{\int^{0}_{s}\left(\frac{\sigma^2}{2}-2\sigma y(\vartheta_{\upeta}\omega)\right)\d \upeta} \z^{-2}(s,\omega) \bigg[\int_{-\infty}^{s} e^{-\int^{0}_{\zeta}\left(\frac{\sigma^2}{2}-2\sigma y(\vartheta_{\upeta}\omega)\right)\d \upeta} \z^2(\zeta,\omega) \|\f(\cdot,\zeta+\mathfrak{r})\|^2_{\H^{-1}}\d \zeta\bigg]^2 \d s
		\nonumber\\ & \leq  2 e^{-\sigma^2 t + 4\sigma \int^{0}_{-t} y(\vartheta_{\upeta}\omega) \d \upeta}\|\v_{\mathfrak{r}-t}\|^4_{2}   \int_{-t}^{0}e^{\int^{0}_{s}\left(\frac{\sigma^2}{2}-2\sigma y(\vartheta_{\upeta}\omega)\right)\d \upeta} \z^{-2}(s,\omega) \d s
		 \nonumber\\ & \quad +  \frac{8}{\min\left\{2\nu\varepsilon_0,\sigma^2\right\}} 
		   \int_{\widehat{R}}^{0}\hspace{-2mm}e^{\int^{0}_{s}\left(\frac{\sigma^2}{2}-2\sigma y(\vartheta_{\upeta}\omega)\right)\d \upeta} \z^{-2}(s,\omega) \bigg[\int_{-t}^{s} e^{-\int^{0}_{\zeta}\left(\frac{\sigma^2}{2}-2\sigma y(\vartheta_{\upeta}\omega)\right)\d \upeta} \z^2(\zeta,\omega) \|\f(\cdot,\zeta+\mathfrak{r})\|^2_{\H^{-1}}\d \zeta\bigg]^2 \d s
		   \nonumber\\ & \quad +  \frac{8}{\min\left\{2\nu\varepsilon_0,\sigma^2\right\}} 
		   \int_{-t}^{\widehat{R}}\hspace{-2mm}e^{\int^{0}_{s}\left(\frac{\sigma^2}{2}-2\sigma y(\vartheta_{\upeta}\omega)\right)\d \upeta} \z^{-2}(s,\omega) \bigg[\int_{-t}^{s} e^{-\int^{0}_{\zeta}\left(\frac{\sigma^2}{2}-2\sigma y(\vartheta_{\upeta}\omega)\right)\d \upeta} \z^2(\zeta,\omega) \|\f(\cdot,\zeta+\mathfrak{r})\|^2_{\H^{-1}}\d \zeta\bigg]^2 \d s
		   \nonumber\\ & \leq  2 \underbrace{e^{-\sigma^2 t + 4\sigma \int^{0}_{-t} y(\vartheta_{\upeta}\omega) \d \upeta}\|\v_{\mathfrak{r}-t}\|^4_{2}   \int_{-t}^{0}e^{-\frac{\sigma^2}{2}s -2\sigma\int^{0}_{s} y(\vartheta_{\upeta}\omega)\d \upeta} \z^{-2}(s,\omega) \d s}_{\to 0 \text{ as } t\to\infty, \text{ due to } \eqref{UE82-24} \text{ below}.}
		   \nonumber\\ & \quad +  \frac{8}{\min\left\{2\nu\varepsilon_0,\sigma^2\right\}} 
		   \underbrace{\int_{\widehat{R}}^{0}e^{\int^{0}_{s}\left(\frac{\sigma^2}{2}-2\sigma y(\vartheta_{\upeta}\omega)\right)\d \upeta} \z^{-2}(s,\omega) \d s}_{<+\infty, \text{ since } y \text{ is continuous.}} 
		   \nonumber\\ & \quad \times \underbrace{\bigg[\int_{-\infty}^{0} e^{-\int^{0}_{\zeta}\left(\frac{\sigma^2}{2}-2\sigma y(\vartheta_{\upeta}\omega)\right)\d \upeta} \z^2(\zeta,\omega) \|\f(\cdot,\zeta+\mathfrak{r})\|^2_{\H^{-1}}\d \zeta\bigg]^2 \d s}_{< +\infty, \text{ due to } \eqref{UE7}.}
		   +  \frac{8}{\min\left\{2\nu\varepsilon_0,\sigma^2\right\}} 
		  \nonumber\\ & \quad \times \underbrace{\int_{-\infty}^{\widehat{R}}e^{\int^{0}_{s}\left(\frac{\sigma^2}{2}-2\sigma y(\vartheta_{\upeta}\omega)\right)\d \upeta} \z^{-2}(s,\omega) \bigg[\int_{-\infty}^{s} e^{-\int^{0}_{\zeta}\left(\frac{\sigma^2}{2}-2\sigma y(\vartheta_{\upeta}\omega)\right)\d \upeta} \z^2(\zeta,\omega) \|\f(\cdot,\zeta+\mathfrak{r})\|^2_{\H^{-1}}\d \zeta\bigg]^2 \d s}_{< +\infty, \text{ due to } \eqref{UE62-24} \text{ below}.}.
	\end{align}	
	We now show that the final term of \eqref{UE52-24} is finite. It follows from \eqref{R-hat} and \eqref{forcing1-U} that 
	\begin{align}\label{UE62-24}
		& \int_{-\infty}^{\widehat{R}}e^{\int^{0}_{s}\left(\frac{\sigma^2}{2}-2\sigma y(\vartheta_{\upeta}\omega)\right)\d \upeta} \z^{-2}(s,\omega) \bigg[\int_{-\infty}^{s} e^{-\int^{0}_{\zeta}\left(\frac{\sigma^2}{2}-2\sigma y(\vartheta_{\upeta}\omega)\right)\d \upeta} \z^2(\zeta,\omega) \|\f(\cdot,\zeta+\mathfrak{r})\|^2_{\H^{-1}}\d \zeta\bigg]^2 \d s
		\nonumber\\ 
		& \leq  \int_{-\infty}^{\widehat{R}}e^{\int^{0}_{s}\left(\frac{\sigma^2}{2}-2\sigma y(\vartheta_{\upeta}\omega)\right)\d \upeta} \z^{-2}(s,\omega) \bigg[\int_{-\infty}^{s} e^{\frac12 \left(\frac{\sigma^2}{2} -  \delta \right)\zeta + \delta\zeta} \|\f(\cdot,\zeta+\mathfrak{r})\|^2_{\H^{-1}}\d \zeta\bigg]^2 \d s
		\nonumber\\ 
		& \leq  \int_{-\infty}^{\widehat{R}}e^{\int^{0}_{s}\left(\frac{\sigma^2}{2}-2\sigma y(\vartheta_{\upeta}\omega)\right)\d \upeta} \z^{-2}(s,\omega) e^{ \left(\frac{\sigma^2}{2} -  \delta \right)s } \bigg[\int_{-\infty}^{s} e^{\delta\zeta} \|\f(\cdot,\zeta+\mathfrak{r})\|^2_{\H^{-1}}\d \zeta\bigg]^2 \d s
		\nonumber\\ 
		& \leq  \underbrace{\int_{-\infty}^{\widehat{R}}e^{-\delta s - \int^{0}_{s}2\sigma y(\vartheta_{\upeta}\omega)\d \upeta} \z^{-2}(s,\omega) \d s}_{<+\infty, \text{ due to } \eqref{Z3}.} \times   \bigg[\underbrace{\int_{-\infty}^{0} e^{\delta\zeta} \|\f(\cdot,\zeta+\mathfrak{r})\|^2_{\H^{-1}}\d \zeta}_{<+\infty, \text{ due to } \eqref{forcing1-U}.}\bigg]^2  
		\nonumber\\ & < + \infty.
	\end{align}
Next we show that $e^{-\sigma^2 t + 4\sigma \int^{0}_{-t} y(\vartheta_{\upeta}\omega) \d \upeta}\|\v_{\mathfrak{r}-t}\|^4_{2}   \int_{-t}^{0}e^{-\frac{\sigma^2}{2}s -2\sigma\int^{0}_{s} y(\vartheta_{\upeta}\omega)\d \upeta} \z^{-2}(s,\omega) \d s \to 0$ as $t\to\infty$.  By \eqref{Z3}, we have that there exist $R_5<0$ such that for all $s\leq R_5$,
\begin{align*}%\label{R4}
\frac{\sigma^2}{8}  s \leq  2\sigma\int^{0}_{s} y(\vartheta_{\upeta}\omega)\d \upeta - 2\sigma y(\vartheta_s\omega) \leq - \frac{\sigma^2}{8} s.
\end{align*}
Therefore, for all $t> -R_5$, we have 
	\begin{align}\label{UE72-24}
		& e^{-\sigma^2 t + 4\sigma \int^{0}_{-t} y(\vartheta_{\upeta}\omega) \d \upeta}\|\v_{\mathfrak{r}-t}\|^4_{2}   \int_{-t}^{0}e^{-\frac{\sigma^2}{2}s -2\sigma\int^{0}_{s} y(\vartheta_{\upeta}\omega)\d \upeta} \z^{-2}(s,\omega) \d s
		\nonumber\\ & \leq  e^{-\sigma^2 t + \frac{\sigma^2}{4} t}\|\v_{\mathfrak{r}-t}\|^4_{2}   \bigg[\int_{R_5}^{0}e^{-\frac{\sigma^2}{2}s -2\sigma\int^{0}_{s} y(\vartheta_{\upeta}\omega)\d \upeta} \z^{-2}(s,\omega) \d s + \int_{-t}^{R_5}e^{-\frac{\sigma^2}{2}s - \frac{\sigma^2}{8} s} \d s\bigg]
		\nonumber\\ & \leq  e^{- \frac{3\sigma^2}{4} t}\|\v_{\mathfrak{r}-t}\|^4_{2}   \int_{R_5}^{0}e^{-\frac{\sigma^2}{2}s -2\sigma\int^{0}_{s} y(\vartheta_{\upeta}\omega)\d \upeta} \z^{-2}(s,\omega) \d s  +  e^{- \frac{3\sigma^2}{4} t}\|\v_{\mathfrak{r}-t}\|^4_{2} \bigg[\frac{8}{5\sigma^2}e^{\frac{5\sigma^2}{8}t}\bigg] 
		\nonumber\\ & \leq  e^{- \frac{\sigma^2}{8} t} \|\v_{\mathfrak{r}-t}\|^4_{2} \bigg[\underbrace{ \int_{R_5}^{0}e^{-\frac{\sigma^2}{2}s -2\sigma\int^{0}_{s} y(\vartheta_{\upeta}\omega)\d \upeta} \z^{-2}(s,\omega) \d s}_{<+\infty, \text{ since } y \text{ is continuous.}} +  \frac{8}{5\sigma^2}\bigg] .
	\end{align}
Since $\v_{\mathfrak{r}-t}\in B(\mathfrak{r}-t,\vartheta_{-t}\omega)$, we have  that 
%\label{initial_data}
\begin{align*}
	e^{- \frac{\sigma^2}{16} t}\|\v_{\mathfrak{r}-t}\|^2_{2}\leq e^{-\frac{\sigma^2}{16} t}\|B(\mathfrak{r}-t,\vartheta_{-t}\omega)\|^2_{\H}\to 0\ \text{ as } \ t \to \infty,
\end{align*} 
	which implies from \eqref{UE72-24}
	\begin{align}\label{UE82-24}
		& e^{-\sigma^2 t + 4\sigma \int^{0}_{-t} y(\vartheta_{\upeta}\omega) \d \upeta}\|\v_{\mathfrak{r}-t}\|^4_{2}   \int_{-t}^{0}e^{-\frac{\sigma^2}{2}s -2\sigma\int^{0}_{s} y(\vartheta_{\upeta}\omega)\d \upeta} \z^{-2}(s,\omega) \d s\to 0\ \text{ as } \ t \to \infty.
	\end{align}
	Hence, from \eqref{UE52-24}, \eqref{UE62-24} and \eqref{UE82-24}, \eqref{UE-24} follows. This completes the proof of this lemma.
\end{proof}

 Let $\boldsymbol{\wp}$ be a smooth function such that $0\leq\boldsymbol{\wp}(\xi)\leq 1$ for $\xi\in[0,\infty)$ and
\begin{align}\label{337}
	\boldsymbol{\wp}(\xi)=\begin{cases*}
		0, \text{ for }0\leq \xi\leq 1,\\
		1, \text{ for } \xi\geq2.
	\end{cases*}
\end{align}
By the hypothesis $\boldsymbol{\wp}$ is constant (hence all derivatives vanish) on
$[0,1]$ and on $[2,\infty)$. Thus the only interval where $\boldsymbol{\wp}'$ and
$\boldsymbol{\wp}''$ can be nonzero in the compact interval $[1,2]$. Since $\boldsymbol{\wp}'\in \mathrm{C}([1,2])$ and $\boldsymbol{\wp}''\in \mathrm{C}([1,2])$, each attains a finite maximum
on $[1,2]$. Let us define
\begin{align*}
	M_1:=\sup_{\xi\in[1,2]}|\boldsymbol{\wp}'(\xi)|,\ 
	M_2:=\sup_{\xi\in[1,2]}|\boldsymbol{\wp}''(\xi)|,
\end{align*}
and set $C:=\max\{M_1,M_2\}$. For $\xi\notin[1,2]$ we have $\boldsymbol{\wp}'(\xi)=\boldsymbol{\wp}''(\xi)=0$,
so the inequalities $|\boldsymbol{\wp}'(\xi)|\le C$ and $|\boldsymbol{\wp}''(\xi)|\le C$ hold for all
$\xi\in[0,\infty)$.

Next result provides that the solution to system \eqref{CTGF} satisfies the uniform tail estimates.\\

\begin{lemma}\label{LR}
	Assume that condition \eqref{third-grade-paremeters-res} and Hypothesis \ref{Hyp-f-U} are satisfied, and let $\mathfrak{r}\in\R$, $\omega\in\Omega$ and $B=\{B(\mathfrak{r},\omega): \mathfrak{r}\in\R, \; \omega\in\Omega\}\in\mathfrak{D}$. Then, for each $\varepsilon>0$ and for each $s\in[\mathfrak{r}-1,\mathfrak{r}]$, there exists $\mathfrak{k}_0:=\mathfrak{k}_0(\varepsilon,s)\in\N$ such that  the solution $\v(\cdot)$ to system \eqref{CTGF} satisfies
	\begin{align}
		\left\|\v(s;\mathfrak{r}-t,\vartheta_{-\mathfrak{r}}\omega,\v_{\mathfrak{r}-t})\right\|^2_{\L^2(\O^{c}_{\mathfrak{k}})}\leq\varepsilon, 
	\end{align}
	for all $\v_{\mathfrak{r}-t}\in B(\mathfrak{r}-t,\vartheta_{-t}\omega)$, for all $t\geq \widetilde{\mathcal{T}}=\max\{\mathcal{T}, \widehat{\mathcal{T}}\}$ and for all $\mathfrak{k}\geq \mathfrak{k}_0$,	where  $\O_{\mathfrak{k}}=\{x\in\O :\,|x|< \mathfrak{k}\}$, $\O^{c}_{\mathfrak{k}}=\O\backslash\O_{\mathfrak{\mathfrak{k}}}$,
	and   ${\mathcal{T}}$ and $\widehat{\mathcal{T}}$ are the same time given in Lemmas \ref{LemmaUe} and \ref{LemmaUe-24}, respectively.
\end{lemma}
\begin{proof} 
	For $\mathfrak{k}\in \N$ and $x\in\R^d$, taking the inner product of the first equation of \eqref{eqn-CTGF-Pressure} with $\boldsymbol{\wp}^2\left(\mathfrak{k}^{-2}|x|^2\right)\v$ in $\mathbb{L}^2(\O)$, we have
	\begin{align}\label{ep1}
		&\frac{1}{2} \frac{\d}{\d t}\int_{\O}\boldsymbol{\wp}^2\left(\mathfrak{k}^{-2}|x|^2\right)|\v|^2\d x +\left[\frac{\sigma^2}{2}-\sigma y(\vartheta_{t}\omega)\right]\int_{\O}\boldsymbol{\wp}^2\left(\mathfrak{k}^{-2}|x|^2\right)|\v|^2\d x
		\nonumber \\& = \underbrace{\nu\left<\Delta\v , \boldsymbol{\wp}^2\left(\mathfrak{k}^{-2}|x|^2\right) \v \right>}_{I_{1}(\mathfrak{k},t)} \underbrace{- \z^{-1}(t,\omega) \left<(\v\cdot \nabla)\v , \boldsymbol{\wp}^2\left(\mathfrak{k}^{-2}|x|^2\right)\v\right> }_{I_{2}(\mathfrak{k},t)} 
		\nonumber\\ & \quad + \underbrace{\alpha \z^{-1}(t,\omega) \left<\text{div}((\Arm(\v))^2) , \boldsymbol{\wp}^2\left(\mathfrak{k}^{-2}|x|^2\right) \v \right>}_{I_{3}(\mathfrak{k},t)} 
		  + \underbrace{ \beta \z^{-2}(t,\omega) \left< \text{div}(|\Arm(\v)|^2\Arm(\v)) , \boldsymbol{\wp}^2\left(\mathfrak{k}^{-2}|x|^2\right) \v \right>}_{I_{4}(\mathfrak{k},t)}
		  \nonumber\\ & \quad  \underbrace{- \z(t,\omega) \left< \nabla \textbf{P}, \boldsymbol{\wp}^2\left(\mathfrak{k}^{-2}|x|^2\right)\v\right>}_{I_{5}(\mathfrak{k},t)} + \underbrace{ \z(t,\omega) \left<\f\boldsymbol{\wp}^2\left(\mathfrak{k}^{-2}|x|^2\right)\v\right>}_{I_{6}(\mathfrak{k},t)}.
	\end{align}	
	We first recall some useful bounds for $\textbf{P}_1$ and $\textbf{P}_2$ from the Appendix \ref{PR}, where $\nabla\textbf{P}=\nabla\textbf{P}_1+\nabla\textbf{P}_2$. From \eqref{estimate-pressure3}-\eqref{estimate-pressure4}, we write
	\begin{align}\label{estimate-pressure1}
		\int_{\mathfrak{r}-t}^{\mathfrak{r}}e^{-\int^{\mathfrak{r}}_{\zeta}\left(\frac{\sigma^2}{2}-2\sigma y(\vartheta_{\upeta-\mathfrak{r}}\omega)\right)\d \upeta}\z^2 (\zeta,\vartheta_{-\mathfrak{r}}\omega)\|\nabla\mathbf{P}_1(\zeta;\mathfrak{r}-t,\vartheta_{-\mathfrak{r}}\omega,\v_{\mathfrak{r}-t})\|^2_{\H^{-1}}\d\zeta < +\infty,
	\end{align}
	and 
	\begin{align}\label{estimate-pressure2}
		\int_{\mathfrak{r}-t}^{\mathfrak{r}}e^{-\int^{\mathfrak{r}}_{\zeta}\left(\frac{\sigma^2}{2}-2\sigma y(\vartheta_{\upeta-\mathfrak{r}}\omega)\right)\d \upeta}\z^2(\zeta,\vartheta_{-\mathfrak{r}}\omega)\|\nabla\mathbf{P}_2(\zeta;\mathfrak{r}-t,\vartheta_{-\mathfrak{r}}\omega,\v_{\mathfrak{r}-t})\|^{\frac43}_{\mathbb{W}^{-1,\frac43}}\d\zeta < +\infty.
	\end{align}
	Let us now estimate each term on the right hand side of \eqref{ep1}. Integration by parts, condition \eqref{third-grade-paremeters-res}, inequalities \eqref{Gen_lady-3}, \eqref{Gen_lady-4} and \eqref{Gag-Korn-ineq}, and H\"older's and Young's inequalities help us to obtain
	\begin{align}
		I_{1}(\mathfrak{k},t)
		&= -\frac{\nu}{2} \int_{\O} \boldsymbol{\wp}^2\left(\mathfrak{k}^{-2}|x|^2\right) |\Arm(\v)|^2  \d x - \frac{4\nu}{\mathfrak{k}^2} \int\limits_{\O}\boldsymbol{\wp}\left(\mathfrak{k}^{-2}|x|^2\right) \boldsymbol{\wp}'\left(\mathfrak{k}^{-2}|x|^2\right)
		[(x\cdot\nabla) \v\cdot\v+(\v\cdot\nabla) \v\cdot x]\d x\nonumber\\
		&\leq -\frac{\nu}{2} \int_{\O}\boldsymbol{\wp}^2\left(\mathfrak{k}^{-2}|x|^2\right)  |\Arm(\v)|^2   \d x+  \frac{4\sqrt{2}\nu}{\mathfrak{k}} \int\limits_{\O\cap\{\mathfrak{k}\leq|x|\leq \sqrt{2}\mathfrak{k}\}} \left|\boldsymbol{\wp}'\left(\mathfrak{k}^{-2}|x|^2\right)\right| \left|\v\right| \left|\nabla \v\right| \d x\nonumber\\&\leq -\frac{\nu}{2} \int_{\O}\boldsymbol{\wp}^2\left(\mathfrak{k}^{-2}|x|^2\right)  |\Arm(\v)|^2    \d x+  \frac{C}{\mathfrak{k}} \int_{\O}\left|\v\right| \left|\nabla \v\right| \d x
		\nonumber\\ & \leq  - \frac{\nu}{2} \int_{\O} \boldsymbol{\wp}^2\left(\mathfrak{k}^{-2}|x|^2\right)  |\Arm(\v)|^2    \d x +   \frac{C}{\mathfrak{k}} \left(\|\v\|^2_{2}+\|\Arm(\v)\|^2_{2}\right),\label{ep2}
		\\  
		\left|I_{2}(\mathfrak{k},t)\right|&= \z^{-1}(t,\omega)\left|\frac{2}{\mathfrak{k}^2}\int_{\O} \boldsymbol{\wp}\left(\mathfrak{k}^{-2}|x|^2\right)\boldsymbol{\wp}'\left(\mathfrak{k}^{-2}|x|^2\right)x\cdot\v |\v|^2 \d x\right|\nonumber\\&= \z^{-1}(t,\omega) \left|\frac{2}{\mathfrak{k}^2} \int\limits_{\O\cap\{\mathfrak{k}\leq|x|\leq \sqrt{2}\mathfrak{k}\}}\boldsymbol{\wp}\left(\mathfrak{k}^{-2}|x|^2\right) \boldsymbol{\wp}'\left(\mathfrak{k}^{-2}|x|^2\right)x\cdot\v |\v|^2 \d x\right|\nonumber\\&\leq  \z^{-1}(t,\omega)\frac{2\sqrt{2}}{\mathfrak{k}} \int\limits_{\O\cap\{\mathfrak{k}\leq|x|\leq \sqrt{2}\mathfrak{k}\}} \left|\boldsymbol{\wp}'\left(\mathfrak{k}^{-2}|x|^2\right)\right| |\v|^3 \d x 
		\nonumber\\ &  \leq \frac{C}{\mathfrak{k}} \z^{-1}(t,\omega)\|\v\|^3_{ 3} \leq \frac{C}{\mathfrak{k}} \z^{-1}(t,\omega)\|\Arm(\v)\|^{\frac{2d}{4+d}}_{4} \|\v\|_2^{\frac{12+d}{4+d}} 
		\nonumber\\ & =   \frac{C}{\mathfrak{k}} \z^{-1+\frac{d}{4+d}+\frac{12+d}{2(4+d)}}(t,\omega) \z^{-\frac{d}{4+d}}(t,\omega)\|\Arm(\v)\|^{\frac{2d}{4+d}}_{4}  \z^{-\frac{12+d}{2(4+d)}}(t,\omega) \|\v\|_2^{\frac{12+d}{4+d}}
		\nonumber\\ & =   \frac{C}{\mathfrak{k}} \z^{\frac12}(t,\omega) \z^{-\frac{d}{4+d}}(t,\omega)\|\Arm(\v)\|^{\frac{2d}{4+d}}_{4}  \z^{-\frac{12+d}{2(4+d)}}(t,\omega) \|\v\|_2^{\frac{12+d}{4+d}}
		\nonumber\\ & \leq    \frac{C}{\mathfrak{k}} \left[\z^{2}(t,\omega) + \z^{-2}(t,\omega)\|\Arm(\v)\|^{4}_{4}  +  \z^{-2}(t,\omega) \|\v\|_2^{4}\right],\label{ep3}
		\\
		I_{4}(\mathfrak{k},t) & = -	\frac{\beta}{2} \z^{-2}(t,\omega)  \int_{\O} \boldsymbol{\wp}^2\left(\mathfrak{k}^{-2}|x|^2\right)|\Arm(\v)|^4 \d x 
		\nonumber\\ & \quad - \frac{4\beta}{\mathfrak{k}^2} \z^{-2}(t,\omega) \int_{\O} \boldsymbol{\wp}\left(\mathfrak{k}^{-2}|x|^2\right)\boldsymbol{\wp}^{\prime}\left(\mathfrak{k}^{-2}|x|^2\right) [|\Arm(\v)|^2\Arm(\v): (x^{T} \v)] \d x
		\nonumber\\ & = -	\frac{\beta}{2} \z^{-2}(t,\omega)  \int_{\O} \boldsymbol{\wp}^2\left(\mathfrak{k}^{-2}|x|^2\right)|\Arm(\v)|^4 \d x 
		\nonumber\\ & \quad - \frac{4\beta}{\mathfrak{k}^2} \z^{-2}(t,\omega) \int\limits_{\O\cap\{\mathfrak{k}\leq|x|\leq \sqrt{2}\mathfrak{k}\}} \boldsymbol{\wp}^{\prime}\left(\mathfrak{k}^{-2}|x|^2\right) [|\Arm(\v)|^2\Arm(\v): (x^{T} \v)] \d x
		\nonumber\\ & \leq -	\frac{\beta}{2} \z^{-2}(t,\omega)  \int_{\O} \boldsymbol{\wp}^2\left(\mathfrak{k}^{-2}|x|^2\right)|\Arm(\v)|^4 \d x + \frac{C}{\mathfrak{k}} \z^{-2}(t,\omega) \int\limits_{\O\cap\{\mathfrak{k}\leq|x|\leq \sqrt{2}\mathfrak{k}\}}  |\Arm(\v)|^3 |\v| \d x
		\nonumber\\ & \leq -	\frac{\beta}{2}  \z^{-2}(t,\omega) \int_{\O} \boldsymbol{\wp}^2\left(\mathfrak{k}^{-2}|x|^2\right)|\Arm(\v)|^4 \d x + \frac{C}{\mathfrak{k}} \z^{-2}(t,\omega)\|\Arm(\v)\|^3_{4}\|\v\|_{4}
		\nonumber\\ & \leq -	\frac{\beta}{2} \z^{-2}(t,\omega) \int_{\O} \boldsymbol{\wp}^2\left(\mathfrak{k}^{-2}|x|^2\right)|\Arm(\v)|^4 \d x + \frac{C}{\mathfrak{k}} \z^{-2}(t,\omega) \left[\|\Arm(\v)\|^4_{4}+\|\v\|^4_{4}\right]
		\nonumber\\ & \leq -	\frac{\beta}{2} \z^{-2}(t,\omega) \int_{\O} \boldsymbol{\wp}^2\left(\mathfrak{k}^{-2}|x|^2\right)|\Arm(\v)|^4 \d x + \frac{C}{\mathfrak{k}} \z^{-2}(t,\omega) \left[\|\Arm(\v)\|^4_{4}+\|\v\|^4_{2}\right],\label{ep4}
		\\
		|I_{3}(\mathfrak{k},t)| & = \z^{-1}(t,\omega)\left|	\frac{\alpha}{2}  \int_{\O} \boldsymbol{\wp}^2\left(\mathfrak{k}^{-2}|x|^2\right)[(\Arm(\v))^2: \Arm(\v)]  \d x + \frac{4\alpha}{\mathfrak{k}^2} \int_{\O} \boldsymbol{\wp}\left(\mathfrak{k}^{-2}|x|^2\right)\boldsymbol{\wp}^{\prime}\left(\mathfrak{k}^{-2}|x|^2\right) (\Arm(\v))^2: (x^{T} \v) \d x\right|
		\nonumber\\ & \leq  \frac{\nu(1-\varepsilon_0)}{4} \int_{\O} \boldsymbol{\wp}^2\left(\mathfrak{k}^{-2}|x|^2\right)|\Arm(\v)|^2   \d x  + \frac{\beta(1-\varepsilon_0)}{2} \z^{-2}(t,\omega)  \int_{\O} \boldsymbol{\wp}^2\left(\mathfrak{k}^{-2}|x|^2\right)|\Arm(\v)|^4 \d x
		\nonumber\\ & \quad + \frac{C}{\mathfrak{k}} \left[ \z^{-2}(t,\omega)\|\Arm(\v)\|^4_{4}+\|\v\|^2_{2}\right],\label{ep5}
		\\
		|I_5(\mathfrak{k},t)|&\leq \left|\z(t,\omega) \left< \nabla \textbf{P}_1, \boldsymbol{\wp}^2\left(\mathfrak{k}^{-2}|x|^2\right)\v\right>\right| + \left|\z(t,\omega) \left< \nabla \textbf{P}_2, \boldsymbol{\wp}^2\left(\mathfrak{k}^{-2}|x|^2\right)\v\right>\right|
		\nonumber\\ 
		&\leq \z(t,\omega) \|\nabla \textbf{P}_1\|_{\H^{-1}} \left\|\boldsymbol{\wp}^2\left(\mathfrak{k}^{-2}|x|^2\right)\v\right\|_{\H^1} + \z(t,\omega) \|\nabla \textbf{P}_2\|_{\mathbb{W}^{-1,\frac{4}{3}}} \left\|\boldsymbol{\wp}^2\left(\mathfrak{k}^{-2}|x|^2\right)\v\right\|_{\mathbb{W}^{1,4}} 
		\nonumber\\ 
		&\leq \z(t,\omega) \|\nabla \textbf{P}_1\|_{\H^{-1}} \left(\left\|\boldsymbol{\wp}^2\left(\mathfrak{k}^{-2}|x|^2\right)\v\right\|_{2}^2 + \left\|\nabla\left[\boldsymbol{\wp}^2\left(\mathfrak{k}^{-2}|x|^2\right)\v\right]\right\|_{2}^2 \right)^{\frac12} 
		\nonumber\\ & \quad + \z(t,\omega) \|\nabla \textbf{P}_2\|_{\mathbb{W}^{-1,\frac{4}{3}}} \left(\left\|\boldsymbol{\wp}^2\left(\mathfrak{k}^{-2}|x|^2\right)\v\right\|_{4}^4 + \left\|\nabla\left[\boldsymbol{\wp}^2\left(\mathfrak{k}^{-2}|x|^2\right)\v\right]\right\|_{4}^4 \right)^{\frac14}
		\nonumber\\ 
		&\leq  C \z(t,\omega) \|\nabla \textbf{P}_1\|_{\H^{-1}} \left(\int_{\O\cap\{\mathfrak{k}\leq|x|\}}|\v|^2\d x +\int_{\O\cap\{\mathfrak{k}\leq|x|\}}|\Arm(\v)|^2\d x \right)^{\frac12} 
		\nonumber\\ & \quad + C \z(t,\omega) \|\nabla \textbf{P}_2\|_{\mathbb{W}^{-1,\frac{4}{3}}} \left(\left\{\int_{\O\cap\{\mathfrak{k}\leq|x|\}}|\v|^2\d x\right\}^2 +\int_{\O\cap\{\mathfrak{k}\leq|x|\}}|\Arm(\v)|^4\d x \right)^{\frac14},\label{ep6}
		\\
		|I_{6}(\mathfrak{k},t)|&\leq \frac{\sigma^2}{4}\int_{\O}\boldsymbol{\wp}^2\left(\mathfrak{k}^{-2}|x|^2\right)|\v|^2\d x + C \z^{2}(t,\omega) \int_{\O}\boldsymbol{\wp}^2\left(\mathfrak{k}^{-2}|x|^2\right)|\f|^2\d x.\label{ep7}
	\end{align}
	A combination of \eqref{ep1}-\eqref{ep7} yields
	\begin{align}\label{ep8}
		&	 \frac{\d}{\d t}\int_{\O}\boldsymbol{\wp}^2\left(\mathfrak{k}^{-2}|x|^2\right)|\v|^2\d x  +\left[\frac{\sigma^2}{2}-2\sigma y(\vartheta_{t}\omega)\right]\int_{\O}\boldsymbol{\wp}^2\left(\mathfrak{k}^{-2}|x|^2\right)|\v|^2\d x
		\nonumber\\ & \leq   C\int_{\O}\boldsymbol{\wp}^2\left(\mathfrak{k}^{-2}|x|^2\right)|\f|^2\d x+ \frac{C}{\mathfrak{k}} \left[\|\v\|^2_{2}   + \|\Arm(\v)\|^2_{2} + \z^{2}(t,\omega) + \z^{-2}(t,\omega)\|\Arm(\v)\|^{4}_{4}  +  \z^{-2}(t,\omega) \|\v\|_2^{4}  \right]
		\nonumber\\ & \quad C \z(t,\omega) \|\nabla \textbf{P}_1\|_{\H^{-1}} \left(\int_{\O\cap\{\mathfrak{k}\leq|x|\}}|\v|^2\d x +\int_{\O\cap\{\mathfrak{k}\leq|x|\}}|\Arm(\v)|^2\d x \right)^{\frac12} 
		\nonumber\\ & \quad + C \z(t,\omega) \|\nabla \textbf{P}_2\|_{\mathbb{W}^{-1,\frac{4}{3}}} \left(\left\{\int_{\O\cap\{\mathfrak{k}\leq|x|\}}|\v|^2\d x\right\}^2 +\int_{\O\cap\{\mathfrak{k}\leq|x|\}}|\Arm(\v)|^4\d x \right)^{\frac14}.
	\end{align}
	Finally, an application of variation of constants formula to \eqref{ep8}, Lemmas \ref{LemmaUe} and \ref{LemmaUe-24}, bounds \eqref{estimate-pressure1}-\eqref{estimate-pressure2}, and Hypothesis \ref{Hyp-f-U} helps us to obtain that for all $s\in[\mathfrak{r}-1,\mathfrak{r}]$ and for all $t\geq \widetilde{\mathcal{T}}=\max\{\mathcal{T},\widehat{\mathcal{T}}\}$ 
	\begin{align}\label{ep9}
		&	\int_{\O}\boldsymbol{\wp}^2\left(\mathfrak{k}^{-2}|x|^2\right)|\v(s;\mathfrak{r}-t, \vartheta_{-\mathfrak{r}}\omega,\v_{\mathfrak{r}-t})|^2\d x  \to 0\text{ as } \mathfrak{k}\to\infty.
	\end{align}
	Hence, from \eqref{ep9}, we conclude that for any $\varepsilon>0$ and for any $\omega\in\Omega$ and $s\in[\mathfrak{r}-1,\mathfrak{r}]$, there exists a $\mathfrak{k}_0\in\N$ such that 
	\begin{align*}%\label{ep10-N}
		&	\int_{\O\cap\{\mathfrak{k}\leq|x|\}}|\v(s;\mathfrak{r}-t, \vartheta_{-\mathfrak{r}}\omega,\v_{\mathfrak{r}-t})|^2 \d x  \leq \varepsilon,
	\end{align*}
	for all $\mathfrak{k}\geq \mathfrak{k}_0$ and $t\geq \widetilde{\mathcal{T}}$. This completes the proof.
\end{proof}

The following lemma provides the $\mathfrak{D}$-pullback asymptotic compactness of $\Phi$ on unbounded domains using Lemma \ref{LR}.
\begin{lemma}\label{Asymptotic_UB_GS}
	Assume that condition \eqref{third-grade-paremeters-res} and Hypothesis \ref{Hyp-f-U} are satisfied. Then, for every $\mathfrak{r}\in \R,$ $\omega\in\Omega$, ${B}=\{B(\mathfrak{r},\omega):\mathfrak{r}\in \R, \;\omega\in\Omega\}\in \mathfrak{D}$, $t_m\to \infty$ and $\v_{0m}\in B(\mathfrak{r}-t_m, \vartheta_{-t_{m}}\omega)$, the sequence $\Phi(t_m,\mathfrak{r}-t_m,\vartheta_{-t_{m}}\omega,\v_{0m})$ or $\v(\mathfrak{r};\mathfrak{r}-t_m,\vartheta_{-\mathfrak{r}}\omega,\v_{0m})$ of solutions to system \eqref{STGF} has a convergent subsequence in $\H$.
\end{lemma}
\begin{proof}
	Lemma \ref{LemmaUe} implies that there exists $\mathcal{T}=\mathcal{T}(\mathfrak{r},\omega,{K})\geq 1$ such that for all $t\geq \mathcal{T}$ and $\v_{\ast}\in B(\mathfrak{r}-t,\vartheta_{-t}\omega)$,
	{
		\begin{align}\label{Uac1}
			\|\v(\mathfrak{r}-1;\mathfrak{r}-t,\vartheta_{-\mathfrak{r}}\omega,\v_{\ast})\|_2\leq C,
		\end{align}
		where $C$ is a positive constant independent of $t$ and  $\v_{\ast}.$}
	Since $t_m\to \infty$, there exists $M_1=M_1(\mathfrak{r},\omega,B)\in\N$ such that $t_m\geq \mathcal{T},$ for all $m\geq M_1$. Since $\v_{0m} \in B(\mathfrak{r}-t_m,\vartheta_{-t_m}\omega)$, \eqref{Uac1} implies that for all $m\geq M_1$, {the sequence 
		\begin{align}\label{Uac2}
			\{\v(\mathfrak{r}-1;\mathfrak{r}-t_m, \vartheta_{-\mathfrak{r}}\omega,\v_{0m})\}_{m\geq M_1} \subset\H
		\end{align}
		is bounded in $\H$.}
	We infer from \eqref{Uac2} and Lemma \ref{D-convege} that there exists $s\in(\mathfrak{r}-1,\mathfrak{r})$, $\tilde{\u}\in\H$ and a subsequence (not relabeling) such that for every $\mathfrak{k}\in\N$ 
	\begin{align}\label{Uac3}
		\v(s;\mathfrak{r}-t_m, \vartheta_{-\mathfrak{r}}\omega,\v_{0m})=\v(s;\mathfrak{r}-1, \vartheta_{-\mathfrak{r}}\omega,\v(\mathfrak{r}-1;\mathfrak{r}-t_m,\vartheta_{-\mathfrak{r}}\omega,\v_{0m}))\to \tilde{\u} \  \text{ in }\   \L^2(\O_\mathfrak{k}).
	\end{align}
	as $m\to\infty$. Therefore, we infer from {the proof of } Lemma \ref{Continuity} that there exists a positive constant $C_{Lip}$ such that
	\begin{align}\label{Uac7}
		&\|\v(\mathfrak{r};s,\v(s;\mathfrak{r}-t_m, \vartheta_{-\mathfrak{r}}\omega,\v_{0m}))-\v(\mathfrak{r};s, \vartheta_{-\mathfrak{r}},\tilde{\u})\|^2_{2}\leq C_{Lip}\|\v(s;\mathfrak{r}-t_m, \vartheta_{-\mathfrak{r}}\omega,\v_{0m})-\tilde{\u}\|^2_{2}.
	\end{align}
	Let us now choose $\eta>0$ and fix it. Since $\tilde{\u}\in\H$, there exists $K_1=K_1(\mathfrak{r},\omega,\eta)>0$ such that for all $\mathfrak{k}\geq K_1$, 
	\begin{align}\label{Uac4}
		\int_{\O\cap\{\mathfrak{k}\leq|x|\}}|\tilde{\u}|^2\d x<\frac{\eta^2}{6C_{Lip}},
	\end{align}
	where $C_{Lip}>0$ is a constant defined in \eqref{Uac7}. Also, we know from Lemma \ref{LR} that there exists $M_2=M_2(s,\omega,B,\eta)\in\N$ and $K_2=K_2(s,\omega,\eta)\geq K_1$ such that for all $m\geq M_2$ and $\mathfrak{k}\geq K_2$,
	\begin{align}\label{Uac5}
		\int_{\O\cap\{\mathfrak{k}\leq|x|\}}|\v(s;\mathfrak{r}-t_m,\vartheta_{-\mathfrak{r}}\omega,\v_{0m})|^2\d x<\frac{\eta^2}{6C_{Lip}}.
	\end{align}
	From \eqref{Uac3}, we have that there exists $M_3=M_3(\mathfrak{r},\omega,{B},\eta)>M_2$ such that for all $m\geq M_3$,
	\begin{align}\label{Uac6}
		\int_{\O\cap\{|x|< K_2\}}|\v(s;\mathfrak{r}-t_m,\vartheta_{-\mathfrak{r}}\omega,\v_{0m})-\tilde{\u}|^2\d x<\frac{\eta^2}{3C_{Lip}}.
	\end{align}
	Finally, we infer from \eqref{Uac7} that
	\begin{align}\label{Uac8}
		&\|\v(\mathfrak{r};s,\v(s;\mathfrak{r}-t_m,\vartheta_{-\mathfrak{r}}\omega,\v_{0m}))-\v(\mathfrak{r};s,\vartheta_{-\mathfrak{r}}\omega,\tilde{\u})\|^2_{2}\nonumber\\&\leq  C_{Lip}\left[\int\limits_{\O\cap\{|x|<K_2\}}|\v(s;\mathfrak{r}-t_m, \vartheta_{-\mathfrak{r}}\omega,\v_{0m})-\tilde{\u}|^2\d x+\int\limits_{\O\cap\{|x|\geq K_2\}} |\v(s;\mathfrak{r}-t_m, \vartheta_{-\mathfrak{r}}\omega,\v_{0m})-\tilde{\u}|^2\d x\right]\nonumber\\&\leq C_{Lip}\left[\int\limits_{\O\cap\{|x|<K_2\}}|\v(s;\mathfrak{r}-t_m,\vartheta_{-\mathfrak{r}}\omega,\v_{0m})-\tilde{\u}|^2\d x+2\int\limits_{\O\cap\{|x|\geq K_2\}}(|\v(s;\mathfrak{r}-t_m,\vartheta_{-\mathfrak{r}}\omega,\v_{0m})|^2+|\tilde{\u}|^2)\d x\right]
		\nonumber\\ & < \eta^2,
	\end{align}
	for every $m\geq M_3$, where we have used \eqref{Uac4}-\eqref{Uac6}. Since $\eta>0$ is arbitrary, we conclude the proof.
\end{proof}

\begin{theorem}\label{PRA_U}
	 Let the condition \eqref{third-grade-paremeters-res} and Hypothesis \ref{Hyp-f-U} be satisfied. Then, there exists a unique $\mathfrak{D}$-pullback random attractor $\mathscr{A}=\{\mathscr{A}(\mathfrak{r},\omega):\mathfrak{r}\in\R, \omega\in\Omega\}\in\mathfrak{D},$ for the  non-autonomous random dynamical system $\Phi$ associated with the system \eqref{STGF} in $\H$.
\end{theorem}
\begin{proof}
	The proof follows from Lemma \ref{Asymptotic_UB_GS} ($\mathfrak{D}$-pullback asymptotic compactness of $\Phi$), Lemma \ref{PAS} (existence of $\mathfrak{D}$-pullback random absorbing set) and the abstract theory given in \cite[Theorem 2.23]{SandN_Wang}.
\end{proof}

\section{Invariant Measures and Ergodicity}\label{Sec6}\setcounter{equation}{0}

In this section, we establish the existence and uniqueness of invariant measures, as well as the ergodicity, for the random dynamical system $\Phi$ associated with \eqref{STGF}. Since our analysis relies on the abstract framework developed in \cite{CF}, we restrict attention to the case where the deterministic forcing term is independent of time. Henceforth, we assume that $\f$ is autonomous and that $\f \in \H^{-1}(\O)$ when $\O \subset \R^{d}$ is a bounded domain, while $\f \in \L^{2}(\O)$ when $\O \subseteq \R^{d}$ is unbounded.

The results in \cite{CF} show that if a random dynamical system $\Phi$ possesses a compact invariant random set, then $\Phi$ admits an invariant measure (see \cite[Corollaries 4.4 and 4.6]{CF}). Therefore, by \cite[Corollaries 4.4 and 4.6]{CF} together with Theorems~\ref{PRA_B} (for bounded domains) and~\ref{PRA_U} (for unbounded domains), we immediately obtain the existence of invariant measures for the autonomous system \eqref{STGF}, since its random attractor serves as a compact invariant random set.

\subsection{Existence}
Since the deterministic forcing term in the system \eqref{STGF} is time-independent, the non-autonomous random dynamical system will reduce to autonomous random dynamical system. Therefore, the random dynamical system $\Phi:\R^+\times\Omega\times\H\to\H$, for $t\geq\mathfrak{r}$, $\omega\in\Omega$ and $\u_\mathfrak{r}\in\H$, is defined by
\begin{align}\label{Phi2}
	\Phi(t-\mathfrak{r},\vartheta_{\mathfrak{r}}\omega,\u_{\mathfrak{r}}) :=\u(t;\mathfrak{r},\omega,\u_{\mathfrak{r}})=\frac{\v(t;\mathfrak{r},\omega,\v_{\mathfrak{r}})}{\z(t,\omega)} \ \text{ with }\  \v_{\mathfrak{r}}=\z(\mathfrak{r},\omega)\u_{\mathfrak{r}}.
\end{align}
For a Banach space $\X$, let $\mathscr{B}_b(\X)$ be the space of all bounded and Borel measurable functions on $\X$, and $\C_b(\X)$ be the space of all bounded and continuous functions on $\X$. Let us define the transition operator $\{\mathrm{T}_t\}_{t\geq 0}$ by 
\begin{align}\label{71}
	\mathrm{T}_t g(\x)=\int_{\Omega}g(\Phi(t,\omega,\x))\d\mathbb{P}(\omega)=\E\left[g(\Phi(t,\cdot,\x))\right],
\end{align}
for all $g\in\mathscr{B}_b(\H)$, where $\Phi$ is the random dynamical system corresponding to system \eqref{STGF} defined by \eqref{Phi2}. Since $\Phi$ is continuous (Lemma \ref{Continuity}), the following result holds due to \cite[Proposition 3.8]{BL}. 
\begin{lemma}\label{Feller}
	The family $\{\mathrm{T}_t\}_{t\geq 0}$ is Feller, that is, $\mathrm{T}_tg\in\C_{b}(\H)$ if $g\in\C_b(\H)$. Moreover, for any $g\in\C_b(\H)$, $\mathrm{T}_tg(\x)\to g(\x)$ as $t\downarrow 0$. 
\end{lemma}
\begin{definition}
	A Borel probability measure $\upmu$ on $\H$  is called an \emph{invariant measure} for a Markov semigroup $\{\mathrm{T}_t\}_{t\geq 0}$ of Feller operators on $\C_b(\H)$ if and only if $$\mathrm{T}_{t}^*\upmu=\upmu, \ t\geq 0,$$ where $(\mathrm{T}_{t}^*\upmu)(\Gamma)=\int_{\H}\mathrm{P}_{t}(\y,\Gamma)\upmu(\d\y),$ for $\Gamma\in\mathscr{B}(\H)$ and  $\mathrm{T}_t(\y,\cdot)$ is the transition probability, $\mathrm{T}_{t}(\y,\Gamma)=\mathrm{T}_{t}(\chi_{\Gamma})(\y),\ \y\in\H$.
\end{definition}

It can be shown that $\Phi$ is a Markov random dynamical system (see \cite[Theorem 5.6]{CF}), meaning that $\mathrm{T}_{s_1+s_2} = \mathrm{T}_{s_1}\mathrm{T}_{s_2}$ holds for all $s_1, s_2 \ge 0$. Moreover, \cite[Corollary 4.6]{CF} asserts that any Markov random dynamical system defined on a Polish space and possessing an invariant compact random set admits a Feller invariant probability measure. Therefore, by combining this abstract result with Theorems~\ref{PRA_B} (in the case of bounded domains) and~\ref{PRA_U} (for unbounded domains), we obtain the following conclusion.

\begin{theorem}\label{thm6.3}
	Assume that $\f \in \H^{-1}(\O)$ when $\O \subset \R^{d}$ is a bounded domain, while $\f \in \L^{2}(\O)$ when $\O \subseteq \R^{d}$ is unbounded. Then, the Markov semigroup $\{\mathrm{T}_t\}_{t\geq 0}$ induced by $\Phi$ on $\H$ has an invariant measure $\upmu$.
\end{theorem}

\subsection{Uniqueness} 
In this section, we show that the invariant measure for the system \eqref{STGF} is unique, relying on the linear structure of the multiplicative noise and the exponential stability exhibited by the system’s solutions.
 To this end, we set the external forcing to zero, that is, we take  $\f=\boldsymbol{0}$ in \eqref{STGF}.

\begin{lemma}\label{ExpoStability}
	Assume that \emph{$\f=\textbf{0}$}. Then, there exists $T(\omega)>0$ such that the solution of the system \eqref{STGF} satisfies the following exponential estimate:
	\begin{align}\label{U1}
		\|\u_1(t)-\u_2(t)\|^2_{2}\leq e^{-\frac{\sigma^2}{4}t}\|\u_{1,0}-\u_{2,0}\|^2_{2},  \ \  \text{ for all }\  t\geq T(\omega),
	\end{align}
	where $\u_1(t):=\u_{1}(t;0,\omega,\u_{1,0})$ and $\u_{2}(t):=\u_{2}(t;0,\omega,\u_{2,0})$ be two solutions of the system \eqref{STGF} with respect to the initial data $\u_{1,0}$ and $\u_{2,0}$ at $0$, respectively. 
\end{lemma}
\begin{proof}
	Let  $\v_1(t):=\z(t,\omega)\u_{1}(t)$ and $\v_{2}(t):=\z(t,\omega)\u_{2}(t)$ be two solutions of the system \eqref{CTGF} with respect to the initial data $\v_{1,0}=\z(0,\omega)\u_{1,0}$ and $\v_{2,0}=\z(0,\omega)\u_{2,0}$ at $0$, respectively. Then $\mathfrak{X}(\cdot)=\v_{1}(\cdot)-\v_{2}(\cdot)$ with $\mathfrak{X}(0)=\v_{1,0}-\v_{2,0}$ satisfies
\begin{align}\label{ES1}
	&	\frac{\d\mathfrak{X}(t)}{\d t}+\nu \A\mathfrak{X}(t)  +\left[\frac{\sigma^2}{2}-\sigma y(\vartheta_{t}\omega)\right]\mathfrak{X}(t) %\nonumber\\&=-\z^{-1}(t,\omega)\left\{\B\big(\v_1(t)\big)-\B\big(\v_2(t)\big)\right\} - \alpha \z^{-1}(t,\omega)\left\{\J\big(\v_1(t)\big)-\J\big(\v_2(t)\big)\right\} -\z^{-2}(t,\omega)\beta\left\{\K\big(\v_1(t)\big)-\K\big(\v_2(t)\big)\right\}
	\nonumber\\ & = -\z^{-1}(t,\omega)\left\{\B\big(\mathfrak{X}(t),\v_1(t)\big)-\B\big(\v_2(t), \mathfrak{X}(t)\big)\right\} - \alpha \z^{-1}(t,\omega)\left\{\J\big(\v_1(t)\big)-\J\big(\v_2(t)\big)\right\}
	\nonumber\\ & \quad  -\z^{-2}(t,\omega)\beta\left\{\K\big(\v_1(t)\big)-\K\big(\v_2(t)\big)\right\} ,
\end{align}
	for a.e. $t\geq0$ in $\X'$.  A calculation similar to \eqref{Conti6} yields
	\begin{align}\label{ES2}
		&	\frac{\d}{\d t}\|\mathfrak{X}(t)\|_2^2 +\left[\sigma^2 - 2\sigma y(\vartheta_{t}\omega) - \frac{(\M_d)^2}{\nu\varepsilon_0} \z^{-2}(t,\omega)   \|\Arm(\v_1(t))\|^2_{4}\right]\|\mathfrak{X}(t)\|_2^2  
		  \leq  0 ,
	\end{align}
	for a.e. $t\geq0$ and for all $\omega\in\Omega$,  and an application of variation of constant formula implies 
	\begin{align*}
		\|\mathfrak{X}(t)\|^2_{2}&\leq \exp\left\{-\sigma^2t + 2\sigma \int_{0}^{t}y(\vartheta_{r}\omega)\d r+ \frac{(\M_d)^2}{\nu\varepsilon_0} \int_{0}^{t}\z^{-2}(r,\omega)   \|\Arm(\v_1(r))\|^2_{4}\d r \right\}\|\mathfrak{X}(0)\|^2_{2},
	\end{align*} 
	for all $t\geq0$ and $\omega\in\Omega$. Using \eqref{Phi2} and \eqref{Z3}, we can find a time $T_1(\omega)>0$ such that 
	\begin{align}\label{ES3}
		&\|\u_1(t)-\u_2(t)\|^2_{2}
		\nonumber\\&\leq \exp\left\{-\sigma^2t+2\sigma y(\vartheta_{t}\omega) +2\sigma \int_{0}^{t}y(\vartheta_{r}\omega)\d r+\frac{(\M_d)^2}{\nu\varepsilon_0} \int_{0}^{t}\z^{-2}(r,\omega)   \|\Arm(\v_1(r))\|^2_{4}\d r -2\sigma y(\omega) \right\}
		  \|\u_{1,0}-\u_{2,0}\|^2_{2}
		\nonumber\\&\leq \exp\left\{-\frac{\sigma^2}{2}t+\frac{(\M_d)^2}{\nu\varepsilon_0} \int_{0}^{t}\z^{-2}(r,\omega)   \|\Arm(\v_1(r))\|^2_{4}\d r-2\sigma y(\omega) \right\} \|\u_{1,0}-\u_{2,0}\|^2_{2}, 
	\end{align} 
	for all $t\geq T_1(\omega)$ and $\omega\in\Omega$. From $\eqref{CTGF}_1$, \eqref{b0} and Young's inequality, we find (see \eqref{UE4} above)
	\begin{align*}
		& \frac{\d}{\d t}\|\v_1(t)\|_2^2  + \left[\frac{\sigma^2}{2} - 2\sigma y (\vartheta_{t}\omega)\right]\|\v_1(t)\|^2_2 + \beta\varepsilon_0  \z^{-2}(t,\omega)\|\Arm(\v_1(t))\|^4_4  \leq 0, % \frac{2\z^2(t,\omega)}{\min\left\{2\nu\varepsilon_0,\sigma^2\right\}} \|\f\|^2_{\H^{-1}}.
	\end{align*}
	which gives (by variation of constant formula)
	\begin{align}\label{ES4}
		& \|\v_1(t)\|^2_{2}e^{\frac{\sigma^2}{2}t-2\sigma\int_{0}^{t}y(\vartheta_{r}\omega)\d r}+\beta\varepsilon_0 \int_{0}^{t}e^{\frac{\sigma^2}{2}s-2\sigma\int_{0}^{s}y(\vartheta_{r}\omega)\d r} \z^{-2}(s,\omega)\|\Arm(\v_1(s))\|^4_4\d s 
		 \leq \|\v_{1,0}\|^2_{2}, 
		%+ \frac{2 \|\f\|^2_{\H^{-1}}}{\min\left\{2\nu\varepsilon_0,\sigma^2\right\}} \int_{0}^{t}e^{\frac{\sigma^2}{2}s-2\sigma\int_{0}^{s}y(\vartheta_{r}\omega)\d r} \z^{2}(s,\omega) \d s,
	\end{align}
	for all $t\geq 0$ and $\omega\in\Omega$.  By \eqref{Z3}, we find $T_2(\omega)>0$ such that for all $t>T_2(\omega)$
	\begin{align}\label{ES5}
		&\int_{0}^{t}e^{-\frac{\sigma^2}{2}s+2\sigma\int_{0}^{s}y(\vartheta_{r}\omega)\d r} \z^{-2}(s,\omega)\d s   = \int_{0}^{t}e^{-\frac{\sigma^2}{2}s+2\sigma\int_{0}^{s}y(\vartheta_{r}\omega)\d r + 2\sigma y(\vartheta_s\omega)} \d s
		\nonumber\\ & = \int_{0}^{T_2(\omega)}e^{-\frac{\sigma^2}{2}s+2\sigma\int_{0}^{s}y(\vartheta_{r}\omega)\d r + 2\sigma y(\vartheta_s\omega)} \d s + \int_{T_2(\omega)}^{t}e^{-\frac{\sigma^2}{2}s+2\sigma\int_{0}^{s}y(\vartheta_{r}\omega)\d r + 2\sigma y(\vartheta_s\omega)} \d s
		\nonumber\\ & \leq  \int_{0}^{T_2(\omega)}e^{-\frac{\sigma^2}{2}s+2\sigma\int_{0}^{s}y(\vartheta_{r}\omega)\d r + 2\sigma y(\vartheta_s\omega)} \d s + \int_{T_2(\omega)}^{t}e^{-\frac{\sigma^2}{4}s} \d s
		\nonumber\\ & \leq  \int_{0}^{T_2(\omega)}e^{-\frac{\sigma^2}{2}s+2\sigma\int_{0}^{s}y(\vartheta_{r}\omega)\d r + 2\sigma y(\vartheta_s\omega)} \d s + \frac{4}{\sigma^2}e^{-\frac{\sigma^2}{4}T_2(\omega)} := [\rho(\omega)]^2 < +\infty.
	\end{align}
Now by H\"older's inequality, \eqref{ES4} and \eqref{ES5}, we obtain for all $t\geq T_2(\omega)$
\begin{align}\label{ES6}
	& \int_{0}^{t}\z^{-2}(s,\omega)   \|\Arm(\v_1(s))\|^2_{4}\d s
	\nonumber\\ & = \int_{0}^{t} e^{-\frac{\sigma^2}{4}s+\sigma\int_{0}^{s}y(\vartheta_{r}\omega)\d r} \z^{-1}(s,\omega) e^{\frac{\sigma^2}{4}s-\sigma\int_{0}^{s}y(\vartheta_{r}\omega)\d r} \z^{-1}(s,\omega)   \|\Arm(\v_1(s))\|^2_{4}\d s
	\nonumber\\ & \leq \left(\int_{0}^{t}e^{-\frac{\sigma^2}{2}s+2\sigma\int_{0}^{s}y(\vartheta_{r}\omega)\d r} \z^{-2}(s,\omega)\d s \right)^{\frac12} \left( \int_{0}^{t}e^{\frac{\sigma^2}{2}s-2\sigma\int_{0}^{s}y(\vartheta_{r}\omega)\d r} \z^{-2}(s,\omega)\|\Arm(\v_1(s))\|^4_4\d s \right)^{\frac12}
	\nonumber\\ & \leq \frac{\rho(\omega) \|\v_{1,0}\|_2}{\sqrt{\beta\varepsilon_0}}.
\end{align}
 
 Now we choose $T_3(\omega)\geq\max\{T_1(\omega),T_2(\omega)\}$, then \eqref{ES3} and \eqref{ES6} give
	\begin{align}\label{ES9}
		&\|\u_1(t)-\u_2(t)\|^2_{2}\leq \exp\left\{-\frac{\sigma^2}{2}t+ \frac{(\M_d)^2}{\nu\varepsilon_0} \frac{\rho(\omega) \|\v_{1,0}\|_2}{\sqrt{\beta\varepsilon_0}} -2\sigma y(\omega) \right\} \|\u_{1,0}-\u_{2,0}\|^2_{\H},
	\end{align}
	for all $t\geq T_3(\omega)$ and $\omega\in\Omega$. Now, we can find a time $T(\omega)\geq T_3(\omega)$ such that 
	\begin{align}\label{ES10}
		\frac{(\M_d)^2}{\nu\varepsilon_0} \frac{\rho(\omega) \|\v_{1,0}\|_2}{\sqrt{\beta\varepsilon_0}} -2\sigma y(\omega) \leq\frac{\sigma^2}{4}t,
	\end{align}
	for all $t\geq T(\omega)$ and $\omega\in\Omega$. By \eqref{ES9}-\eqref{ES10}, we have
	\begin{align*}
		\|\u_1(t)-\u_2(t)\|^2_{\H}\leq e^{-\frac{\sigma^2}{4}t}\|\u_{1,0}-\u_{2,0}\|^2_{\H},  \ \ \ \text{ for all } t\geq T(\omega)\ \text{ and }\ \omega\in\Omega,
	\end{align*}
	which completes the proof.
\end{proof}

Now, we are ready to show the uniqueness of invariant measures for $\f=\boldsymbol{0}$.

\begin{theorem}\label{UEIM}
	Assume that \emph{$\f=\textbf{0}$} and $\u_0\in\H$ be given. Then, there is a unique invariant measure to the system \eqref{STGF} which is ergodic and strongly mixing.
\end{theorem}
\begin{proof}
	For $\Xi\in \text{Lip}(\H)$ (Lipschitz $\Xi$) and an invariant measure $\upmu$, we have for all $t\geq T(\omega)$,
	\begin{align*}
		\left|\mathrm{T}_t\Xi(\u_0)-\int_{\H}\Xi(\v_0)\upmu(\d \v_0)\right|&=	\left|\int_{\H}\E\left[\Xi(\u(t,\u_0))-\Xi(\u(t,\v_0))\right]\upmu(\d \v_0)\right|\nonumber\\&\leq L_{\Xi}\int_{\H}\E\left[\left\|\u(t,\u_0)-\u(t,\v_0)\right\|_{2}\right]\upmu(\d \v_0)\nonumber\\&\leq L_{\Xi}\mathbb{E}\left[e^{-\frac{\sigma^2}{8}t}\right]\int_{\H}\left\|\u_0-\v_0\right\|_{2}\upmu(\d \v_0)
		\to 0 \ \text{ as }\  t\to \infty,
	\end{align*}
	since $\int_{\H}\|\v_0\|_{2}\upmu(\d \v_0)+\int_{\H}\|\u_0\|_{2}\upmu(\d \v_0)<+\infty$. Hence, we conclude
	\begin{align}\label{U2}
		\lim_{t\to\infty}\mathrm{T}_t\Xi(\u_0)=\int_{\H}\Xi(\v_0)\d\upmu(\v_0), \ \upmu\text{-a.s., for all }\ \u_0\in\H\ \text{ and }\  \Xi\in\C_b(\H),
	\end{align} 
	by the density of $\text{Lip}(\H)$ in $\C_b (\H)$.  Since we have already established the stronger result that $\mathrm{T}_t\Xi(\u_0)$ converges to the equilibrium at an exponential rate, the system enjoys what is referred to as the \emph{exponential mixing property}. Now, let $\widetilde{\upmu}$ be another invariant measure. Then, for every $t \geq T(\omega)$, we obtain
	\begin{align}
		\left|\int_{\H}\Xi(\u_0)\upmu(\d\u_0)-\int_{\H}\Xi(\v_0)\widetilde{\upmu}(\d \v_0)\right| 
		&= \left|\int_{\H}\mathrm{T}_t\Xi(\u_0)\upmu(\d \u_0)-\int_{\H}\mathrm{T}_t\Xi(\v_0)\widetilde{\upmu}(\d \v_0)\right|\nonumber\\&=\left|\int_{\H}\int_{\H}\left[\mathrm{T}_t\Xi(\u_0)-\mathrm{T}_t\Xi(\v_0)\right]\upmu(\d \u_0)\wi\upmu(\d \v_0)\right|\nonumber\\&=\left|\int_{\H}\int_{\H}\E\left[\Xi(\u(t,\u_0))-\Xi(\u(t,\v_0))\right]\upmu(\d \u_0)\wi\upmu(\d \v_0)\right|\nonumber\\&\leq L_{\Xi}\int_{\H}\int_{\H}\E\left[\left\|\u(t,\u_0)-\u(t,\v_0)\right\|_{2}\right]\upmu(\d \u_0)\wi\upmu(\d \v_0) 
		\nonumber\\&\leq L_{\Xi}\mathbb{E}\left[e^{-\frac{\sigma^2}{8}t}\right] \int_{\H}\int_{\H}\E\left[\left\|\u_0-\v_0\right\|_{2}\right]\upmu(\d \u_0)\wi\upmu(\d \v_0) %\nonumber\\&\leq L_{\Xi}\mathbb{E}\left[e^{-\frac{\sigma^2}{8}t}\right]\bigg(\int_{\H}\|\u_0\|_{2}\upmu(\d \u_0)+\int_{\H}\|\v_0\|_{2}\widetilde{\upmu}(\d \v_0)\bigg)
		\to 0 \ \text{ as }\ t\to\infty,
	\end{align}
	since $\int_{\H}\|\u_0\|_{2}\upmu(\d \u_0)<+\infty$ and $\int_{\H}\|\v_0\|_{2}\widetilde{\upmu}(\d \v_0)<+\infty$.  Since $\upmu$ is the unique invariant measure for $(\mathrm{T}_t)_{t\geq 0}$, it follows from \cite[Theorem 3.2.6]{GDJZ} that $\upmu$ is ergodic also. 
\end{proof}

\begin{remark}
	Since the zero function is a solution of \eqref{STGF} corresponding to zero initial data and vanishing forcing $\f=\boldsymbol{0}$, Lemma~\ref{ExpoStability} implies that every nontrivial solution of \eqref{STGF} converges exponentially fast to $\boldsymbol{0}$ as $t \to \infty$. Consequently, when $\f=\boldsymbol{0}$, the only ergodic invariant measure of the system \eqref{STGF} is the trivial one (the Dirac delta measure concentrated at zero) which is clearly invariant because the zero solution is the equilibrium state of the system.
	%Similarly, the pullback random attractor is just a singleton set $\{\boldsymbol{0}\}$. 
\end{remark}

	\begin{appendix}
		\renewcommand{\thesection}{\Alph{section}}
		\numberwithin{equation}{section}
		
		\section{Retrieval of Pressure }\label{PR} \numberwithin{equation}{section}\setcounter{equation}{0}
		We know that the system \eqref{CTGF} has unique weak solution in the sense of Definition \ref{def-CTGF} (see Lemma \ref{lem-Sol-CTGF}). In order to retrieve the pressure, we follow the approach which has been carried out in the work \cite[Section 8]{SS11}. 
		
		Let us choose $\mathfrak{r}\in\R$ and $\omega\in\Omega$. Let $\C_c^{\infty}(\O;\R^d)$ be the space of all infinitely differentiable functions  ($\R^d$-valued) with compact support in $\O\subseteq\R^d$.
		Let us introduce the space of vector-valued test functions  $\mathrm{C}_c^{\infty}(\mathfrak{r},\mathfrak{r}+T;\C_c^{\infty}(\O;\R^d))$, which consists of all infinitely differentiable $\C_c^{\infty}(\O;\R^d)$-valued functions having compact support in $(\mathfrak{r},\mathfrak{r}+T)$. Moreover, we denote by $\mathscr{D}'((\mathfrak{r},\mathfrak{r}+T) \times \O)$, the dual space of $\mathrm{C}_c^{\infty}(\mathfrak{r},\mathfrak{r}+T;\C_c^{\infty}(\O;\R^d))$. Let us define a functional $\upchi\in\mathscr{D}'((\mathfrak{r},\mathfrak{r}+T) \times \O)$ by
		\begin{align}\label{press8}
			&\int_\mathfrak{r}^{\mathfrak{r}+T} \langle\upeta,\varphi\rangle\d t\nonumber\\&:=\int_\mathfrak{r}^{\mathfrak{r}+T} 
			\biggl\{\bigg\langle\z^{-1}(t,\omega)\frac{\partial\v}{\partial t} - \nu \z^{-1}(t,\omega) \Delta \v +\left[\frac{\sigma^2}{2}-\sigma y(\vartheta_{t}\omega)\right]\z^{-1}(t,\omega) \v + \z^{-2}(t,\omega)(\v\cdot \nabla)\v  
			\nonumber	\\ & \qquad - \alpha \z^{-2}(t,\omega) \text{div}((\Arm(\v))^2) 
		  -  \beta \z^{-3}(t,\omega)  \text{div}(|\Arm(\v)|^2\Arm(\v))  -  \f,\varphi\bigg\rangle \biggr\}\d t,
		\end{align}
		for all $\varphi\in\mathrm{C}_c^{\infty}(\mathfrak{r},\mathfrak{r}+T;\C_c^{\infty}(\O;\R^d))$.  Since $\v \in  \mathrm{C}([\mathfrak{r},\mathfrak{r}+T]; \H) \cap \mathrm{L}^{2}(\mathfrak{r},\mathfrak{r}+T; \V)\cap \mathrm{L}^{4} (\mathfrak{r},\mathfrak{r}+T; \mathbb{W}^{1,4}(\O))$ and $\frac{\partial\v}{\partial t} \in \mathrm{L}^{2} (\mathfrak{r},\mathfrak{r}+T;\V') + \mathrm{L}^{\frac43} (\mathfrak{r},\mathfrak{r}+T;\mathbb{W}^{-1,\frac43}(\O))$, we have (see Lemma \ref{lem-Sol-CTGF})
		\begin{align}
			\left| \int_{\mathfrak{r}}^{\mathfrak{r}+T} \langle\upeta,\varphi\rangle\d t \right| \leq C(\mathfrak{r},\omega, \alpha, \beta, \nu, T, \|\v_{\mathfrak{r}}\|_{2}, \|\f\|_{\mathrm{L}^2_{\mathrm{loc}}(\R;\H^{-1})})) \|\varphi\|_{\mathrm{L}^{2}(\mathfrak{r},\mathfrak{r}+T; \H^1_0)\cap \mathrm{L}^{4} (\mathfrak{r},\mathfrak{r}+T; \mathbb{W}^{1,4})},
		\end{align}
		for all $\varphi\in\mathrm{C}_c^{\infty}(\mathfrak{r},\mathfrak{r}+T;\C_c^{\infty}(\O;\R^d))$. Therefore, it follows that
		\begin{align*}
			\upeta\in \mathrm{L}^{2} (\mathfrak{r},\mathfrak{r}+T;\H^{-1}(\O)) + \mathrm{L}^{\frac43} (\mathfrak{r},\mathfrak{r}+T;\mathbb{W}^{-1,\frac43}(\O)).
		\end{align*}
		Since $\v(\cdot)$ is unique solution of system \eqref{CTGF} in the sense of Definition \ref{def-CTGF},  \eqref{press8} gives 
		\begin{align}\label{press11}
			\int_\mathfrak{r}^{\mathfrak{r}+T} \langle\upeta,\varphi\rangle\d t=0, \  \text{ for all } \ \varphi\in \mathrm{C}_c^{\infty}(\mathfrak{r},\mathfrak{r}+T;\mathscr{V}).
		\end{align}
		Next, we choose a test function of the form $\varphi(x,t)=\upxi(t)\psi(x)$ with $\upxi\in\mathrm{C}_c^{\infty}(\mathfrak{r},\mathfrak{r}+T)$ and $\psi\in\mathscr{V}$. Then, from \eqref{press11}, we have 
		\begin{align*}
			\int_\mathfrak{r}^{\mathfrak{r}+T}\upxi(t)\langle\upeta,\psi\rangle\d t=0, \ \text{ for all } \  \upxi\in\mathrm{C}_c^{\infty}(\mathfrak{r},\mathfrak{r}+T) \  \text{ and for all } \  \psi\in\mathscr{V}.
		\end{align*}
		Then, we deduce that $\upeta$ satisfies the following equation   for a.e. $t\in[\mathfrak{r},\mathfrak{r}+T]$:
		\begin{align}\label{press12}
			& \langle\upeta,\psi\rangle 
			\nonumber\\ &=  
			\bigg\langle\z^{-1}(t,\omega)\frac{\partial\v}{\partial t} - \nu \z^{-1}(t,\omega) \Delta \v +\left[\frac{\sigma^2}{2}-\sigma y(\vartheta_{t}\omega)\right]\z^{-1}(t,\omega) \v + \z^{-2}(t,\omega)(\v\cdot \nabla)\v  
			\nonumber	\\ & \qquad - \alpha \z^{-2}(t,\omega) \text{div}((\Arm(\v))^2) 
			-  \beta \z^{-3}(t,\omega)  \text{div}(|\Arm(\v)|^2\Arm(\v))  -  \f,\psi\bigg\rangle, 
		\end{align}
		for all $ \psi\in  \C_c^{\infty}(\O;\R^d)$,  and 
		\begin{align*}
			\langle\upeta ,\psi\rangle=0,  \  \text{ for all } \ \psi\in\mathscr{V}.
		\end{align*}
		Finally, by an application of \cite[Proposition 1.1]{Temam_1984}, there exists a ${\mathbf{P}} \in \mathscr{D}'((\mathfrak{r},\mathfrak{r}+T)\times \O)$ such that
		$$\upeta = - \nabla  {\mathbf{P}} \in \mathrm{L}^{2} (\mathfrak{r},\mathfrak{r}+T;\H^{-1}(\O)) + \mathrm{L}^{\frac43} (\mathfrak{r},\mathfrak{r}+T;\mathbb{W}^{-1,\frac43}(\O)).$$
		
		\begin{remark}\label{rem-pressure}
			Indeed, we have 
			\begin{align*}
				\nabla  {\mathbf{P}} 
		 & = (\mathrm{I}- \mathcal{P}) 
				\left\{  \nu \z^{-1}(t,\omega) \Delta \v   - \z^{-2}(t,\omega)(\v\cdot \nabla)\v  
				 + \alpha \z^{-2}(t,\omega) \mathrm{div}((\Arm(\v))^2)  \right.
			\nonumber\\ & \qquad \qquad \quad  \left.	+ \beta \z^{-3}(t,\omega)  \mathrm{div}(|\Arm(\v)|^2\Arm(\v))  + \f\right\}
					\nonumber\\ & =:\nabla  {\mathbf{P}}_1 + \nabla  {\mathbf{P}}_2,
			\end{align*}
			where 
			\begin{align*}
				\nabla  {\mathbf{P}}_1 & = (\mathrm{I}- \mathcal{P}) \left\{  \nu \z^{-1}(t,\omega) \Delta \v   - \z^{-2}(t,\omega)(\v\cdot \nabla)\v  
				+ \alpha \z^{-2}(t,\omega) \mathrm{div}((\Arm(\v))^2)  + \f \right\} \\
				\nabla  {\mathbf{P}}_2 & = \beta \z^{-3}(t,\omega) (\mathrm{I}- \mathcal{P})  \{\mathrm{div}(|\Arm(\v)|^2\Arm(\v))\}.
			\end{align*}
			Note that, by Lemmas \ref{LemmaUe} and \ref{LemmaUe-24}, we have for all  $t\geq \widetilde{\mathcal{T}} := \max\{\mathcal{T},\widehat{\mathcal{T}}\}$ and $\omega\in\Omega$ that
		\begin{align}\label{estimate-pressure3}
			\int_{\mathfrak{r}-t}^{\mathfrak{r}}e^{-\int^{\mathfrak{r}}_{\zeta}\left(\frac{\sigma^2}{2}-2\sigma y(\vartheta_{\upeta-\mathfrak{r}}\omega)\right)\d \upeta}\z^2 (\zeta,\vartheta_{-\mathfrak{r}}\omega)\|\nabla\mathbf{P}_1(\zeta;\mathfrak{r}-t,\vartheta_{-\mathfrak{r}}\omega,\v_{\mathfrak{r}-t})\|^2_{\H^{-1}}\d\zeta < +\infty,
		\end{align}
	and 
	\begin{align}\label{estimate-pressure4}
		\int_{\mathfrak{r}-t}^{\mathfrak{r}}e^{-\int^{\mathfrak{r}}_{\zeta}\left(\frac{\sigma^2}{2}-2\sigma y(\vartheta_{\upeta-\mathfrak{r}}\omega)\right)\d \upeta}\z^2(\zeta,\vartheta_{-\mathfrak{r}}\omega)\|\nabla\mathbf{P}_2(\zeta;\mathfrak{r}-t,\vartheta_{-\mathfrak{r}}\omega,\v_{\mathfrak{r}-t})\|^{\frac43}_{\mathbb{W}^{-1,\frac43}}\d\zeta < +\infty.
	\end{align}
The boundedness obtained in \eqref{estimate-pressure1} and \eqref{estimate-pressure2} helps us to establish the uniform tail estimates for the solution of system \eqref{CTGF}.
		\end{remark}

	\end{appendix}

	\medskip\noindent
	{\bf Acknowledgments:}     This work is funded by national funds through the FCT – Fundação para a Ciência e a Tecnologia, I.P., under the scope of the projects UID/297/2025 and UID/PRR/297/2025 (Center for Mathematics and Applications - NOVA Math). K. Kinra would like to thank Prof. Fernanda Cipriano, Center for Mathematics and Applications (NOVA Math) and Department of Mathematics, NOVA School of Science and Technology (NOVA FCT), Portugal, for useful discussions. K. Kinra would like to thank Prof. Manil T. Mohan, Department of Mathematics, Indian Institute of Technology Roorkee, Roorkee, India, for introducing him into the research area of attractor theory.

	\medskip\noindent
	\textbf{Data availability:} No data was used for the research described in the article.
	
	\medskip\noindent
	\textbf{Declarations}: During the preparation of this work, the authors have not used AI tools.
	
	\medskip\noindent
	\textbf{Author Contributions}: The sole author wrote, and edited the entire manuscript.
	
	\medskip\noindent
	\textbf{Conflict of interest:} The author declares no conflict of interest.

\end{document}